\documentclass[10pt,twoside]{article}
\usepackage{epsfig,amsmath,amssymb,amsthm,color,amsfonts,epstopdf,cite} 
\everymath{\displaystyle}
\def\l{\left}
\def\r{\right}

\def\LL{{\mathcal{L}}}
\def\RR{{\mathbb{R}}}
\def\EE{{\cal{E}}}
\def\DD{{\cal{D}}}
\def\OO{{\cal{O}}}
\def\SS{{\cal{S}}}
\def\R#1{$(\ref{#1})$}

\newcommand{\bb}[1]{\begin{equation}\label{#1}}
\newcommand{\ee}{\end{equation}}
\newcommand{\bbb}{\begin{eqnarray}}
\newcommand{\eee}{\end{eqnarray}}
\newcommand{\bbbb}{\begin{eqnarray*}}
\newcommand{\eeee}{\end{eqnarray*}}
\newcommand{\tT}{\intercal}
\newcommand{\nnn}{\nonumber}

\definecolor{green1}{rgb}{0.1,0.5,0.0}
\newcommand{\red}{\color{black}}

\newtheorem{thm}{Theorem}
\newtheorem{lemma}{Lemma}

\theoremstyle{remark}
\newtheorem{rem}{Remark}
\theoremstyle{define}
\newtheorem{define}{Definition}

\newcommand{\clearallnum}{
    \numberwithin{equation}{section} \setcounter{equation}{0}
    \numberwithin{thm}{section} \setcounter{thm}{0}
    \numberwithin{lemma}{section} \setcounter{lemma}{0}
    \numberwithin{cor}{section} \setcounter{cor}{0}
    \numberwithin{rem}{section} \setcounter{rem}{0}
    \numberwithin{define}{section} \setcounter{define}{0}}

\begin{document}
~

\begin{center}
{\Large\bf On a Nonuniform Crank-Nicolson Scheme for Solving the
Stochastic Kawarada Equation via Arbitrary Grids}

\vspace{8mm}

{Joshua L. 
Padgett\footnote{Principal and corresponding author.
Email address: Josh\underline{~}Padgett@baylor.edu} and
Qin Sheng}

\vspace{3mm}

{Department of Mathematics and
Center for Astrophysics, Space Physics and Engineering Research,
Baylor University, Waco, TX 76798-7328, USA}

\vspace{10mm}

\parbox[t]{13.8cm}{\small{\bf Abstract.} This paper studies a nonuniform finite difference method for solving
the degenerate Kawarada quenching-combustion equation with a vibrant stochastic
source. Arbitrary grids are introduced in both space and time via adaptive principals
to accommodate the uncertainty and singularities involved. It
is shown that, under proper constraints on mesh step sizes, the positivity,
monotonicity of the solution, and numerical stability of the scheme developed are well
preserved. Numerical experiments are given to illustrate our conclusions.
}

\vspace{6mm}

\parbox[t]{13.8cm}{\small{\bf Keywords.} Stochastic Kawarada equation, quenching blow-up,
nonuniform grids, numerical stability, monotonicity, positivity}
\end{center}

\vspace{8mm}

\section{Introduction} \clearallnum
Let $\DD=(-a,a),~\EE=\DD\times (t_0,T),~\partial\DD=\bar{\DD}\setminus\DD$ and
$\SS=\partial\DD\times(t_0,T),$ where $a>0,~0\le t_0<T<\infty.$
We are interested in the monotonically increasing positive solution of the
degenerate stochastic Kawarada problem,
\bbb
&& \sigma(x)u_t=u_{xx}+\varphi(\epsilon)f(u),~~~(x,t)\in\EE,\label{b1}\\
&& u(x,t)=0,~~~(x,t)\in\SS,\label{b2}\\
&& u(x,t_0)=u_0(x),~~~x\in \DD,\label{b3}
\eee
where the degeneracy function $\sigma(x)\geq 0$ for $x\in\bar{\DD},$ and 
the equality occurs only on $\partial\DD.$
The nonlinear source function, $f(u),$
is strictly increasing for $0\leq u<1$ with
$$f(0)=f_0>0,~~\lim_{u\rightarrow 1^-}f(u)=+\infty,$$
and $\varphi(\epsilon):~0 < \varphi_{\min}\leq \varphi\leq \varphi_{\max},$ is a
stochastic inference function of the random variable, or white noise, $\epsilon(x).$ 
{\red The existence and uniqueness of the solution of \R{b1}-\R{b3} can be viewed as a
generalization of the results given by Chan and Levine \cite{Chan2,Levine}.}
It is also observed that
solutions of the stochastic modeling problem \R{b1}-\R{b3} are in general only fractional
order H\"{o}lder continuous \cite{Evans}. Further, the solution
$u$ of \R{b1}-\R{b3} is said to {\em quench\/} if there exists a finite time $T_a>0$ such that
\bb{a1}
\sup\l\{u_t(x,t)\,:\,x\in {\DD}\r\}\rightarrow\infty~\mbox{as}~t\rightarrow
T_a^{-}.
\ee
Such a value $T_a$ is called the {\em quenching time\/} \cite{Chan2,Acker2,Acker1}.
It has been shown that a necessary condition for quenching to occur is
\bb{a2}
\max\l\{u(x,t)\,:\,x\in \bar{\DD}\r\}\rightarrow 1^{-}~\mbox{as}~t\rightarrow T_a^{-}.
\ee
It is known that $T_a$ exists only when $a$ is greater than a
certain {\em critical value\/} $a^*\ll\infty.$
The interval $\DD$ associated with such an $a^*$ is defined as
the {\em critical domain\/} and denoted as $\DD^*.$
Therefore, \R{a2} occurs only when
$\DD^*\subseteq \DD,$ otherwise the monotone positive solution of \R{b1}-\R{b3} exists globally
for $T\rightarrow\infty$ \cite{Levine,Acker2,Acker1}.
In the particular circumstance when
$\sigma(x),\varphi(\epsilon)\equiv 1$ and $f(u) = 1/(1-u),$ it has been shown that
$a^*=\kappa\sqrt{2}$ \cite{Acker1,Kawa}, where
$$\kappa=\max_{0<\xi<\infty}\int_0^{\xi}e^{t^2-\xi^2}dt.$$

Kawarada partial differential equations have been intensively used for modeling numerous 
important phenomena in nature. They characterize not only ignitions of liquid fuels in combustion chambers, 
but also turbulent macro or micro flows between channel walls  \cite{Kawa,Bebernes_89,Sheng4}. The latter is 
particularly meaningful for predicting and preventing oil pipeline decays \cite{Chan3}.
Though computational results for \R{b1}-\R{b3} can be found in numerous recent publications, most of 
numerical analysis presented relies heavily on the 1-norm or $\infty$-norm \cite{Beau1,Josh2,Sheng3}.
The numerical analysis in the current paper implements the more preferred spectral norm.
In addition, the influence of white 
noise associated to the source is considered. 
{\red It should be noted, however, that impacts of nonsmooth sources are in general
different from those due to nonsmooth initial data. In fact, while smooth solutions are still possible if 
nonsmooth initial data are dealt with appropriately, nonsmooth solutions are almost certain when
a nonsmooth reaction term is utilized \cite{Khaliq,Khaliq2}.}

This paper proposes a temporally adaptive Crank-Nicolson scheme. Predetermined nonuniform grids are
utilized in space and are chosen in order to ideally incorporate the effects of the quenching singularity and stochastic influences in space. 
{\red The use of predetermined nonuniform spatial grids is practically preferable,
especially in cases when quenching locations are predictable \cite{Sheng4}. 
A further merit of such a semi-adaptive infrastructure is that it can be conveniently 
extended for solving multidimensional Kawarada problems. It also makes the subsequent
numerical analysis much simpler and straightforward. Initial approaches of such an idea with uniform 
spatial grids can be found in \cite{Sheng3}.  }

It is crucial that the numerical 
solution acquired preserves fundamental features of the physical solution, 
such as the positivity, monotonicity, quenching time, and location. To this end, our discussions
will be organized as follows. In the next section, the adaptive Crank-Nicolson 
scheme for solving \R{b1}-\R{b3} is implemented and evaluated. Its positivity 
is investigated. In Section 3, constraints under which the numerical solution is monotone
are determined.
Our stability analysis is conducted in Section 4. We first accomplish a standard 
stability analysis for a fully linearized scheme. Then an extended stability 
analysis is fulfilled for a fully nonlinear method. In Section 5, several numerical examples are provided. 
These examples provide interesting insights into the effects of the degeneracy 
and stochastic functions on not only quenching times but also quenching locations. Finally, 
our investigations are concluded through remarks and proposed future 
problems in Section 6.

\section{Semi-adaptive Crank-Nicolson scheme and its positivity} \clearallnum
Utilizing the transformation $\tilde{x}=x/a$ and reusing the original variable and other
notations for simplicity, we may reformulate \R{b1}-\R{b3} as
\bbb
&&u_t=\psi(x)u_{xx} +g(u,\epsilon),~~~(x,t)\in\EE,~~~~~~\label{c1}\\
&&u(-1,t)=u(1,t)=0,~~~t>t_0,\label{c2}\\
&&u(x,t_0)=u_0(x),~~~x\in \DD,\label{c2b}
\eee
where $\DD=(-1,1),~\EE=\DD\times (t_0,T),~\psi(x) = 1/({a^2\sigma(x)}),$ and 
$g(u,\epsilon) = {\varphi(\epsilon)f(u)}/{\sigma(x)}.$
For $N\gg 1,$ we inscribe over $\bar{\DD}$ the variable grid:
$\DD_h=\l\{x_i \,:\, i=0,\dots,N+1;~x_0=-1,~x_i<x_{i+1},~x_{N+1}=1\r\}.$
Denote $h_{i} = x_{i+1}-x_i$ for $0\leq i\leq N.$
Let $u_i=u_{i}(t)$ be an approximation of $u(x_i,t)$ and adopt the nonuniform finite difference \cite{Sheng3},
\bbbb
\l.\frac{\partial^2u}{\partial x^2}\r|_{(x_i,t)} &\approx& \frac{2u_{i-1}}{h_{i-1}(h_{i-1}+h_{i})}-\frac{2u_{i}}{h_{i-1}h_{i}}
+\frac{2u_{i+1}}{h_{i}(h_{i-1}+h_{i})},\quad x_i\in\DD_h^\circ,
\eeee
where $\DD_h^\circ = \DD_h\setminus\{x_0,x_{N+1}\}.$
Further, denote
$v(t)=(u_{1},u_{2},\dots,u_{N})^{\tT}\in\RR^{N}$
and let $g(v)$ be a discretization of the source term in \R{c1}.
We obtain readily from \R{c1}-\R{c2b} the following semi-discretized problem
\bbb
v'(t)&=& M v(t)+g(v(t)),~~~t_0<t<T,\label{c8}\\
v(t_0)&=&v_0,\label{c8b}
\eee
where $M=\frac{1}{a^2}B{\red P}\in\RR^{N\times N},$
\bb{matrixT}
B = \mbox{diag}\l(1/\sigma_1,\dots,1/\sigma_N\r),~
{\red P} = \mbox{tridiag}\l(l_{i},~ m_{i},~n_{i}\r)
\ee
and for the above
\bbbb
l_{i} &=& \frac{2}{h_{i}(h_{i}+h_{i+1})},~n_{i} ~=~ \frac{2}{h_{i}(h_{i-1}
+h_{i})},~~~i = 1,\dots,N-1,\\
m_{i} &=& -\frac{2}{h_{i-1}h_{i}},~~~i=1,\dots,N.
\eeee
The formal solution of \R{c8}, \R{c8b} can thus be written as
\bb{rs1}
v(t)~~=~~E(tM)v_0+\int_{t_0}^tE((t-\tau)M)g(v(\tau))d\tau,~~~t_0<t<T,
\ee
where $E(\cdot)=\exp(\cdot)$ is the matrix exponential \cite{Sheng3}.
We proceed by approximating \R{rs1} via a trapezoidal rule and a
[1/1] Pad\'{e} approximation; that is, $E(tM) = p(t) + \OO\l(t^3\r),$
where
$$p(t)~~=~~\l(I-\frac{t}{2}M\r)^{-1}\l(I+\frac{t}{2}M\r),~~~t_0<t<T.$$
These lead to  
$$v(t)=p(t)\l[v_0 + \frac{t}{2}g(v_0)\r] + \frac{t}{2}g(v(t)) + \OO\l((t-t_0)^3\r),~~~|t-t_0|\rightarrow 0^+.$$  
Based on the above, we obtain the following second-order in time
semi-adaptive Crank-Nicolson scheme on variable spatial grids:
\bb{c3}
v_{\ell+1}~~=~~\l(I-\frac{\tau_\ell}{2}M\r)^{-1}
\l(I+\frac{\tau_\ell}{2}M\r)\l(v_\ell+\frac{\tau_\ell}{2}g(v_\ell)\r)+
\frac{\tau_\ell}{2}g(v_{\ell+1}),
\ee
where $v_{\ell}$ and $v_{\ell+1}$ are approximations of $v(t_{\ell})$ and $v(t_{\ell+1}),$
respectively, $v_0$ is the initial vector,
$t_{\ell}=t_0+{ \textstyle\sum_{k=0}^{\ell-1}\tau_k},~0<\tau_{\ell}\ll 1,~\ell=0,1,2,\ldots,$
and $\{\tau_{\ell}\}_{\ell\ge 0}$ is a set of adaptive temporal steps.
In order to avoid a fully implicit scheme, $g(v_{\ell+1})$ may be
approximated by $g(w_{\ell}),$ where $w_{\ell}$ is an approximation to $v_{\ell+1},$
such as
\bb{approx}
w_{\ell} ~~=~~ v_{\ell} + \tau_{\ell}(Mv_{\ell}+g(v_{\ell})),~~~0<\tau_{\ell}\ll 1,
\ee
in practical computations.

Recall \R{a1}. Due to the strong singularity of $u_t$ as $t$ approaches $T_a,~a\geq a^*,$
selecting the proper temporal steps $\tau_\ell$ is vital in computations. To this end, we may consider employing 
arc-length monitoring functions \cite{Sheng4,Sheng3,Lang2,Fur}, or allow the temporal steps to be proportional to the source term or its gradient{\red \cite{Beau1,Josh2,Sheng3,Josh1}}.

Positivity is one of the most important characteristics of the solution of Kawarada problems including
\R{b1}-\R{b3} and \R{c1}-\R{c2b} \cite{Chan2,Levine,Acker2,Acker1}. 
In order for our numerical solutions to be valid, it is crucial that they preserve this property. To this end, we 
let $\wedge$ be one of the operations $<,~\leq,~>,~\geq$ and $\alpha,~\beta\in
\RR^{K_1\times K_2}.$ We assume the following notations in subsequent discussions:

\begin{enumerate}
\item $\alpha\wedge\beta$ means $\alpha_{i,j}\wedge\beta_{i,j},~i=1,2,\ldots,K_1;\,j=1,2,\ldots,K_2;$

\item $a\wedge\alpha$ means $a\wedge\alpha_{i,j},~~i=1,2,\ldots,K_1;\,j=1,2,\ldots,K_2,$
for any given scalar $a.$
\end{enumerate}

\begin{lemma}
$\|T\|_2 \le\max_{i=0,1,\dots,N}\l\{{4}/{h_{i}^2}\r\}.$
\end{lemma}

\begin{proof}
The proof is similar to the one from our earlier investigations \cite{Josh1}.
\end{proof}

For the following we denote $\beta_{\min} = h^2/2\|B\|_2$ and $h = \textstyle\min_{i=0,1,\dots,N}\{h_i\}.$

\begin{lemma}
If
\bb{cfl}
\tau_\ell<a^2\beta_{\min},
\ee
then
$I-\frac{\tau_\ell}{2}M$ and $I+\frac{\tau_\ell}{2}M$
are nonsingular. Further, $I-\frac{\tau_\ell}{2}M$ is monotone and
inverse-positive, and $I+\frac{\tau_\ell}{2}M$ is nonnegative.
\end{lemma}

\begin{proof}
First, we note that
$$\l\|\frac{\tau_\ell}{2}M\r\|_2 \leq \frac{\tau_\ell}{2a^2}\|B\|_2\|T\|_2 \leq \frac{2\tau_\ell}{a^2}\|B\|_2\l(1/\min_{i=0,1,\dots,N}\l\{h_{i}^2\r\}\r) =
\frac{2\tau_\ell}{a^2h^2}\|B\|_2 <1.$$
Hence, $I+\frac{\tau_\ell}{2}M$ is nonsingular, and also nonnegative.

Next, we consider $A = I-\frac{\tau_\ell}{2}M.$ As $A_{ij}\le 0$ for $i\neq j$ and the weak row sum criterion is satisfied;
hence $A$ is monotone, and it follows that its inverse exists and is nonnegative. So, $A$ must be
inverse-positive{\red\cite{Hen}}. This ensures the proof.
\end{proof}

\begin{lemma}
Let $A\in\RR^{N\times N}$ be nonsingular and nonnegative,
and $\beta\in\RR^{N}$ be positive. Then $A \beta > 0.$
\end{lemma}

\begin{proof}
This is clear by the definition of the operations.
\end{proof}

\section{Monotonicity} \clearallnum
Another fundamental feature which distinguishes a quenching solution from other blow-up type solutions
is its monotonicity with respect to time $t\geq t_0$ \cite{Chan2,Levine,Acker2,Acker1,Sheng4}. It is therefore
necessary to guarantee that a numerical solution preserves this important physical property when
solving the Kawarada equation \R{b1}-\R{b3} or \R{c1}-\R{c2b}.

{\red
\begin{lemma}
If $Mv_0+g(v_0)>0,$ then it follows that $Mv_\ell + g(v_\ell)>0$ for all $\ell \ge 0.$
\end{lemma}

\begin{proof}
First, we proceed be computing the following:
\bbbb
Mv_{k + 1} + g(v_{k+1})& = & M\l[\l(I-\frac{\tau_k}{2}M\r)^{-1}\l(I+\frac{\tau_k}{2}M\r)\l(v_k + \frac{\tau_k}{2}g(v_k)\r)+\frac{\tau_k}{2}g(v_{k+1})\r]+g(v_{k+1})\nnn\\
& = & M\l(I-\frac{\tau_k}{2}M\r)^{-1}\l[\l(I+\frac{\tau_k}{2}M\r)\l(v_k + \frac{\tau_k}{2}g(v_k)\r)\r] + \l(I+\frac{\tau_k}{2}M\r)g(v_{k+1})\\
& > & M\l(I-\frac{\tau_k}{2}M\r)^{-1}\l[\l(I+\frac{\tau_k}{2}M\r)\l(v_k + \frac{\tau_k}{2}g(v_k)\r)\r] + \l(I+\frac{\tau_k}{2}M\r)g(v_{k})\\
& = & \l(I-\frac{\tau_k}{2}M\r)^{-1}\l(I+\frac{\tau_k}{2}M\r)\l[Mv_k + \frac{\tau_k}{2}Mg(v_k) + \l(I-\frac{\tau_k}{2}M\r)g(v_{k})\r]\\
& = & \l(I-\frac{\tau_k}{2}M\r)^{-1}\l(I+\frac{\tau_k}{2}M\r)\l[Mv_k + g(v_k)\r],
\eeee
where the inequality follows from the fact that $f(\varepsilon,u)$ is strictly increasing and Lemma 2.3. Second, we proceed by induction. Letting $k = 0$ we have
$$Mv_1+g(v_1) > \l(I-\frac{\tau_0}{2}M\r)^{-1}\l(I+\frac{\tau_0}{2}M\r)\l[Mv_0 + g(v_0)\r] > 0$$
by the assumption and Lemmas 2.2 and 2.3. Third, we assume that the inequality holds for $k=\ell-1.$ It follows that
$$Mv_\ell+g(v_\ell) > \l(I-\frac{\tau_{\ell-1}}{2}M\r)^{-1}\l(I+\frac{\tau_{\ell-1}}{2}M\r)\l[Mv_{\ell-1} + g(v_{\ell-1})\r] > 0,$$
by the inductive assumption and Lemmas 2.2 and 2.3, which completes the induction.
\end{proof}}

{\red
\begin{lemma}
If \R{cfl} holds, $0\le \tau_k \le 1$ for all $0\le k\le\ell,$ and $Mv_0+g(v_0)>0,$
then $v_{\ell+1}\ge v_\ell \mbox{ for all } \ell\ge 0.$ That is, the sequence $\l\{v_\ell\r\}_{\ell=0}^{\infty}$
is monotonically increasing.
\end{lemma}

\begin{proof}
From \R{c3} we observe that
\bbbb
v_{k+1} - v_k & = & \l(I-\frac{\tau_k}{2}M\r)^{-1}\l(I+\frac{\tau_k}{2}
M\r)\l(v_k+\frac{\tau_k}{2}g(v_k)\r)+\frac{\tau_k}{2}g(v_{k+1})-v_k\nnn\\
& > & \l(I-\frac{\tau_k}{2}M\r)^{-1}\l(I+\frac{\tau_k}{2}
M\r)\l(v_k+\frac{\tau_k}{2}g(v_k)\r)+\frac{\tau_k}{2}g(v_{k})-v_k\nnn\\
& = & \l(I-\frac{\tau_k}{2}M\r)^{-1}\l\{\l(I+\frac{\tau_k}{2}M\r)\l[v_k+\frac{\tau_k}{2}g(v_k)\r] - \l(I-\frac{\tau_k}{2}M\r)\l[v_k - \frac{\tau_k}{2}g(v_{k})\r]\r\}\nnn\\
& = & \l(I-\frac{\tau_k}{2}M\r)^{-1}\l(I+\frac{\tau_k}{2}M\r)\l[\tau_k\l(Mv_k +g(v_k)\r)\r]\\
& > & 0,
\eeee
by the assumption and Lemmas 2.2 and 2.3. Since the result holds for all $k\ge 0,$ we have that the sequence $\{v_\ell\}_{\ell\ge 0}$ is monotonically increasing as desired.
\end{proof}}

\begin{lemma}
\label{Lem42}
Let $x=(1,1,\dots,1)^{\tT}\in\RR^N.$ Then for any $\tau_\ell > 0$ we have
$$\l(I-\frac{\tau_\ell}{2}M\r)x \ge x.$$
\end{lemma}

\begin{proof}
Consider
$$w ~=~\l(I-\frac{\tau_\ell}{2}M\r)x ~=~ (w_{1},w_{2},\dots,w_{N})^{\tT}.$$
First, we observe that
\bbbb
w_{1} & = & \l(1-\frac{\tau_\ell}{2}\cdot\frac{-2}{a^2 \sigma_1 h_{0}h_{1}}\r) - \frac{\tau_\ell}{2}\cdot\frac{2}{a^2\sigma_1 h_{1}(h_{0}+h_{1})}\\
& = & 1 + \frac{\tau_\ell}{a^2\sigma_1}\l(\frac{1}{h_{0}h_{1}} - \frac{1}{h_{1}(h_{0}+h_{1})}\r)
 ~>~  1.
\eeee
Secondly, for $i=2,\dots,N-1,$ we have
\bbbb
w_{i} & = & -\frac{\tau_\ell}{2}\cdot\frac{2}{a^2\sigma_ih_{i-1}(h_{i-1}+h_{i})} +
\l(1-\frac{\tau_\ell}{2}\cdot\frac{-2}{a^2\sigma_i h_{i-1}h_{i}}\r)
- \frac{\tau_\ell}{2}\cdot\frac{2} {a^2\sigma_i h_{i}(h_{i-1}+h_{i})}\\
& = & 1 + \frac{\tau_\ell}{a^2\sigma_i}\l[\frac{-h_{i} + (h_{i-1}+h_{i}) - h_{i-1}}{h_{i-1}h_{i}(h_{i-1}+h_{i})}\r]
~ = ~ 1.
\eeee
Finally, we observe that
\bbbb
w_{N} & = & -\frac{\tau_\ell}{2}\cdot\frac{2}{a^2\sigma_Nh_{N-1}(h_{N-1}+h_{N})} +
\l(1-\frac{\tau_\ell}{2}\cdot\frac{-2}{a^2\sigma_N h_{N-1}h_{N}}\r)\\
& = & 1 + \frac{\tau_\ell}{a^2\sigma_N}\l[\frac{1}{h_{N}(h_{N-1}+h_{N})}\r] ~ > ~ 1.
\eeee
Hence, we may conclude that $w_{i} \ge 1,~i=1,\dots,N.$ Therefore $w\ge x.$
\end{proof}

In the next lemma we show that numerical quenching, {\em i.e.,\/}
one or more components of $v_\ell$ reaching or exceeding unity, cannot occur
immediately after the first time step under appropriate constraints. 

\begin{lemma}
Let $x$ be the vector defined in Lemma \ref{Lem42} and $v_0\equiv 0.$
If \R{cfl} hold and $\bar{h}^2 < {2\|B\|_2}/[{a^2f(\tau_0\varphi_{\max}f_0/\sigma_{\min})}],$ {\red where $\bar{h}=\textstyle\max_{i=1,\ldots,N}\l\{h_{i}\r\},$}
then $v_1<x.$
\end{lemma}

\begin{proof}
Recall \R{c3}. For $v_0\equiv 0$ we have{
$$v_1 ~~=~~ \l(I-\frac{\tau_0}{2}M\r)^{-1}\l(I+\frac{\tau_0}{2}M\r)\frac{\tau_0}{2}g(0) +
\frac{\tau_0}{2}g(v_1).$$}
Using
$$g(v_1) \approx g(w_0) = g(v_0 + \tau_0(Mv_0+g(v_0))) = g(\tau_0f_0\gamma),$$
where $\gamma = \l(\varphi_1/\sigma_1,\dots,\varphi_N/\sigma_N\r)^{\tT}\in\RR^{N},$ we have
\bbb
\l(I-\frac{\tau_0}{2}M\r)\l(v_1-\frac{\tau_0}{2}g(\tau_0f_0\gamma)\r) & = & \l(I+\frac{\tau_0}{2}M\r)\frac{\tau_0}{2}f_0\gamma.\label{37c}
\eee
Based on \R{37c} we observe that
\bbbb
v_1 - x & = & \l(I-\frac{\tau_0}{2}M\r)^{-1}\l[\l(I+\frac{\tau_0}{2}M\r)\frac{\tau_0}{2}f_0\gamma + \l(I-\frac{\tau_0}{2}M\r)\l(\frac{\tau_0}{2}g(\tau_0f_0\gamma)-x\r)\r]\\
& = & \l(I-\frac{\tau_0}{2}M\r)^{-1}\l(s^++s^-\r),
\eeee
where
$$s^+=\frac{\tau_0}{2}\l(I+\frac{\tau_0}{2}M\r)f_0\gamma+\frac{\tau_0}{2}\l(I-\frac{\tau_0}{2}M\r)g(\tau_0f_0\gamma),~~s^-=
-\l(I-\frac{\tau_0}{2}M\r)x.$$
It can be seen that
\bbbb
\l|s^+\r| & = &\frac{\tau_0}{2} \l|\l(I+\frac{\tau_0}{2}M\r)f_0\gamma+\l(I-\frac{\tau_0}{2}M\r)g(\tau_0f_0\gamma)\r|\\
& \le & \frac{\tau_0}{2}\max\l\{\l| f_0\gamma\r|,\l|g(\tau_0f_0\gamma)\r|\r\}\l\|\l(I+\frac{\tau_0}{2}M\r)+\l(I-\frac{\tau_0}{2}M\r)\r\|_2\\
& < & \frac{\bar{h}^2}{2\|B\|_2}a^2f(\tau_0\varphi_{\max}f_0/\sigma_{\min}),
\eeee
and the above indicates that
$$s^+ ~~\le~~ \frac{a^2\bar{h}^2f(\tau_0\varphi_{\max}f_0)}{2\|B\|_2}x.$$
By Lemma {\red 3.3} we conclude that $s^- \le -x,$ and therefore,
\bbbb
s^+ + s^- & \le & \l[\frac{a^2\bar{h}^2f(\tau_0\varphi_{\max}f_0/\sigma_{\min})}{2\|B\|_2} - 1\r]x.
\eeee
Since we again wish each component of the above vector to be negative, we need
$$\frac{a^2\bar{h}^2f(\tau_0\varphi_{\max}f_0/\sigma_{\min})}{2\|B\|_2} - 1 ~~<~~ 0,~~\mbox{or}~~ \bar{h}^2~~
<~~ \frac{2\|B\|_2}{a^2f(\tau_0\varphi_{\max}f_0/\sigma_{\min})}.$$
Hence $v_1-x\le 0$ follows immediately from the assumptions. \end{proof}

Combining above results we obtain immediately the following.

\begin{thm}
Assume that for $\ell_0\geq 0,$
\begin{description}
\item[{\rm(i)}] $\bar{h}^2 < \frac{2\|B\|_2}{a^2f(\tau_{\ell_0}\varphi_{\max}f_{\ell_0}/\sigma_{\min})},$ where $\bar{h}=\max_{i=1,\dots,N}\{h_{i}\},$
\item[{\rm(ii)}] $Mv_{\ell_0}+g(v_{\ell_0})>0,$
\end{description}
If \R{cfl} holds for all $\ell\ge \ell_0,$
then the sequence $\l\{v_\ell\r\}_{\ell\ge {\ell_0}}$ produced by the semi-adaptive
nonuniform scheme \R{c3} increases monotonically until unity is reached or exceeded by
one or more components of the solution vector, i.e., until quenching occurs.
\end{thm}

\section{Stability} \clearallnum
Nonlinear stability has been an extremely challenging issue when Kawarada equations are concerned
\cite{Acker2,Sheng4,Beau1,Josh2,Sheng3,Josh1}. On the other hand, linear stability
analysis may uncover crucial information for underlying schemes locally and asymptotically \cite{Sheng4,Twi}.
In the following study, we will carry out a linearized
stability analysis in the von Neumann sense for \R{c3} with its nonlinear source term frozen.
The analysis will then be extended to circumstances where the nonlinear term is not frozen.

In the following, let $A\in\mathbb{C}^{N\times N},$ $I_N\in\mathbb{C}^{N\times N}$ be the identity matrix, and again denote $E(\cdot) = \exp(\cdot).$

\begin{define}
{\rm\cite{Josh1,Golub}}
Let $\|\cdot\|$ be an induced matrix norm. Then the
associated logarithmic norm $\mu : \mathbb{C}^{N\times N}\to \RR$ of $A$ is defined as
$$\mu(A) = \lim_{h\to 0^+} \frac{\|I_N + hA\| - 1}{h}.$$
\end{define}

\begin{rem}
When considering the spectral norm, we have $\mu(A) = \lambda_{\max}[(A+A^*)/2].$
\end{rem}

\begin{lemma}
{\rm\cite{Golub}}  For $t\ge 0$ we have $\|E(t A)\| \le E(t \mu(A)).$
\end{lemma}


\begin{lemma}
{\red Let $P$ be as in \R{matrixT}. Then $P$} is congruent to a symmetric matrix. In particular,
$${\red P} = D^{-1/2}SD^{1/2}\in\RR^{N\times N},$$
where
$$D = \mbox{\rm diag}\l(\delta_1,\dots,\delta_N\r),\ S = \mbox{\rm tridiag}\l(\alpha_i,m_i,\alpha_i\r)$$
for which
$$\delta_j = \frac{h_{j-1}+h_j}{h_0+h_1}\quad\mbox{and}\quad \alpha_k = \sqrt{n_kl_k}.$$
\end{lemma}

\begin{proof}
The proof is clear by the definitions {\red \cite{Beau11}}.
\end{proof}

\begin{lemma}
$\rho({\red P})\in (-\infty,0],$ where $\rho(A) = \sqrt{\lambda_{\max}(AA^{\tT})}~$  for any  $A\in\RR^{N\times N}.$
\end{lemma}

\begin{proof}
By Lemma 4.2 we have that ${\red P}$ is congruent to a symmetric matrix, hence all eigenvalues of ${\red P}$ are real. 
Since ${\red P}$ is diagonally dominant with negative diagonal elements, the result follows immediately from the 
{\red Gershgorin circle theorem \cite{Hen}}.
\end{proof}

\begin{lemma}
All eigenvalues of $M$ are real and negative. Further $\mu(B^{1/2}{\red P}B^{1/2})< 0.$
\end{lemma}

\begin{proof}
Since $M = B{\red P} = B^{1/2}B^{1/2}{\red P},$ we have $B^{-1/2}M = B^{1/2}{\red P}.$ Hence
$$B^{-1/2}M(B^{1/2})^{\tT} = B^{1/2}{\red P}B^{1/2} = B^{1/2}{\red P}(B^{1/2})^{\tT}$$
because $B$ is diagonal. Further,
$$B^{-1/2}MB^{1/2} = B^{-1/2}B{\red P}B^{1/2} = B^{1/2}{\red P}B^{1/2}$$
is congruent to a symmetric matrix. Thus matrices $B^{-1/2}MB^{1/2}$ and ${\red P}$ are congruent. Since the eigenvalues of ${\red P}$ are real and negative, the eigenvalues of $B^{-1/2}MB^{1/2}$ and $M$ must be real and negative.

Since we have $B^{-1/2}MB^{1/2} = B^{1/2}{\red P}B^{1/2},$ then the eigenvalues of $B^{1/2}{\red P}B^{1/2}$ are real and negative, which gives the result.
\end{proof}

\begin{lemma}
If $\tau_k >0,~0\le k\le \ell,$ we have $\l\|\prod_{k=0}^\ell E(\tau_k M)\r\|_2 \le \sqrt{\kappa(B)}~ {\red = \sqrt{\|B^{-1}\|_2\|B\|_2}}.$
\end{lemma}

\begin{proof}
We first show $E(\tau_{\red k} M) = B^{1/2}E(\tau_{\red k} B^{1/2}{\red P}B^{1/2})B^{-1/2}$ for $\tau_{\red k} >0.$
To this end, we have
\bb{xx}
B^{-1/2}E(\tau_{\red k} M)B^{1/2} ~=~ \sum_{j=0}^\infty \frac{\l(\tau_{\red k} B^{1/2}{\red P}B^{1/2}\r)^j}{j!}
~ = ~ E(\tau_{\red k} B^{1/2}{\red P}B^{1/2}).
\ee
It then follows from \R{xx},
\bbbb
\l\|\prod_{k=0}^\ell E(\tau_{\red k} M)\r\|_2 & = & \l\|\prod_{k=0}^\ell\l(B^{1/2}E(\tau_k B^{1/2}{\red P}B^{1/2})B^{-1/2}\r)\r\|_2\\
& = & \l\|B^{1/2}\l(\prod_{k=0}^\ell E(\tau_k B^{1/2}{\red P}B^{1/2})\r)B^{-1/2}\r\|_2\\
& \le & \|B^{1/2}\|_2\|B^{-1/2}\|_2\prod_{k=0}^\ell E(\tau_k \mu (B^{1/2}{\red P}B^{1/2}))
~\leq~  \sqrt{\kappa(B)}.
\eeee
\end{proof}

\begin{lemma}
If \R{cfl} holds and $0\le \tau_k \le 1$ for all $0\le k\le \ell,$ then
$$\l\|\prod_{k=0}^\ell\l(I-\frac{\tau_{\red k}}{2}M\r)^{-1}\l(I+\frac{\tau_{\red k}}{2}M\r)\r\|_2 \le C.$$
\end{lemma}

\begin{proof}
Recalling the [1/1] Pad{\'e} approximation utilized
in Section 2, we have
$$\l(I-\frac{\tau_k}{2}M\r)^{-1}\l(I+\frac{\tau_k}{2}M\r) ~=~ E(\tau_k M) + \OO\l(\;\tau_k^3\;\r).$$
By the definition of $T$ and $\{\tau_\ell\}_{\ell\ge 0},$ we have $\textstyle\sum_{k=0}^\ell \tau_k \le T$ 
and $\tau = \textstyle\max_{0\le k\le\ell} \{\tau_k\}\le 1.$ Now, based on Lemma 4.5,
\bbbb
\l\|\prod_{k=0}^\ell\l(I-\frac{\tau_k}{2}M\r)^{-1}\l(I+\frac{\tau_k}{2}M\r)\r\|_2
& = & \l\|\prod_{k=0}^\ell E(\tau_k M) + \OO\l(\tau_k^3\r)\r\|_2
~\le ~ \l\|\prod_{k=0}^\ell E(\tau_k M)\r\|_2 + c\tau^2\sum_{k=0}^\ell\tau_k\\
& \le & \sqrt{\kappa(B)} + cT ~~\le~~ C,
\eeee
which yields the desired bound.
\end{proof}

Combining the above results gives the following theorem.

\begin{thm}
If \R{cfl} holds and $0\le \tau_k\le 1$ for all $0\le k\le \ell,$ then the semi-adaptive nonuniform
method \R{c3} with the source term frozen is unconditionally stable in the von Neumann sense
under the spectral norm, that is,
$$\|z_{\ell+1}\|_2 \leq C \|z_{0}\|_2,~~~\ell\geq 0, $$
where $z_0=v_0-\tilde{v}_0$ is an initial error, $z_{\ell+1}=v_{\ell+1}-\tilde{v}_{\ell+1}$ is the
$(\ell+1)$th perturbed error vector, and $C>0$ is a constant independent of $\ell$ and $\tau_k$ for each $0\le k\le \ell.$
\end{thm}

\begin{proof}
When the nonlinear source term is frozen, $z_{\ell+1}$ takes the form of
\bbb
z_{\ell+1} &=& \l(I-\frac{\tau_\ell}{2}M\r)^{-1}\l(I+\frac{\tau_\ell}{2}M\r) z_\ell,~~~\ell\geq 0.\label{a}
\eee
Iterating \R{a} gives
\bb{st11}
z_{\ell+1} ~~=~~ \prod_{k=0}^\ell \l(I-\frac{\tau_{\red k}}{2}M\r)^{-1}\l(I+\frac{\tau_{\red k}}{2}M\r)z_0.
\ee
Taking the norm of both sides of \R{st11}, it follows that
$$\|z_{\ell+1}\|_2  ~\le~  \l\|\prod_{k=0}^\ell\l(I-\frac{\tau_{\red k}}{2}M\r)^{-1}\l(I+\frac{\tau_{\red k}}{2}
M\r)\r\|_2\|z_0\|_2 ~\le~ C\|z_0\|_2,$$
where $C$ is a positive constant independent of $\ell$ and $\tau_{k}$ for each $0\leq k\leq\ell.$
\end{proof}

We now consider the case without freezing the nonlinear source term in \R{c3}. In the following, 
let $t_Q$ be the time at which numerical quenching occurs, that is, $\|v_Q\|_\infty \ge 1,$ and 
recall that $t_m = t_0+\textstyle\sum_{k=0}^{m-1}\tau_k$ for any $m\ge 0.$ Also, let
$$\Phi_k = \l(I-\frac{\tau_k}{2}M\r)^{-1}\l(I+\frac{\tau_k}{2}M\r),~0\le k\le \ell.$$

\begin{thm}
If \R{cfl} holds and $\tau_k$ sufficiently small for all $0\le k\le \ell,$ then the semi-adaptive nonuniform 
method \R{c3} is unconditionally stable in the von Neumann sense, that is, 
for every $t_m <t_Q$ there exists a constant $C(t_m)>0$ such that
$$ \|z_{\ell+1}\|_2 \leq C(t_m)\|z_{0}\|_2,~~~0\le \ell\le m, $$
where $z_0=v_0-\tilde{v}_0$ is an initial error, $z_{\ell+1}=v_{\ell+1}-\tilde{v}_{\ell+1}$ is the
$(\ell+1)$th perturbed error vector, and $C(t_m)>0$ is a constant independent of $\ell$ and $\tau_k$ for each $0\le k\le\ell.$
\end{thm}

\begin{proof}
By definition we have
\bbbb
v_{\ell+1} & = & \Phi_\ell \l(v_\ell+\frac{\tau_\ell}{2}g(v_\ell)\r) + \frac{\tau_\ell}{2}g(v_{\ell+1}),~~~\ell\ge 0.
\eeee
It follows that
\bbbb
z_{\ell+1} & = & \Phi_\ell z_\ell + \frac{\tau_\ell}{2}\Phi_\ell(g(v_\ell) - g(\tilde{v}_\ell)) + \frac{\tau_\ell}{2}(g(v_{\ell+1})-g(\tilde{v}_{\ell+1}))\\
& = & \Phi_\ell z_\ell + \frac{\tau_\ell}{2}\Phi_\ell g_v(\xi_\ell)z_\ell + \frac{\tau_\ell}{2}g_v(\xi_{\ell+1})z_{\ell+1},\\
\eeee
where $\xi_k\in\LL(v_k,\tilde{v}_k),~k=\ell,\ell+1.$ Rearranging the above equality, we have
$$\l(I-\frac{\tau_\ell}{2}g_v(\xi_{\ell+1})\r)z_{\ell+1}
= \Phi_\ell\l(I+\frac{\tau_\ell}{2}g_v(\xi_\ell)\r)z_\ell.$$
Further, it follows that for {\red $\tau_k\rightarrow 0^+$ being sufficiently small,}
$$\l(I-\frac{\tau_k}{2}g_v(\xi)\r)^{-1}=E\l(\frac{\tau_k}{2}g_v(\xi)\r)+\OO\l(\tau_k^2\r)\quad\mbox{and}\quad I+\frac{\tau_k}{2}g_v(\xi)
=E\l(\frac{\tau_k}{2}g_v(\xi)\r)+\OO\l(\tau_k^2\r).$$
Thus,
\bbbb
z_{\ell+1} & = & \l(I-\frac{\tau_\ell}{2}g_v(\xi_{\ell+1})\r)^{-1}\Phi_\ell\l(I+\frac{\tau_\ell}{2}g_v(\xi_\ell)\r)z_\ell\nnn\\
& =& \l\{\prod_{k=0}^\ell \l[E\l(\frac{\tau_k}{2}g_v(\xi_{k+1})\r)+\OO\l(\tau_k^2\r)\r]\Phi_k \l[E\l(\frac{\tau_k}{2}g_v(\xi_k)\r)+\OO\l(\tau_k^2\r)\r]\r\}z_0\\
& =& \l\{\prod_{k=0}^\ell \l[E\l(\frac{\tau_k}{2}g_v(\xi_{k+1})\r)\r]\Phi_k \l[E\l(\frac{\tau_k}{2}g_v(\xi_k)\r)\r]+\OO\l(\sum_{k=0}^\ell\tau_k^2\r)\r\}z_0.
\eeee
Letting $G(t_m) = \max_{0\le k\le m} \l\|g_v(\xi_k(t))\r\|_2$ {\red and $\tilde{P} = B^{1/2}PB^{1/2}$} we have
\bbbb
\|z_{\ell+1}\|_2 & \le & \l\{\l\|\prod_{k=0}^\ell E\l(\frac{\tau_k}{2}g_v(\xi_{k+1})\r)\Phi_k E\l(\frac{\tau_k}{2}g_v(\xi_k)\r)\r\|_2+c_1\sum_{k=0}^\ell \tau_k^2 \r\}\|z_0\|_2\nnn\\
& {\red \le} & {\red\l\{\l\|\prod_{k=0}^\ell E\l(\frac{\tau_k}{2}g_v(\xi_{k+1})\r) E(\tau_kM) E\l(\frac{\tau_k}{2}g_v(\xi_k)\r)\r\|_2+c_2\sum_{k=0}^\ell \tau_k^2 \r\}\|z_0\|_2}\nnn\\
& {\red \le} & {\red\l\{\l\|\prod_{k=0}^\ell E\l(\frac{\tau_k}{2}g_v(\xi_{k+1})\r)B^{1/2} E(\tau_k{\red \tilde{P}})B^{-1/2} E\l(\frac{\tau_k}{2}g_v(\xi_k)\r)\r\|_2+c_2\sum_{k=0}^\ell \tau_k^2 \r\}\|z_0\|_2}\nnn\\
& {\red =} & {\red\l\{\l\|\prod_{k=0}^\ell B^{1/2}E\l(\frac{\tau_k}{2}g_v(\xi_{k+1})\r) E(\tau_k{\red \tilde{P}}) E\l(\frac{\tau_k}{2}g_v(\xi_k)\r)B^{-1/2}\r\|_2+c_2\sum_{k=0}^\ell \tau_k^2 \r\}\|z_0\|_2}\nnn\\
& {\red =} & {\red\l\{\l\|B^{1/2}\l[\prod_{k=0}^\ell E\l(\frac{\tau_k}{2}g_v(\xi_{k+1})\r) E(\tau_k{\red \tilde{P}}) E\l(\frac{\tau_k}{2}g_v(\xi_k)\r)\r]B^{-1/2}\r\|_2+c_2\sum_{k=0}^\ell \tau_k^2 \r\}\|z_0\|_2}\nnn\\
& {\red \le} & {\red\l\{\sqrt{\kappa(B)}\prod_{k=0}^\ell \l\|E\l(\frac{\tau_k}{2}g_v(\xi_{k+1})\r)\r\|_2 \l\|E(\tau_k{\red \tilde{P}})\r\|_2 \l\|E\l(\frac{\tau_k}{2}g_v(\xi_k)\r)\r\|_2+c_2\sum_{k=0}^\ell \tau_k^2 \r\}\|z_0\|_2}\nnn\\
& {\red \le} & {\red\l\{\sqrt{\kappa(B)}\prod_{k=0}^\ell E(\tau_kG(t_m)) + c_2\sum_{k=0}^\ell \tau_k^2 \r\}\|z_0\|_2}\nnn\\
& {\red \le} & {\red \l\{\sqrt{\kappa(B)}E\l(G(t_m)\sum_{k=0}^\ell \tau_k\r) + c_2\sum_{k=0}^\ell \tau_k\r\}\|z_0\|_2}\nnn\\
& {\red \le} & {\red \l\{\sqrt{\kappa(B)}E(G(t_m)T)+c_2\tau T\r\}\|z_0\|_2 ~ \le ~ C(t_m)\|z_0\|_2,}
\eeee
where $c_1,~c_2$ are positive constants independent of $\ell$ and $\tau_k,~k=0,1,\dots,\ell.$ Since $t_m<t_Q,$ 
it follows that $G(t_m)<\infty$ for all $t_m<t_Q.$ This yields the desired stability.
\end{proof}

{\red\begin{rem}
The nonuniform constant $G(t_m)$ is anticipated to be large within the interval $[t_0,t_m]\subset [t_0,t_Q).$ 
However, we observe in experiments that its values remain to be well-manageable as far as proper stopping
criteria are adopted. This allows for the claimed stability.
\end{rem}}

\section{Numerical Experiments} \clearallnum

We consider the following stochastic Kawarada model problem in the first three experiments:
\bbb
&&\sigma(x)u_t = \frac{1}{a^2}u_{xx}+\frac{\rho(\epsilon)}{(1-u)^\theta},\quad -1<x<1,~t_0<t\le T,\label{num1}\\
&&u(-1,t)=u(1,t)=0,\quad t>t_0,\label{num2}\\
&&u(x,t_0)=u_0(x),\quad -1<x<1,\label{num3}
\eee
where $T<\infty,~\theta>0,~u_0\in C[-1,1],$ and $0\leq u_0\ll 1.$

Without loss of generality, we set $\theta=1,$ $t_0=0,$ $u_0(x)=0.001(1-\cos(2\pi x)),~-1\leq x\leq 1,$ and let our temporal 
adaptations start once $v=\textstyle\max_{-1\le x\le 1}u(x,t)$ 
reaches a certain value $v^*.$ In most computational procedures, we adopt $v^*=0.90$ for such a triggering criterion \cite{Sheng4}. The main 
purposes of our computations are to examine the numerical method built, and to explore impacts of varying physical 
parameters used in the equation. In the first experiment we will focus on the consequences due to variations in domain sizes. 
Quenching behaviors and possible temporal blow-up times are recorded. The second experiment will illustrate impacts of the degeneracy 
on quenching profiles. The third experiment will explore the effects that a stochastic component plays on the numerical solution and 
quenching criteria. Our semi-adaptive algorithm \R{c3} coupled with \R{num2}, \R{num3} is shown to be satisfactorily reliable, 
effective, and accurate. Finally, we extend our pursuits with a two-dimensional stochastic Kawarada equation problem to 
demonstrate the usability and effectiveness of our semi-adaptive infrastructure in higher-dimensional applications. A typical LOD \cite{ADI}
strategy is used.

\subsection*{Experiment 1}

Letting $\sigma(x)=\varphi(\epsilon)\equiv 1$ be fixed, 
we vary the value of $a$ to study its effects on the quenching phenomenon. Since quenching, if it exists, must occur 
at $x_q=0$ \cite{Sheng4,Sheng3}, we may consider herein a set of symmetric grids for the simplicity of
computations. {\red We generate the nonuniform grids in the following way. Consider a parabola designed to have 
a prescribed minimum at $x_{\lfloor (N+2)/2\rfloor}$ and a prescribed maximum at $x_0$ and $x_{N+1}.$ 
Then a standard arc-length adaptation procedure based on the curvature \cite{Beau1} 
produces the nonuniform grids, which is scaled to fit $[-1, 1].$ }

\begin{center}
{\epsfig{file=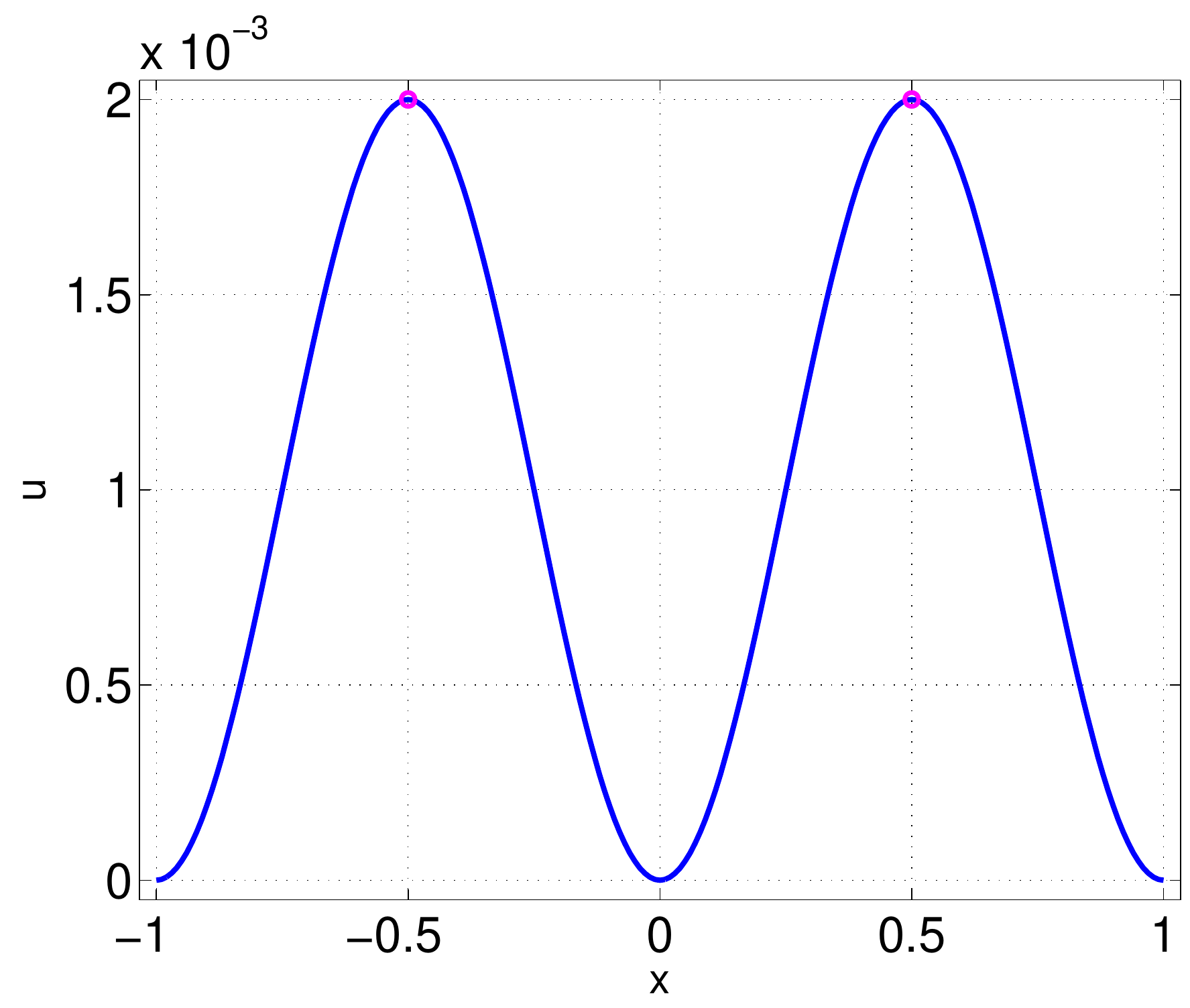,width=2.3in,height=1.28in}}~~
{\epsfig{file=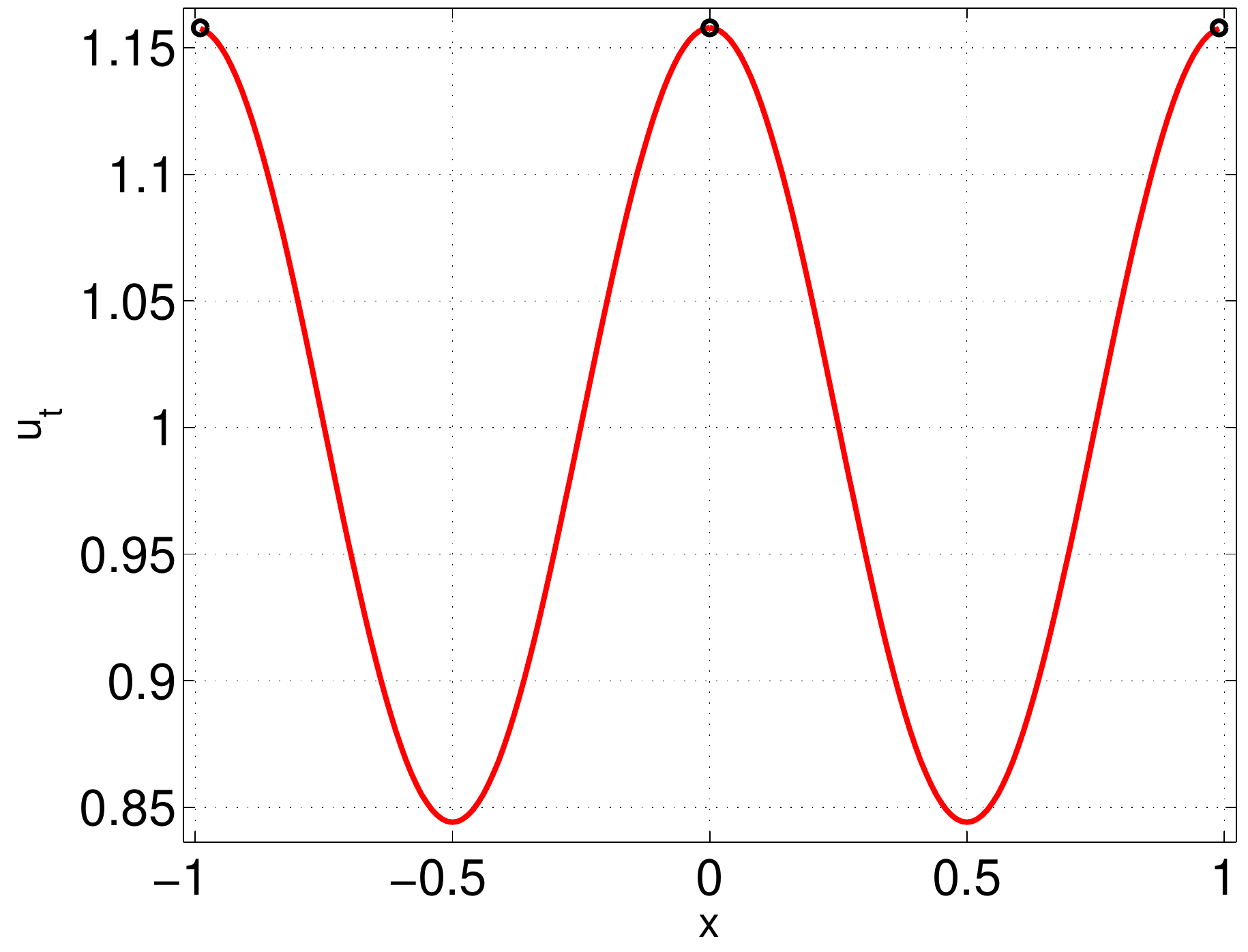,width=2.3in,height=1.19in}}

\parbox[t]{12.8cm}{\scriptsize{\bf Figure 1.} Initial function $u_0(x)=0.001(1-\cos(2\pi x))$ [LEFT] and its
estimated temporal derivative via \R{num1} [RIGHT]. Locations of multiple maximal values
of the functions concerned are symmetric with respect to the origin.}
\end{center}

To commence, we adopt $a=0.5$ and plot $u_0(x)$ and its reference temporal derivative in Fig 1. We note 
that the locations of the twin peaks of the initial functions chosen and triple peaks of the temporal derivative are symmetric 
with respect to the origin.

\begin{center}
{\epsfig{file=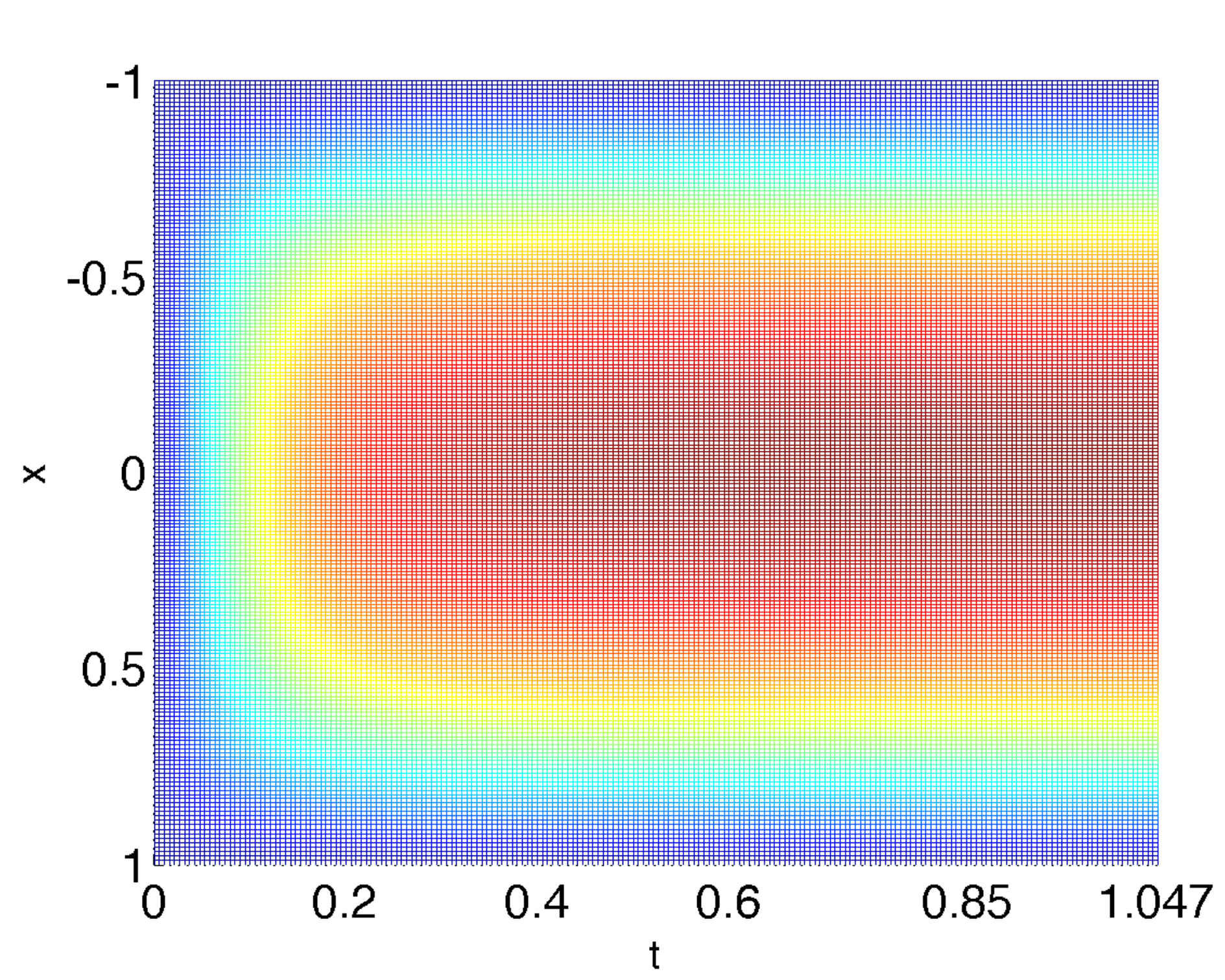,width=2.3in,height=1.28in}}~~
{\epsfig{file=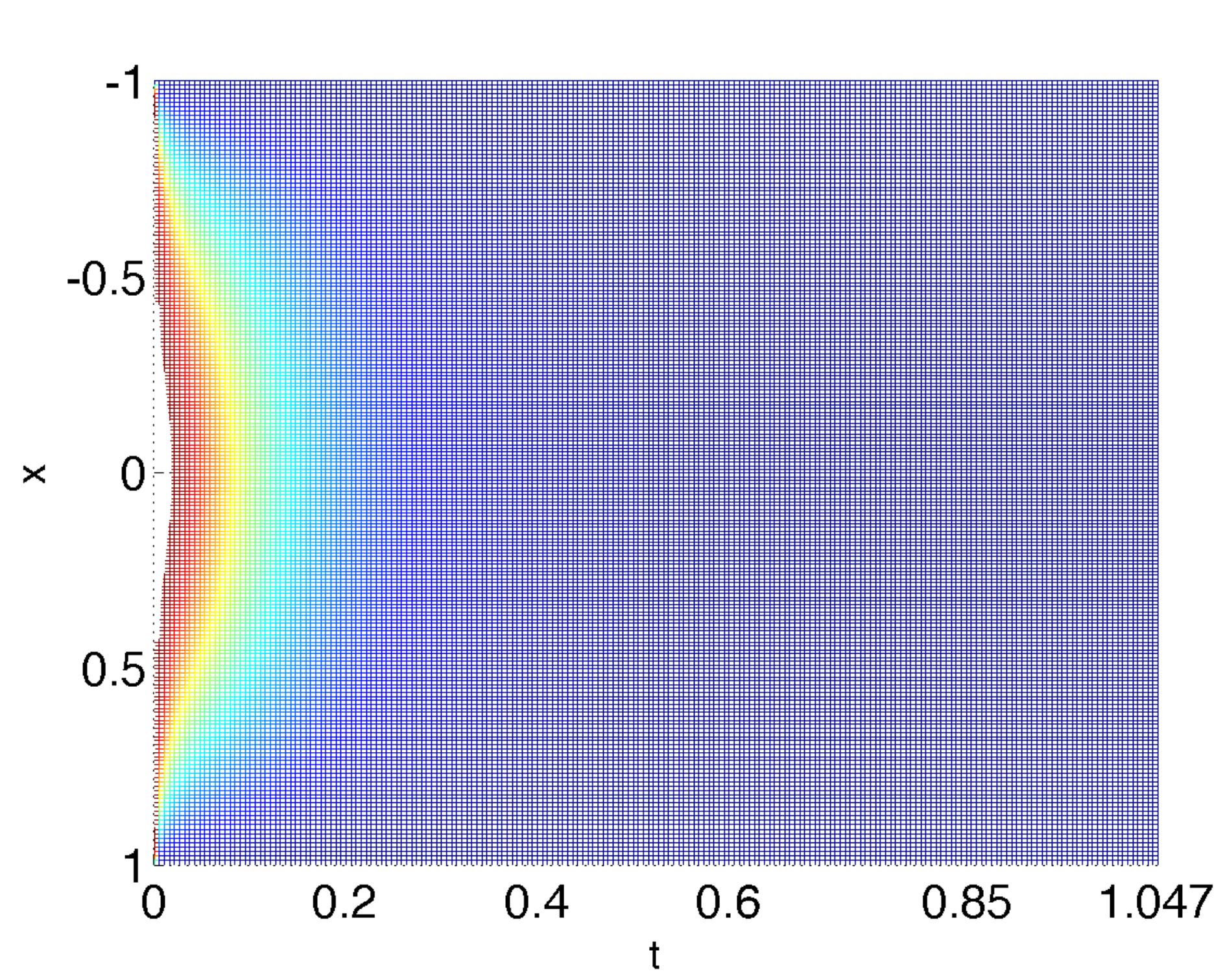,width=2.3in,height=1.28in}}

\parbox[t]{12.8cm}{\scriptsize{\bf Figure 2.} Two-dimensional thermal flow plots of the numerical solution $u$ [LEFT] 
and $u_t$ [RIGHT] for $t\in[0,T].$  It is observed that while the heat flows from the left to the right smoothly,
symmetrically and increases monotonically in the first figure, the flow speed decreases rapidly but monotonically as time 
goes on. There is no quenching found in this situation.}
\end{center}

\begin{center}
{\epsfig{file=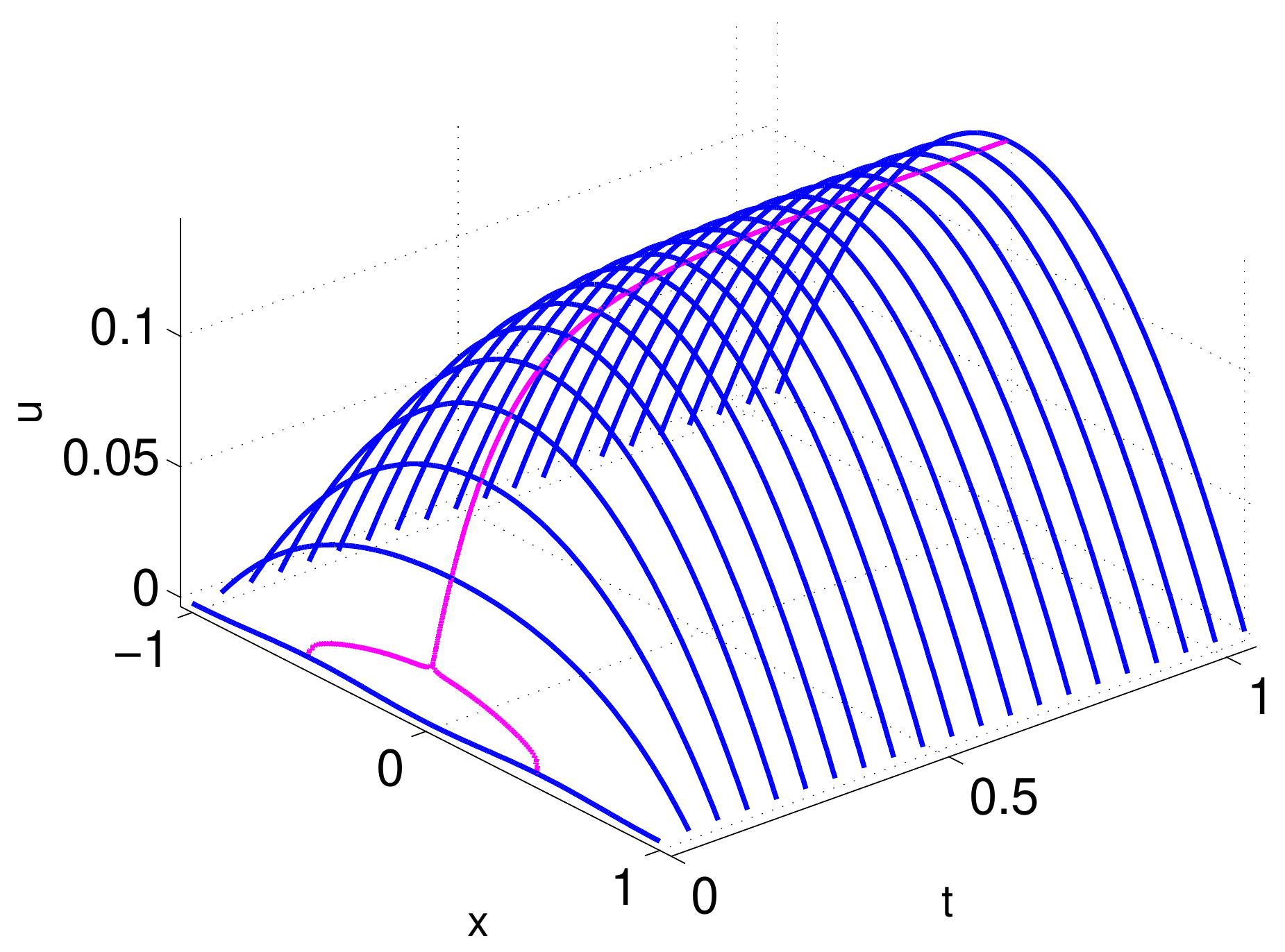,width=2.5in,height=1.38in}}~~
{\epsfig{file=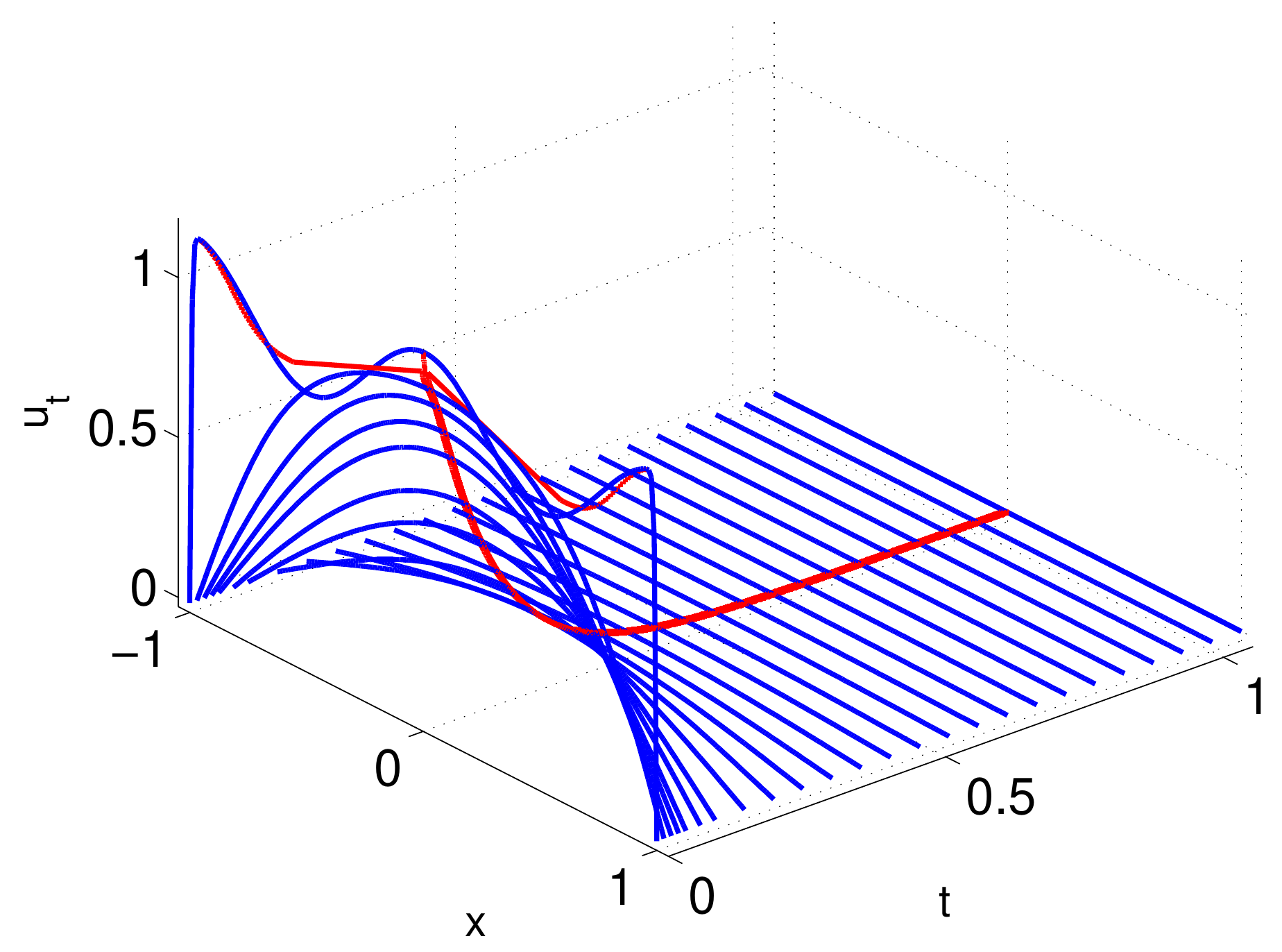,width=2.5in,height=1.38in}}

\parbox[t]{12.8cm}{\scriptsize{\bf Figure 3.} Three-dimensional plots of the numerical solution curvatures of $u$ [LEFT] and 
$u_t$ [RIGHT] for $t\in[0,T].$ The magenta curve represents the maximal value trajectory of $u$ [LEFT], 
and the red curve is for that of $u_t$ [RIGHT] at different $t$-levels for $0<t\leq T.$}
\end{center}

\begin{center}
{\epsfig{file=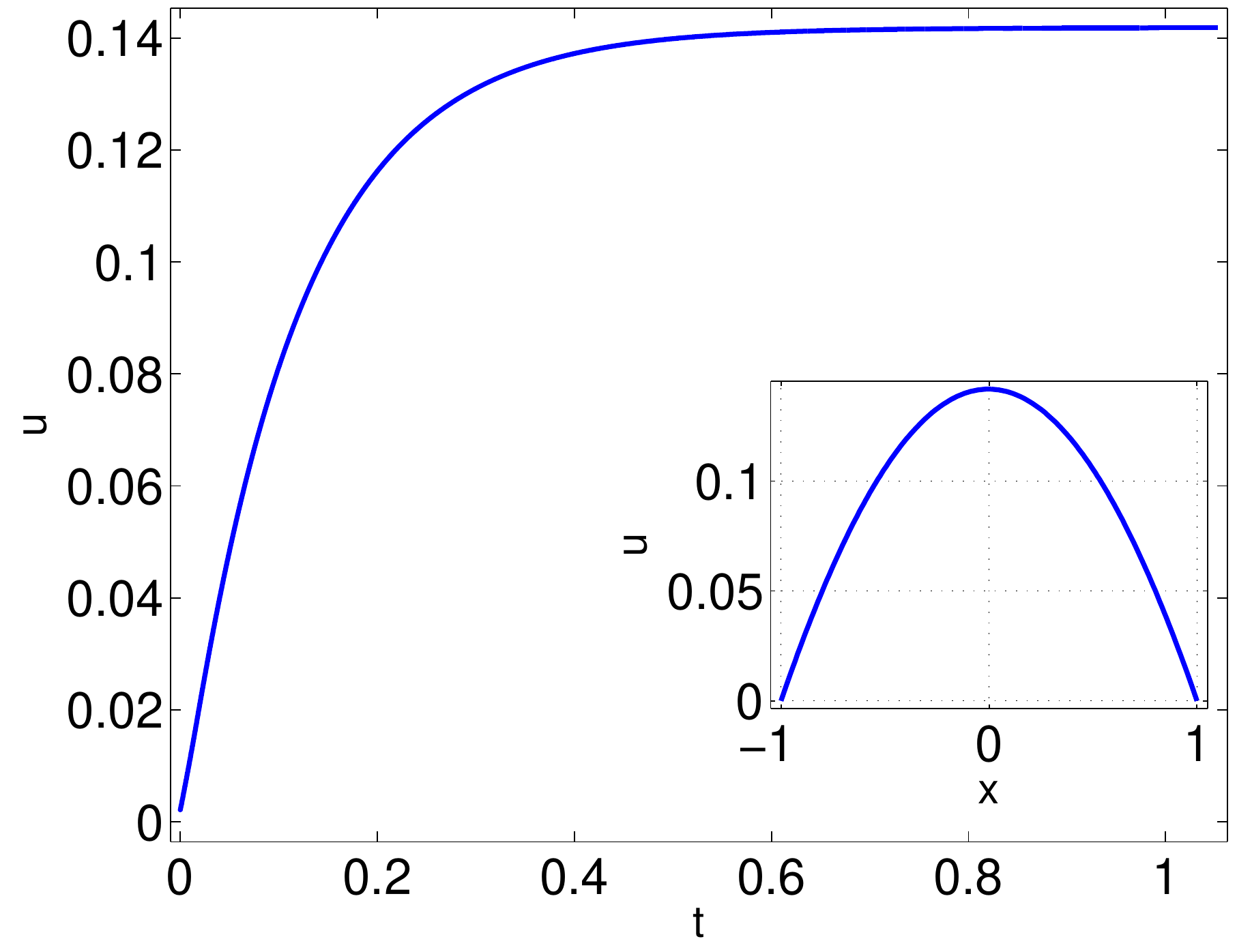,width=2.3in,height=1.28in}}~~
{\epsfig{file=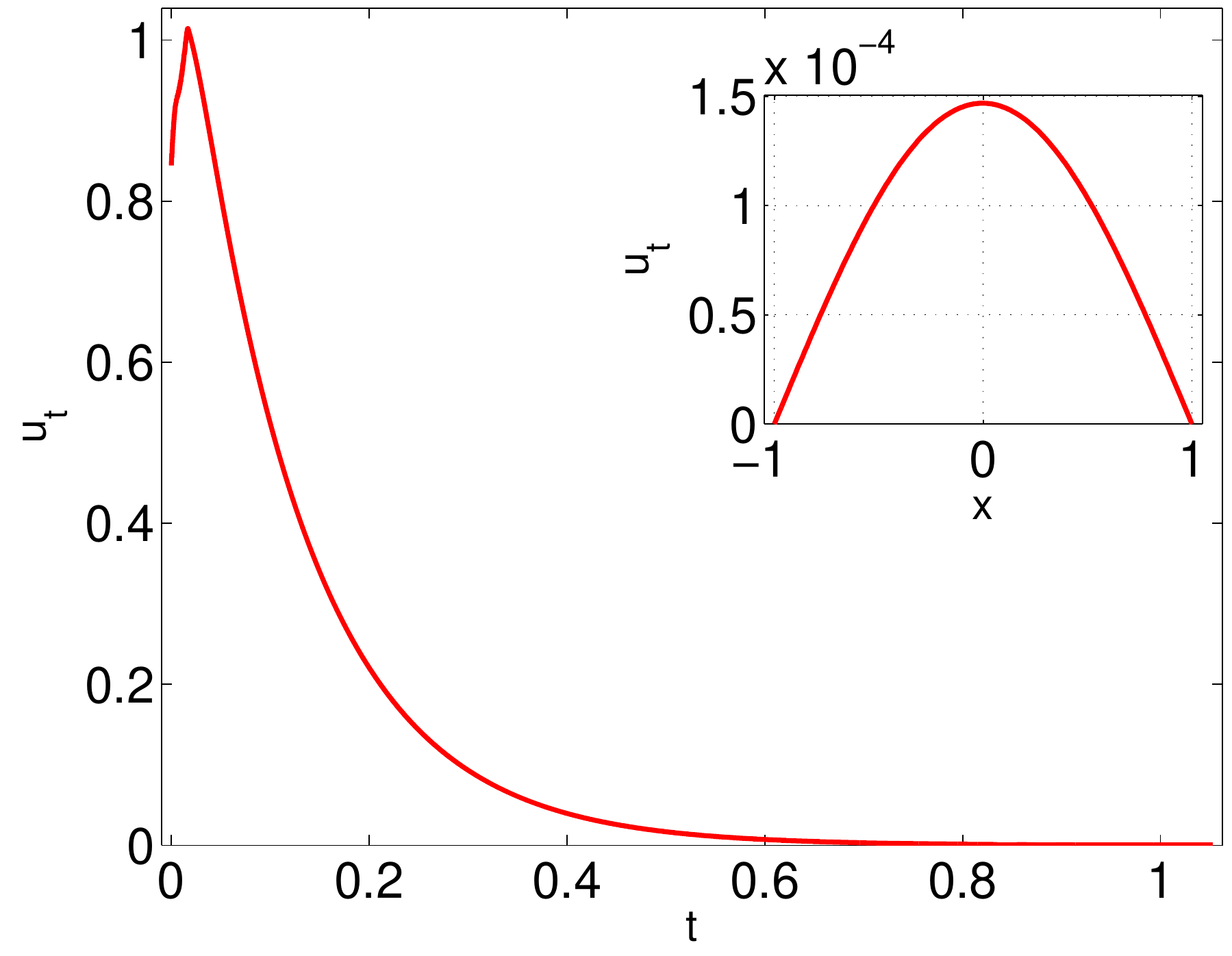,width=2.3in,height=1.28in}}

\parbox[t]{12.8cm}{\scriptsize{\bf Figure 4.} Curves in the main frames are for profiles of the maximal values of $u$ [LEFT] 
and $u_t$ (RIGHT). That is, they are plots of $\textstyle\max_{-1\leq x\leq 1} u,~0\le t\le T,$ and 
$\textstyle\max_{-1\leq x\leq 1}u_t,~0<\epsilon\le t\le T,$ respectively.
The embedded graphics represent profiles of $u$ [LEFT] and $u_t$ [RIGHT] at time $t=T.$ Up to 200,000 temporal steps 
have been executed. Temporal adaptation is never activated in this circumstance since no quenching is generated.}
\end{center}

Let $T=1.0479.$ In Fig. 2, the left and right graphics show the two-dimensional heat flow of $u$ and its velocity, $u_t,$ respectively.
It can be observed in the former that the heat flows smoothly and monotonically from the left to right. However, the multiple-peak
maximal values quickly merge into one and the flow tends to be steady as time increases. 
Temporal locations of the slices in Fig. 3 are chosen by evenly dividing the arc-length 
of the function $L(t) = \textstyle\max_{-1\leq x\leq 1}u(x,t).$ These plots again indicate that 
the solution $u$ tends to be steady with limited changes, while the rate of change function quickly tapers and
then diminishes when $t$ tends to $T.$ Both the numerical solution $u$ and temporal derivative $u_t$ 
preserve symmetry about $x=0$ as expected. The temporal adaptation is never activated as the maximum of the solution 
stays far below unity. The experimental results presented are consistent with existing results \cite{Levine,Sheng4,Sheng3}.

It is noticed in this experiment that $u$ exists globally. More detailed profiles of the maximal values of $u$ 
and $u_t$ at different time levels are shown in Fig. 4 for $0\leq t\leq T.$
Embedded figures are for $u$ and $u_t$ in terminal positions at $t=T.$
It is again observed that while the solution $u$ increases monotonically, the temporal derivative function $u_t$ decreases 
after some initial disturbances. In the terminal position $T_0=1.052907287028235,$ we have
$$\max_{-1\leq x\leq 1} u(x,T_0)\approx   0.141813667464453,~~\max_{-1\leq x\leq 1} u_t(x,T_0) \approx    1.468923350820044\times 10^{-4}.$$

In the next experiment, we choose $a=2$ in order to witness a quenching case for which the physical solution should 
exist only for finite time. We show three-dimensional profiles of the numerical solution $u$ and its temporal derivative 
$u_t$ in Fig. 5 for $0\leq t\leq T^*$ and $T^0\leq t\leq T^*,$ where $T^0 = 0.509286490538884,$
respectively. A quenching time $T^*\approx 0.509391490538887$ is recorded. The sole purpose of using the
temporal interval $[T^0,T^*]$ is to witness the explosive feature of $u_t$ immediately prior to quenching. This corresponds to 
the last 105 time steps in computations. The curvature functions are again selected 
via the arc-length of the maximal value function of $u.$ It is evident that while $u$ quenches peacefully as
$t\rightarrow T^*,$ the function $u_t$ blows up simultaneously. 

\begin{center}
{\epsfig{file=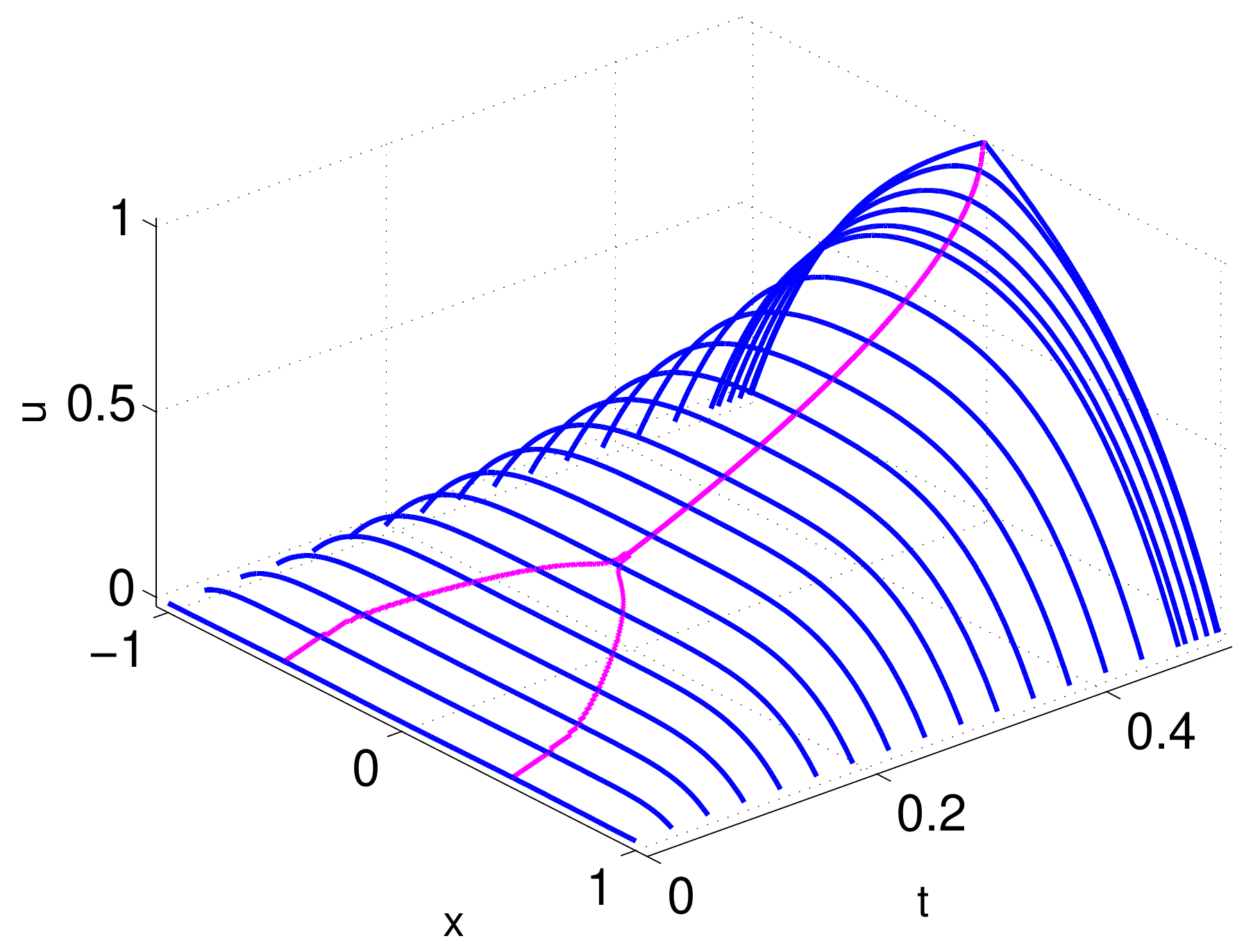,width=2.5in,height=1.38in}}~~
{\epsfig{file=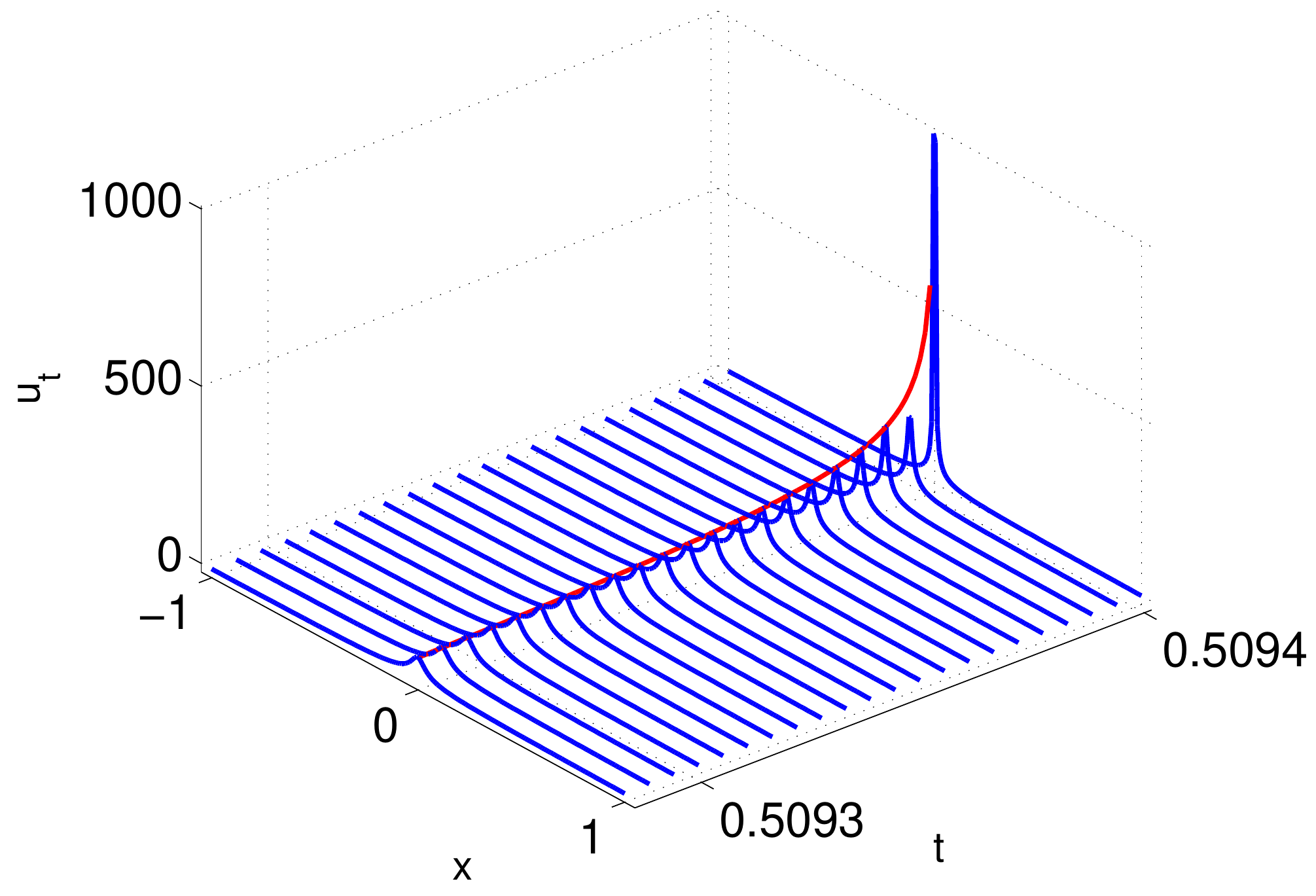,width=2.5in,height=1.38in}}

\parbox[t]{12.8cm}{\scriptsize{\bf Figure 5.} Three-dimensional curvature views of $u,~0\leq t\leq T^*,$ [LEFT] 
and $u_t,~T^0\leq t\leq T^*,$ [RIGHT] where $T^0=0.509286490538884$ and $T^*=0.509391490538887$
are used. The magenta and red curves represent functions $\textstyle\max_{-1\le x\le 1} u$ and $\textstyle\max_{-1\le x\le 1} u_t,$ 
respectively. The temporal derivative values concentrates about the quenching point 
with $\textstyle\max_{-1\le x\le 1} u_t(x,T^*) > 985\gg 1.$}
\end{center}

\begin{center}
{\epsfig{file=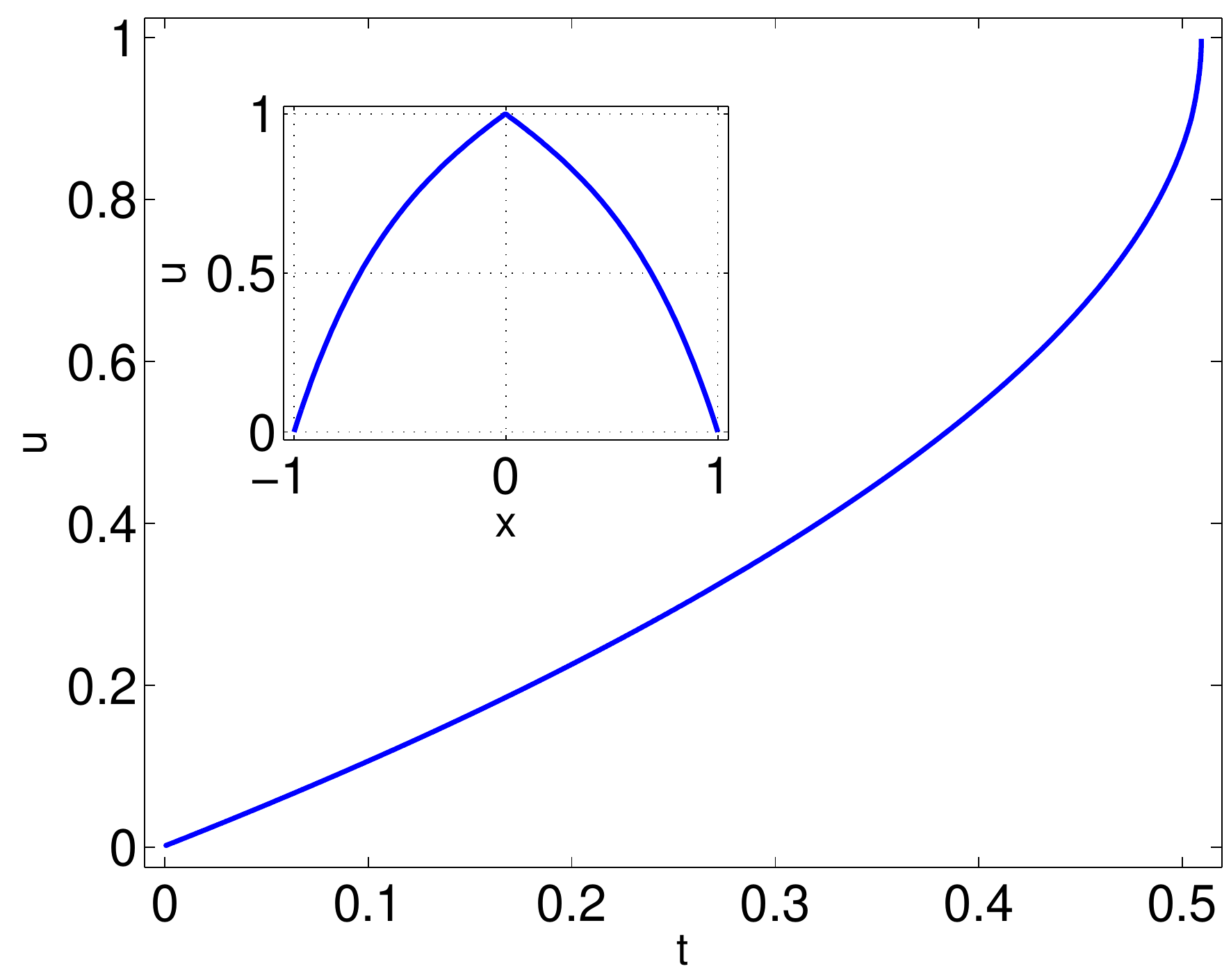,width=2.3in,height=1.28in}}~~
{\epsfig{file=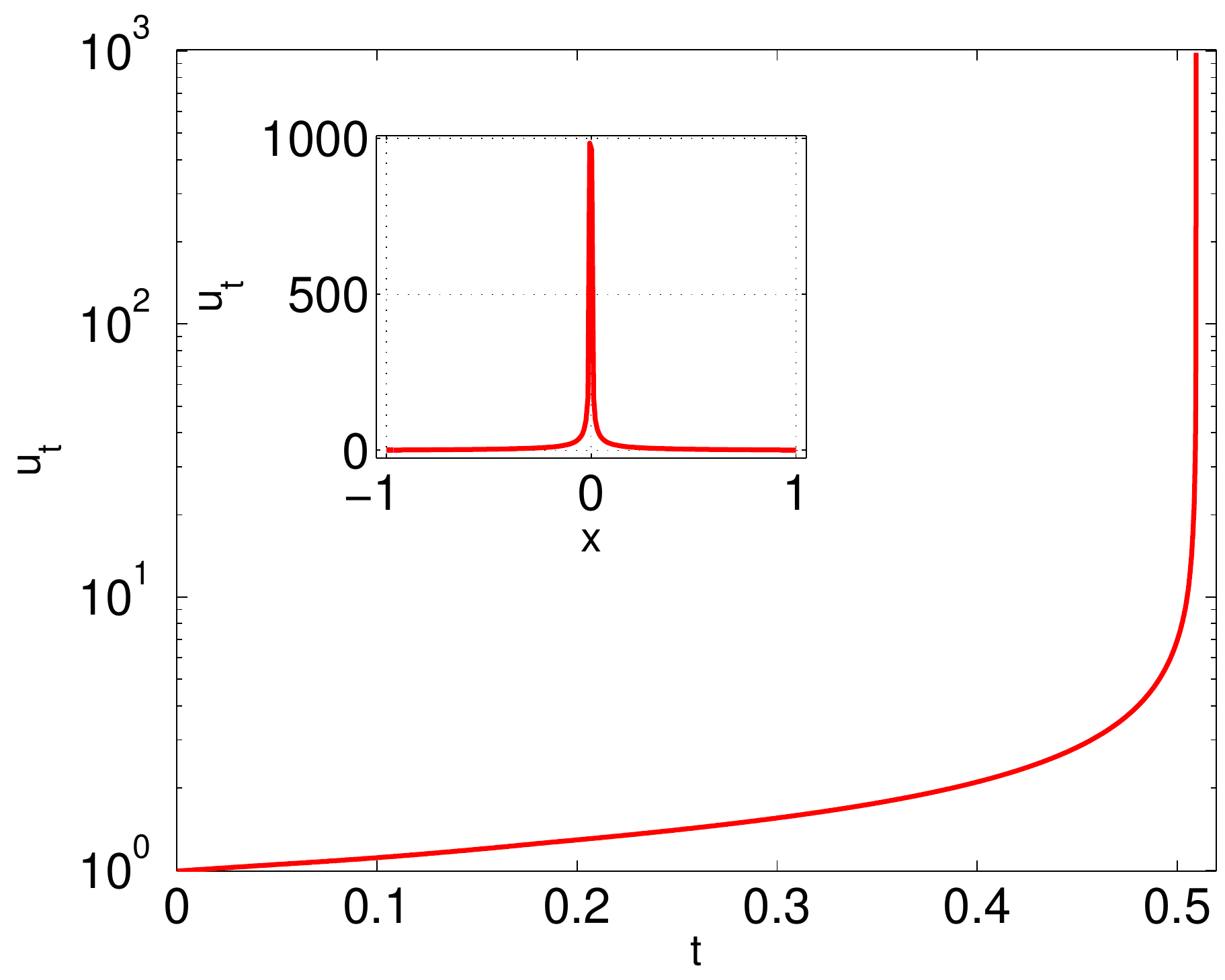,width=2.3in,height=1.28in}}

\parbox[t]{12.8cm}{\scriptsize{\bf Figure 6.} Maximal value profiles of the numerical solution $u$ [LEFT] and $u_t$ [RIGHT]. 
The mainframe curves are for $\textstyle\max_{-1\le x\le 1}u$ and $\textstyle\max_{-1\le x\le 1}u_t$ for $0\le t\le 0.509391490538887.$ 
The embedded graphics represent profiles of $u$ and $u_t$ immediately prior to quenching. The quenching 
time is observed to be $T^{*}\approx 0.509391490538887.$}
\end{center}

More details of the maximal value profiles of $u$ and $u_t$ can be found in Fig. 6. For additional information, we embed the 
terminal solution $u$ and corresponding derivative $u_t$ into the main frames. A logarithmic scale is used for $u_t$ 
in order to provide a better illustration of the explosive feature of the derivative function. The temporal adaptation is triggered 
automatically once $\textstyle\max_{-1\le x\le 1} u\approx 0.90,$ and remains 
activated throughout the remainder of computations. 

\begin{center}
{\epsfig{file=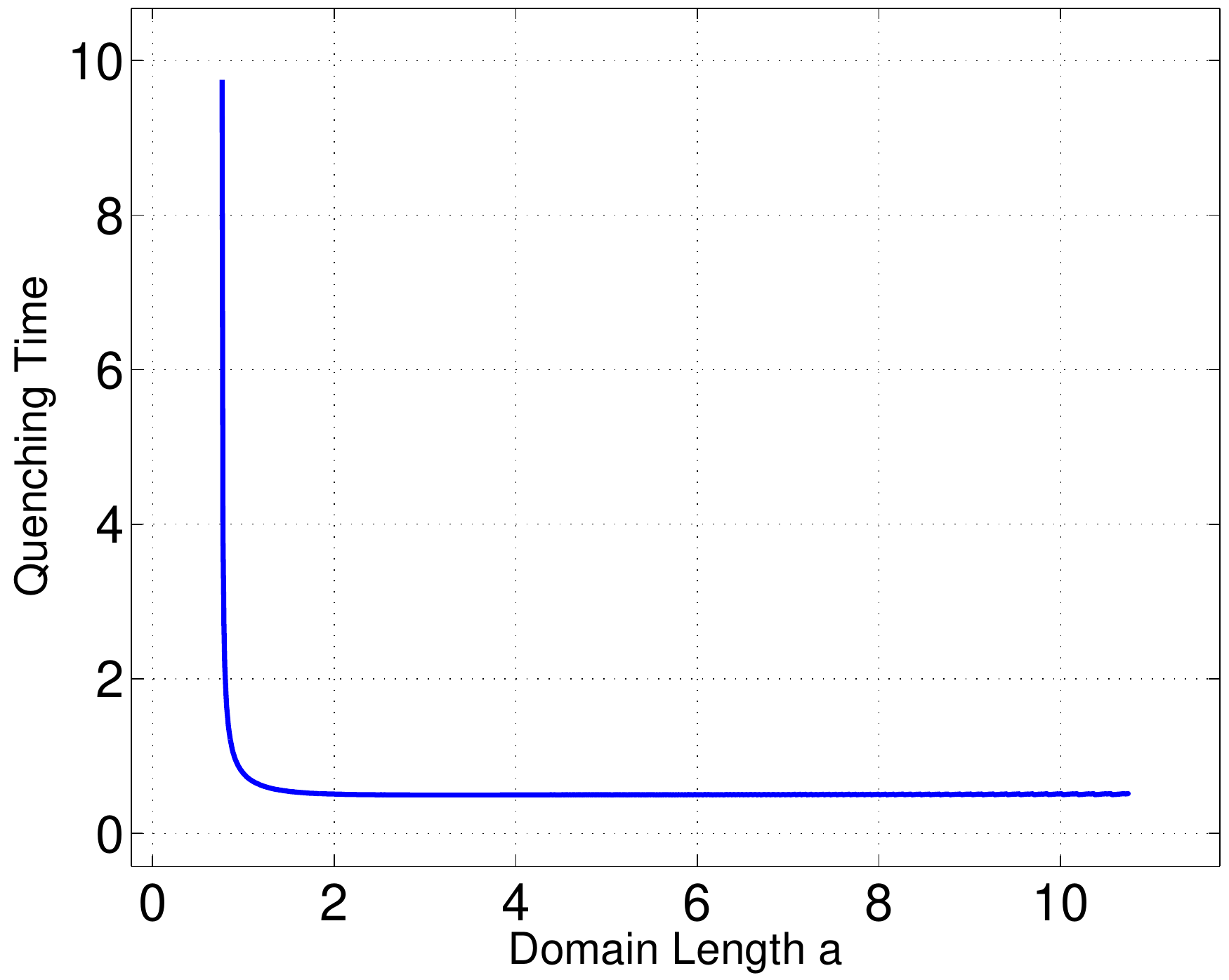,width=2.3in,height=1.28in}}~~
{\epsfig{file=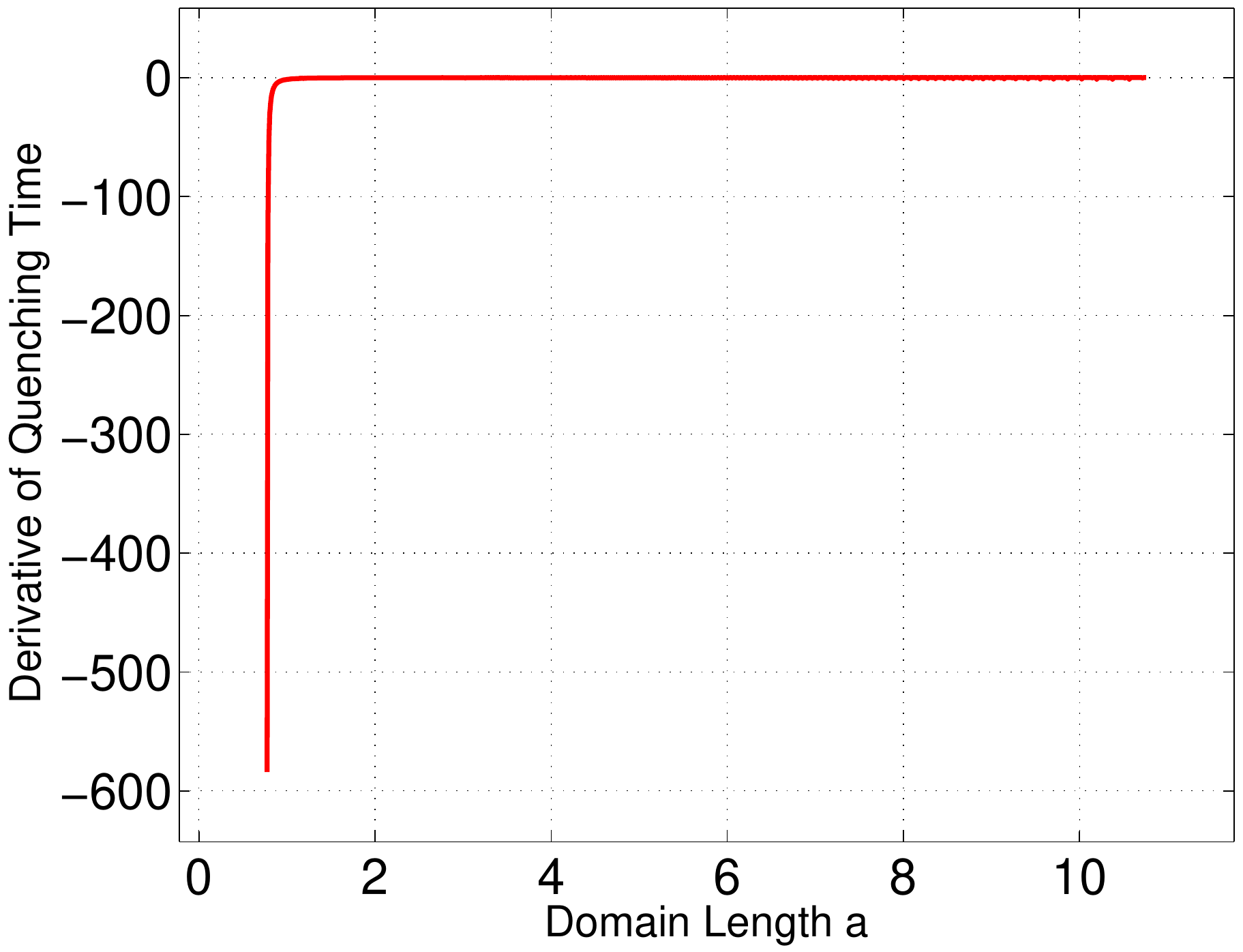,width=2.3in,height=1.28in}}

\parbox[t]{12.8cm}{\scriptsize{\bf Figure 7.} Illustrations of the dynamic connection between quenching times and domain sizes 
$a$ [LEFT], and the rate of change of the quenching time with respect to $a$ [RIGHT]. 
As $a$ increases beyond $a_1,$  quenching time declines rapidly until its minimal value 
$T^*_{a_{307}}\approx 0.499360935318447,$ where $a_{307}=3.8321581.$
Once $a$ increases beyond $a_{307}$ the quenching time increases in a slightly oscillatory manner to reach
 $T^*_{a_{1000}}.$ }
\end{center}

Fig. 7 shows the effect of domain size on quenching time. The numerical solutions based on 1000 different values of 
$a\in [0.7652281, 10.7552281]$ are computed, compared, and analyzed. The reason for choosing $a_1=0.7652281$
is that it is slightly larger than the theoretical critical size of $a^*\approx 0.765228037955310$ \cite{Chan2,Sheng4}. 
The experiments indicate that the quenching time is longer when $a$ is close to $a_1,$ with a maximum 
$T_{a_1}^*=9.752350010587456.$ The quenching time then sharply decreases. These results firmly support theoretical expectations that the quenching time should approach infinity as the domain size decreases to $a^*.$ 
We also observe that the quenching time seems to have a lower threshold. Although the data acquired exhibit slight oscillations, 
{\red the vibrations should not be caused by round-off errors, since various error-reduction measurements are utilized in 
our calculations. For example, we let the temporal step 
minimum $\tau_{\min} = c\times10^{-6},$ where $c>0$ is an usual floating point number.  When a 
temporal derivative value is evaluated, we reformulate the original formula through modifications such as
$$v_{\ell}'\approx \frac{10^6(v_{\ell+1}-v_\ell)}{10^6 \tau_{min}} = \frac{v_{\ell+1}-v_\ell}{c}\times 10^6$$
which effectively reduces the risk of unfavorable round-off errors.} Consequently, we obtain a final value of $T^*_{a_{1000}}\approx 0.515984311015508.$

\subsection*{Experiment 2}

Let us consider $a=2,~\varphi(\epsilon)\equiv 1$ and utilize the same initial function $u_0(x)=0.001(1-\cos(2\pi x)).$ 
Set our degenerate function $\sigma(x) = (x+1)^p(1-x)^{1-p},~-1\le x\le 1,~0\le p\le 1.$ 
Note that $\sigma(x)$ creates a degeneracy near each of the spatial boundaries of the problem \R{num1}-\R{num3}. 
We commence by using the golden ratio $p=(\sqrt{5}-1)/2.$ 

\begin{center}
{\epsfig{file=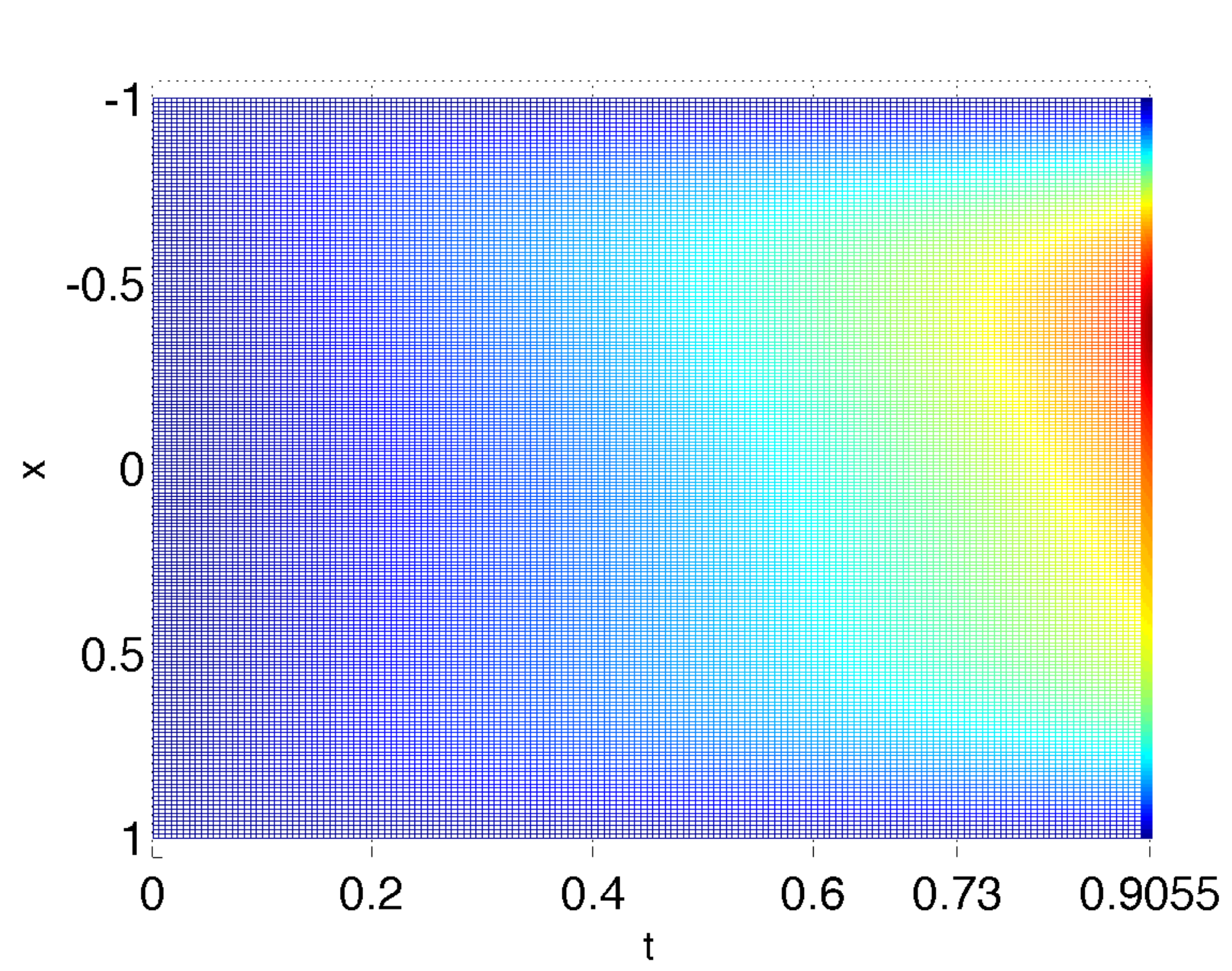,width=2.3in,height=1.28in}}~~
{\epsfig{file=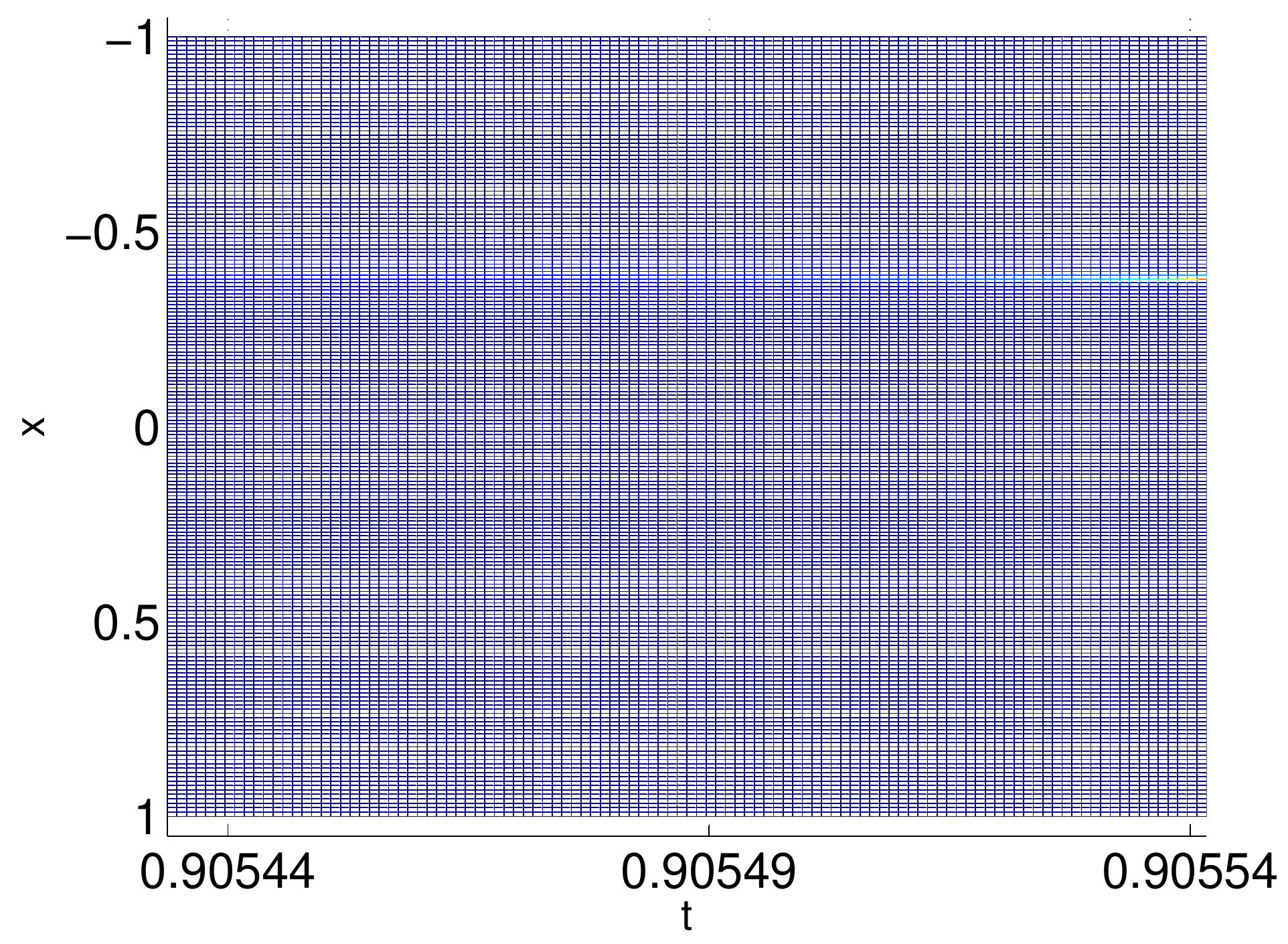,width=2.3in,height=1.18in}}

\parbox[t]{12.8cm}{\scriptsize{\bf Figure 8.} Two-dimensional thermal flows of the numerical solution $u$ for $t\in[0,T]$ [LEFT], 
and $u_t$ for $t\in [0.905433681825884, 0.905541681825887]$ [RIGHT]. In the former case, the heat flows smoothly and monotonically increases
until quenching at $P=(-0.378707538403295, 0.905541681825887).$ Numerical solutions
in last 105 temporal steps immediately before quenching are used for estimating $u_t.$}
\end{center}

\begin{center}
{\epsfig{file=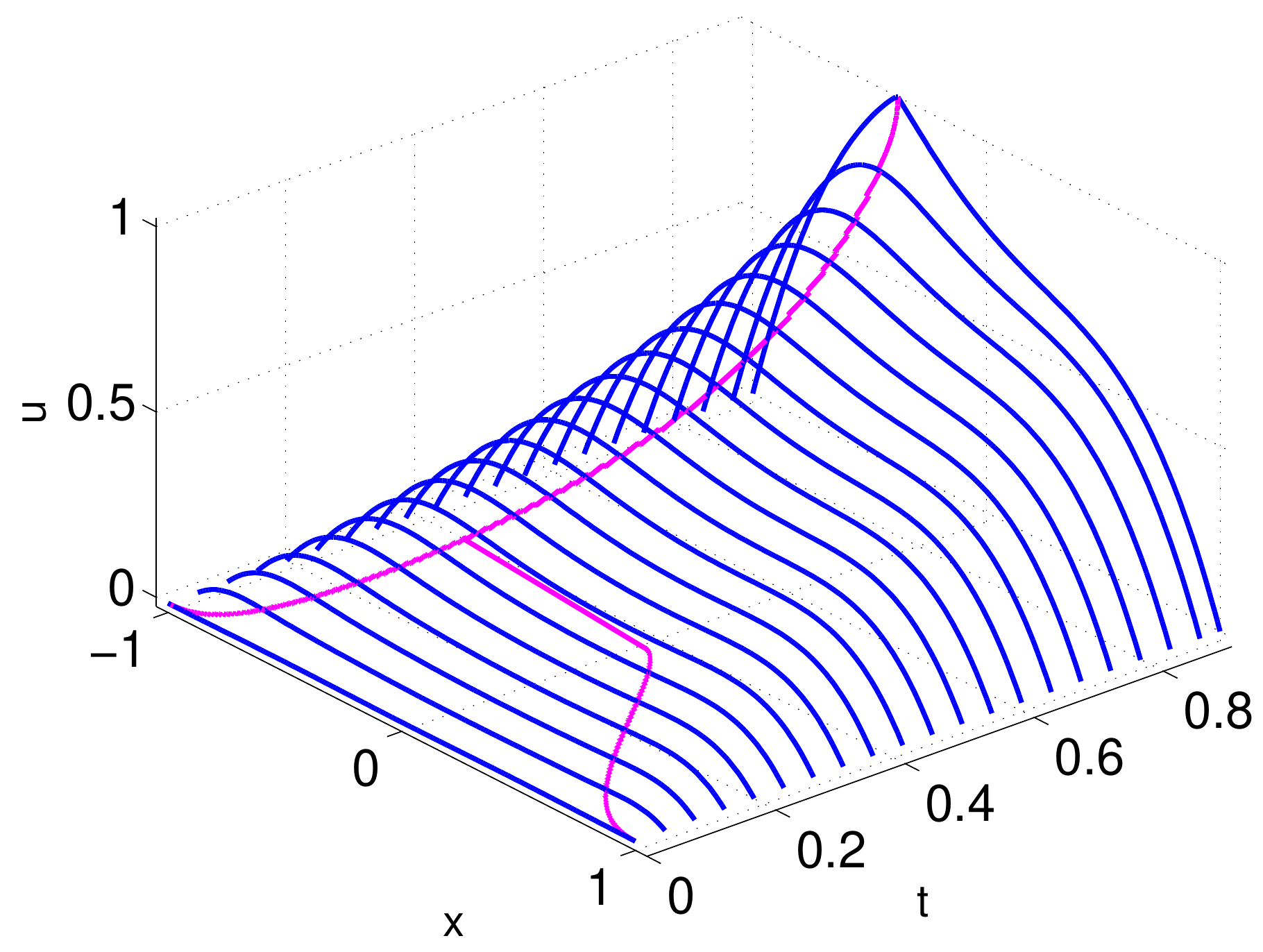,width=2.5in,height=1.38in}}~~
{\epsfig{file=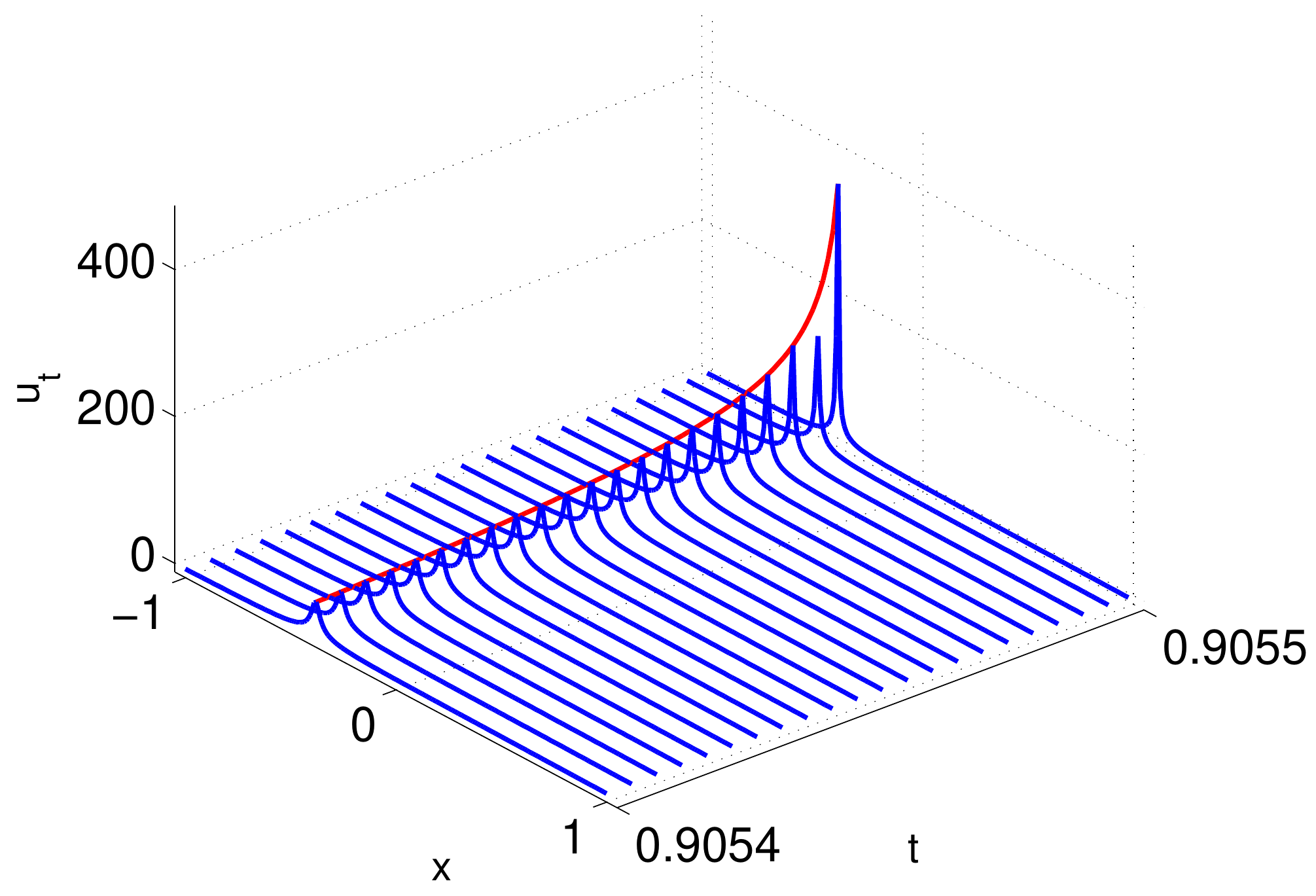,width=2.5in,height=1.38in}}

\parbox[t]{12.8cm}{\scriptsize{\bf Figure 9.} Three-dimensional views of $u$ [LEFT] and $u_t$ [RIGHT]. 
The blue three-dimensional curvature plots represent profiles of $u$ and $u_t$ at different times, respectively. While the magenta curves indicate 
$\textstyle\max_{-1\le x\le 1} u,$ the red curves are for $\textstyle\max_{-1\le x\le 1} u_t.$ }
\end{center}

We may observe in Fig. 8 and 9 how the maximal values 
have shifted away from the center due to the degeneracy. 
The velocity map of $u_t$ matches that of $u$ in the thermal plots.
To view more clearly the explosive profile of $u_t,$ a time interval 
$[0.905433681825884, 0.905541681825887]$ is used in the second frames. Note that
the temporal derivative $u_t$ reaches its maximum at the 
quenching point with $\textstyle\max_{x,t} u_t \approx 475.5902863402550.$ Further, in Fig. 9 and 10, we can observe that
there is a remarkable shift of the local maximal values to the left of the origin when approaching quenching,
as compared to cases without a degeneracy.  The temporal adaptation is activated once 
$\textstyle\max_{-1\le x\le 1} u \approx 0.90$ and remains active throughout the remainder of the computations.

\begin{center}
{\epsfig{file=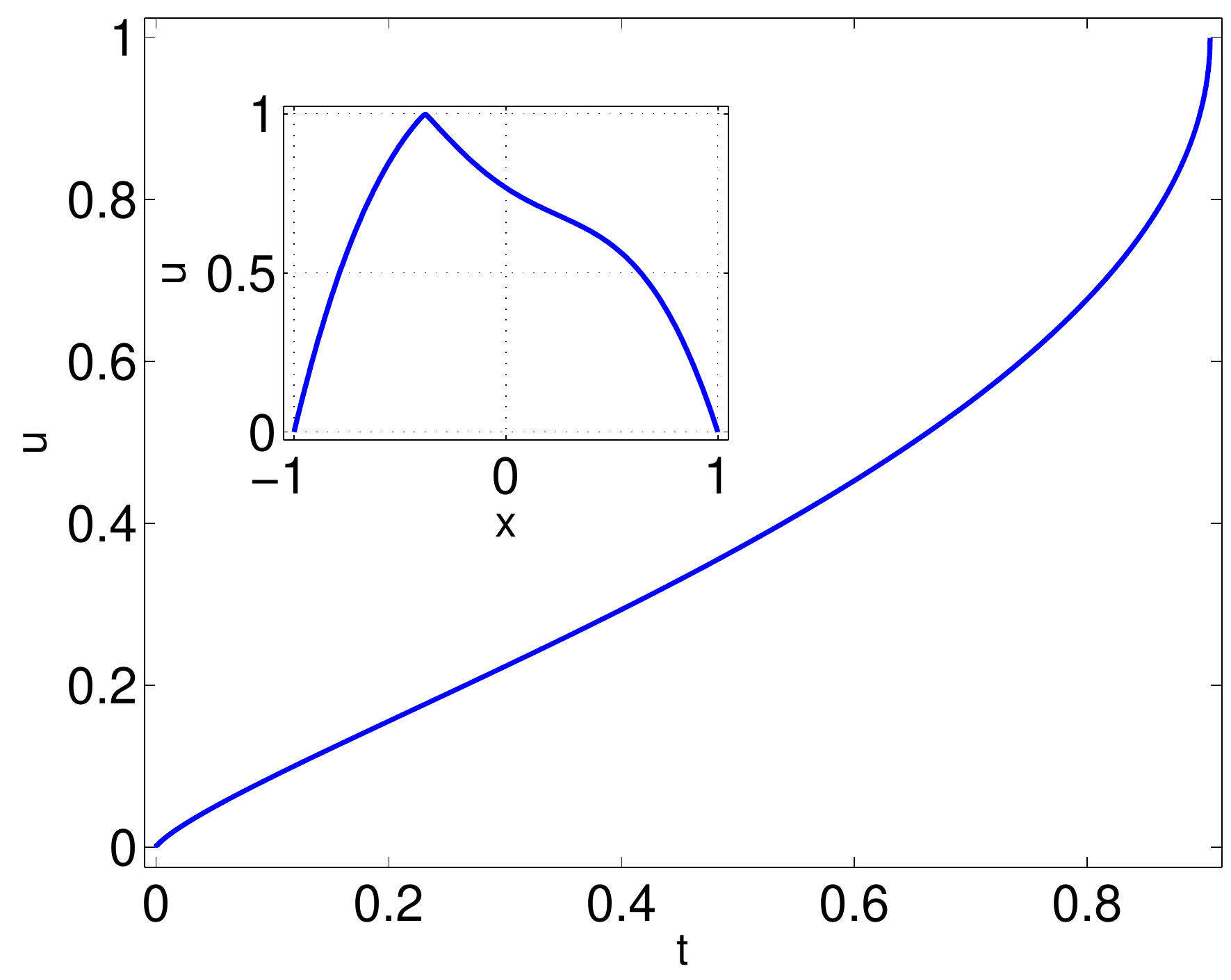,width=2.3in,height=1.28in}}~~
{\epsfig{file=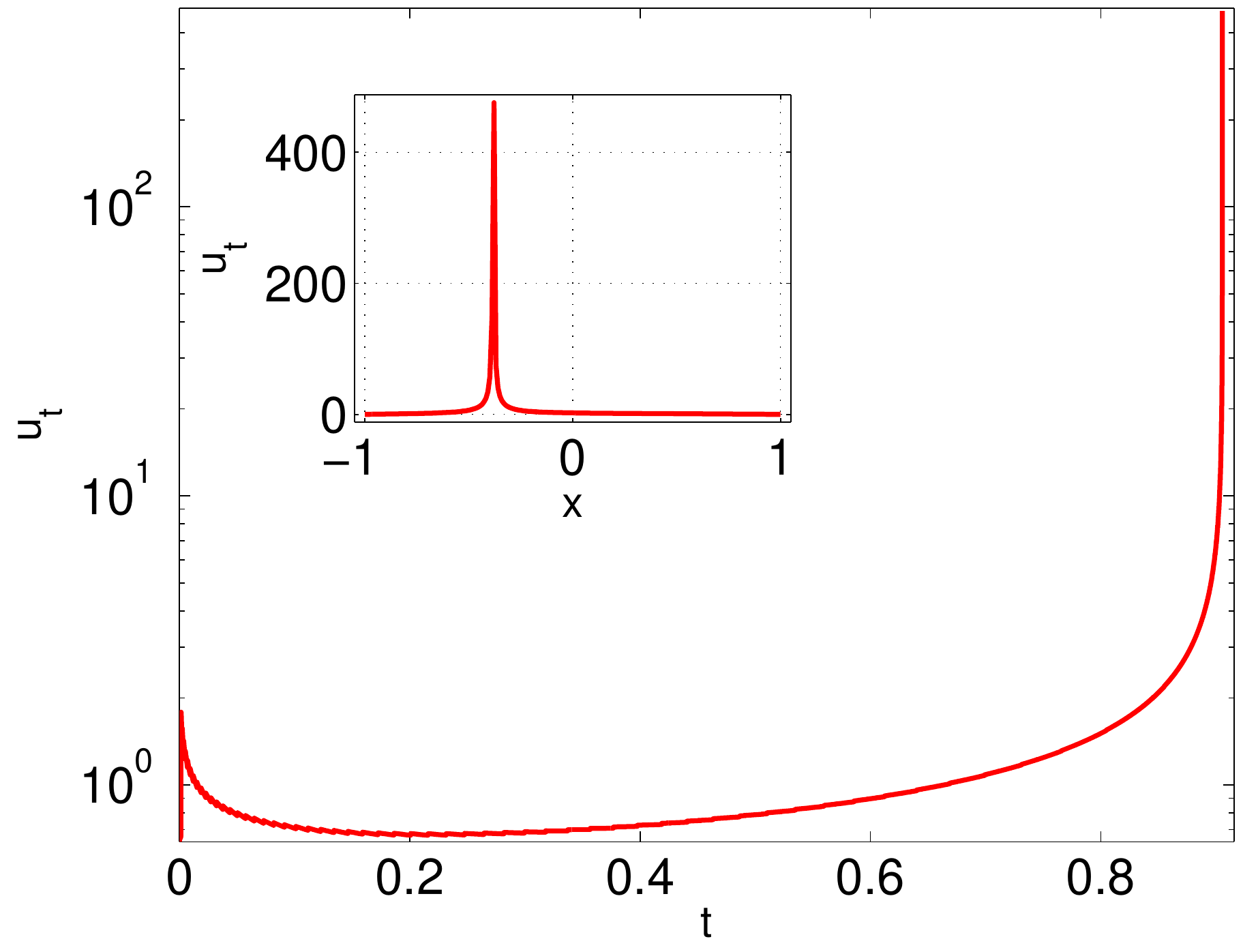,width=2.3in,height=1.28in}}

\parbox[t]{12.8cm}{\scriptsize{\bf Figure 10.} Profiles of the maximal values of the numerical solution $u$ [LEFT] and its derivative $u_t$ [RIGHT]. 
The two embedded pictures are for $u$ and $u_t$ in the last position immediately before quenching.  
The quenching location is $x^*=-0.378707538403295$ and the quenching time is $T^{*}\approx 0.905541681825887$ in the experiments which
are well-agreeable with known results \cite{Chan2,Acker2,Sheng4,Sheng3}.}
\end{center}

\begin{center}
{\epsfig{file=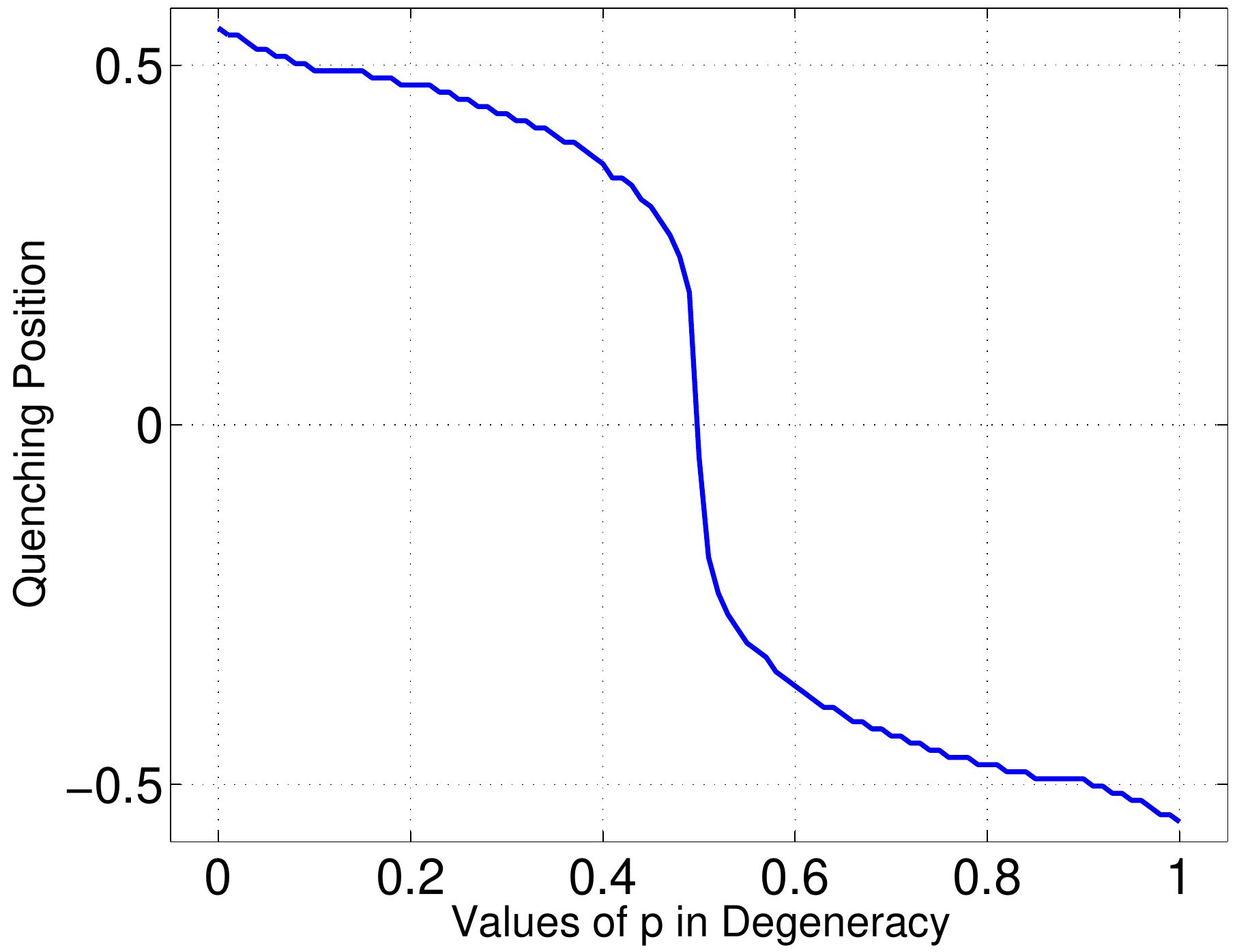,width=2.3in,height=1.28in}}~~
{\epsfig{file=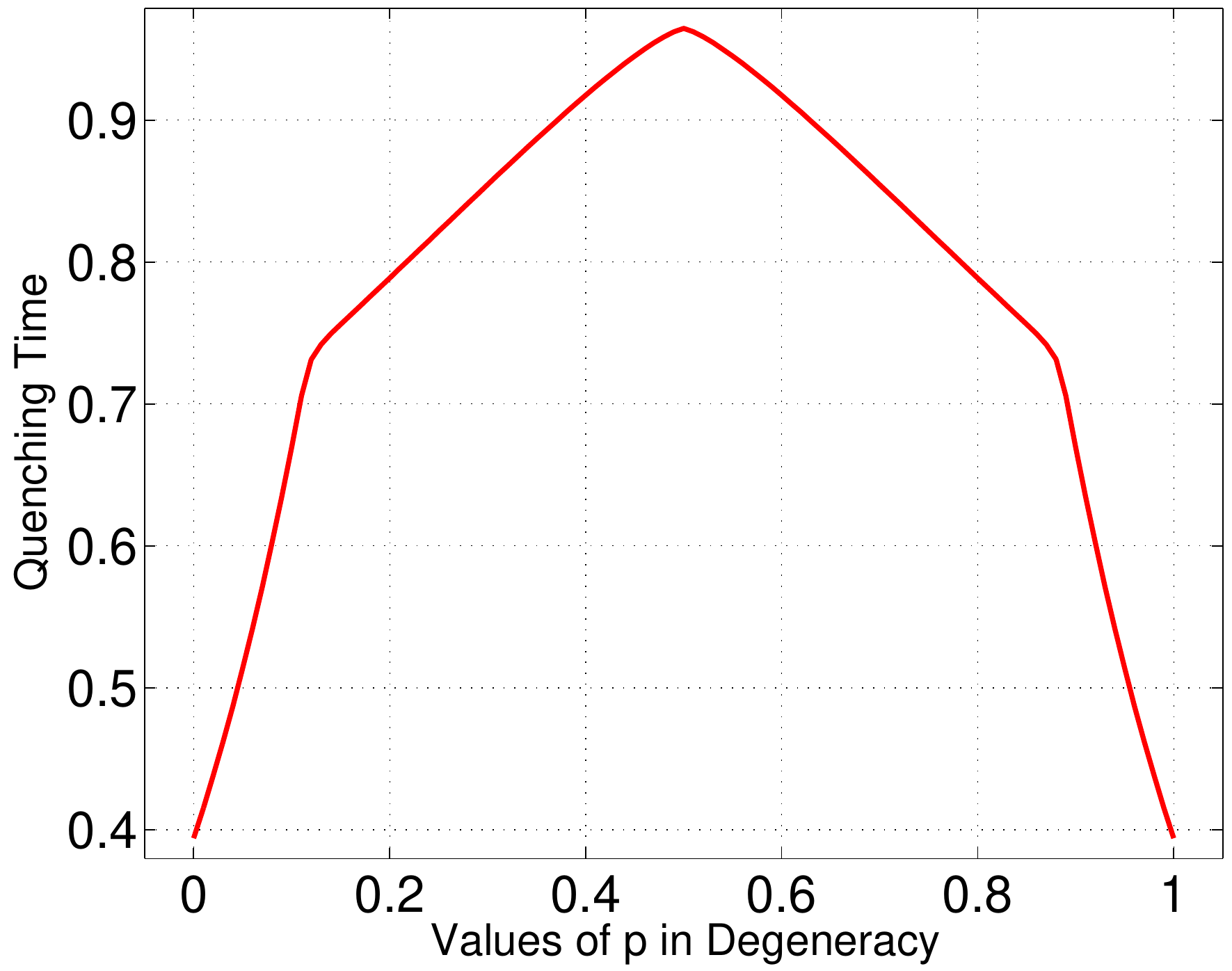,width=2.3in,height=1.28in}}

\parbox[t]{12.8cm}{\scriptsize{\bf Figure 11.} The left figure displays the relationship between quenching position and values of $p$ in the degeneracy function 
$\sigma(x).$ The right picture, on the other hand, shows the connection between the quenching time and values of $p$ in $\sigma(x).$ It is known
that the quenching position and quenching time distribution should be symmetric about the value $p=0.5$
\cite{Bebernes_89,Beau1,Sheng3}. }
\end{center}

Fig. 11 is designed to show possible relations between the quenching location and $p,$ as well as the connections 
between the quenching time and
$p.$ Values of $p$ are specified through the given degeneracy function $\sigma(x).$ 

It is found in our numerical experiments that the quenching position decays monotonically as $p$ increases. Extreme values, $x^*=
-0.552238805970151,~0.552238805970147,$ are taken as $p=0,~1,$ respectively. As expected, the quenching location is at
the origin when $p=1/2.$ However, the decay of quenching location function is apparently nonlinear and exhibits a pattern of antisymmetry. 
It seems that the impact of $p$ on quenching locations is relatively more significant for $p\in[0.4, 0.6].$

Further, for the range of $p$ values used, we may observe that the minimal quenching time, 
$T^*\approx 0.394063444318618,$ occurs as $p$ approaches either $0^+$ or $1^-.$ On the other hand, the maximum quenching time, 
$T^*\approx 0.964575637131343$ can be witnessed at $p=0.5.$ The quenching time function is symmetric about $p=0.5$ but again
nonlinear. The impact of $p$ on the quenching time is relatively more pronounced when $p$ is closer to the end of its defined interval, that is, when
$p\in [0,0.12]$ or $p\in [0.88, 1].$

\subsection*{Experiment 3}

We proceed with $a=2$ and $\sigma(x)\equiv 1$ in this particular exploration. In order to study the effects of the stochastic influence we consider 
the function $\varphi(\epsilon)=\epsilon^2,~0.01\le\epsilon\le 1.$ In this example we explore the numerical solutions resulting from 
two different white noise vectors $\epsilon(x).$ The purpose of these considerations is to investigate how slight changes in the 
vector $\epsilon(x)$ can result in drastically different solutions profiles.

\begin{center}
{\epsfig{file=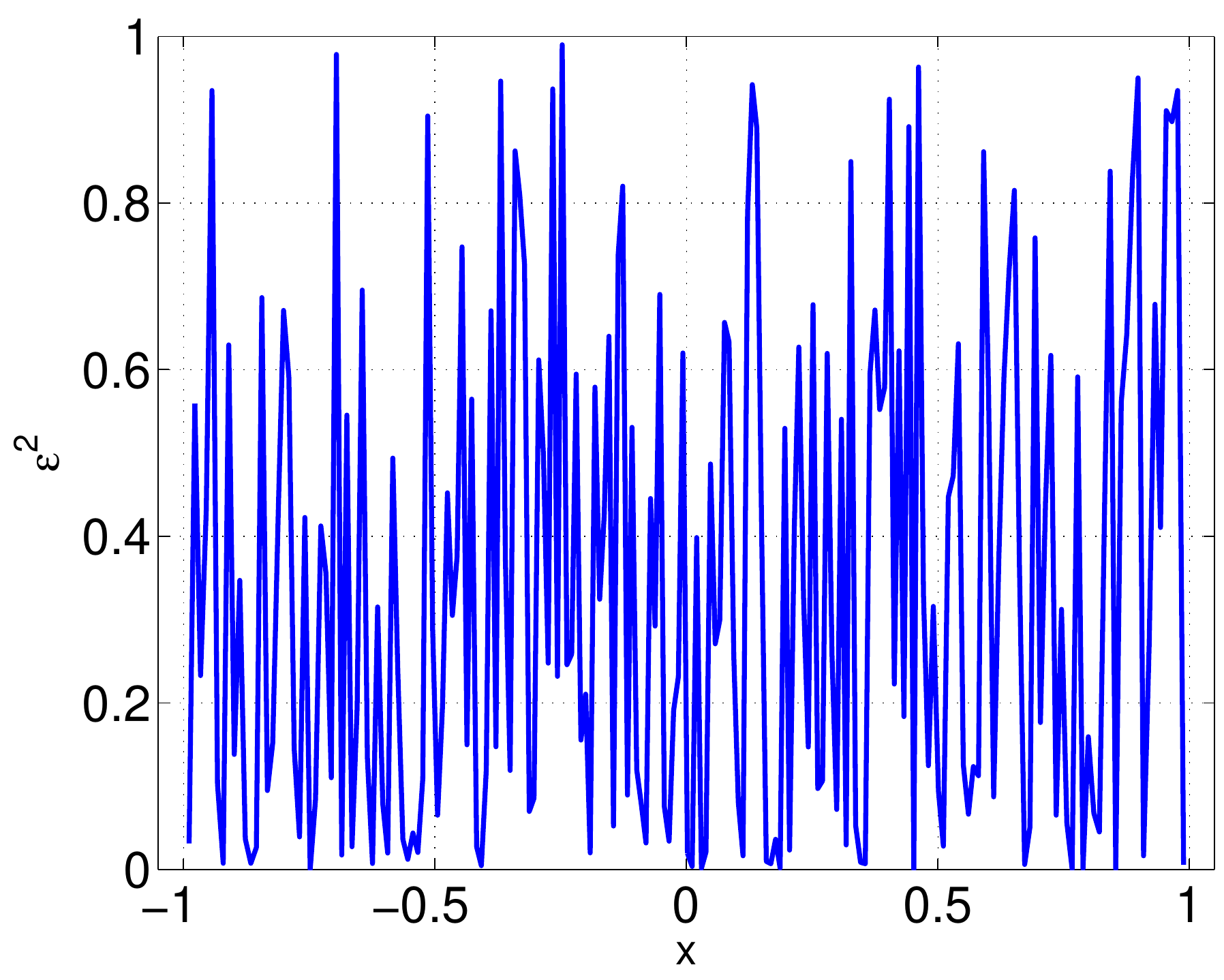,width=2.3in,height=1.28in}}
~~
{\epsfig{file=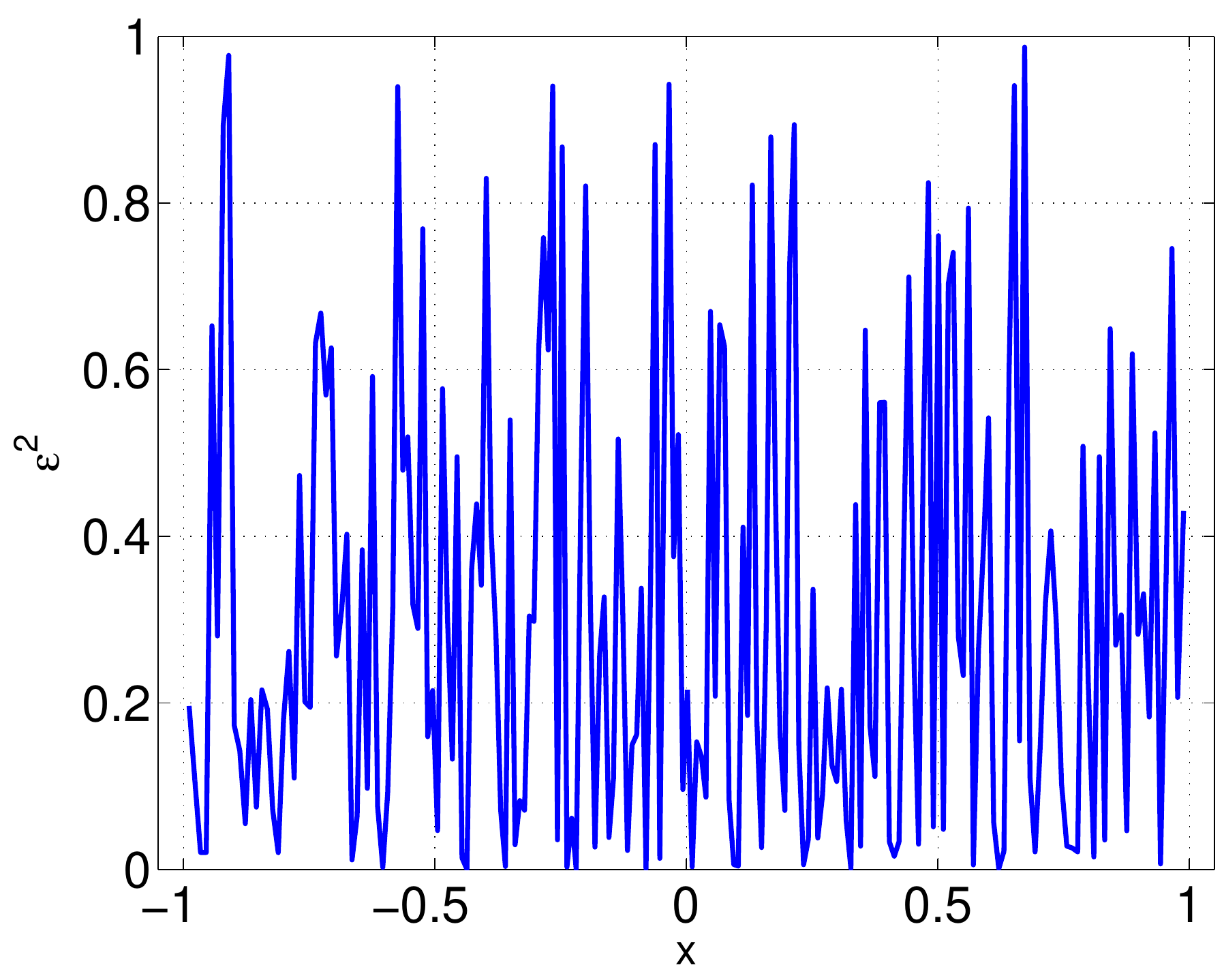,width=2.3in,height=1.28in}}
~~
\parbox[t]{12.8cm}{\scriptsize{\bf Figure 12.} Plots of two different stochastic function values corresponding to 
$\varphi(\epsilon)=\epsilon^2$ are shown. Herewith we have $0.01\le\epsilon\le 1.$ }
\end{center}

\begin{center}
{\epsfig{file=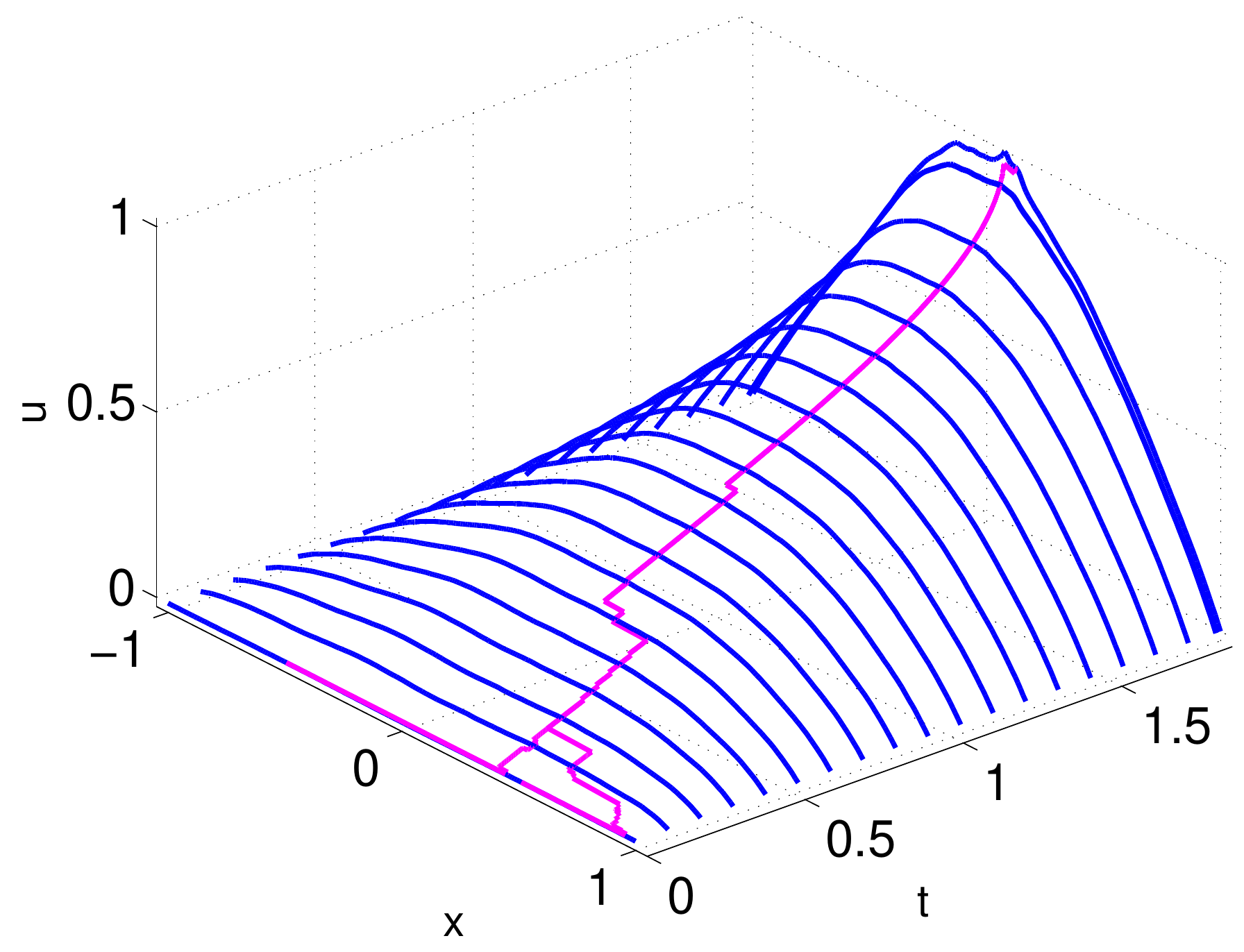,width=2.5in,height=1.38in}}
~~
{\epsfig{file=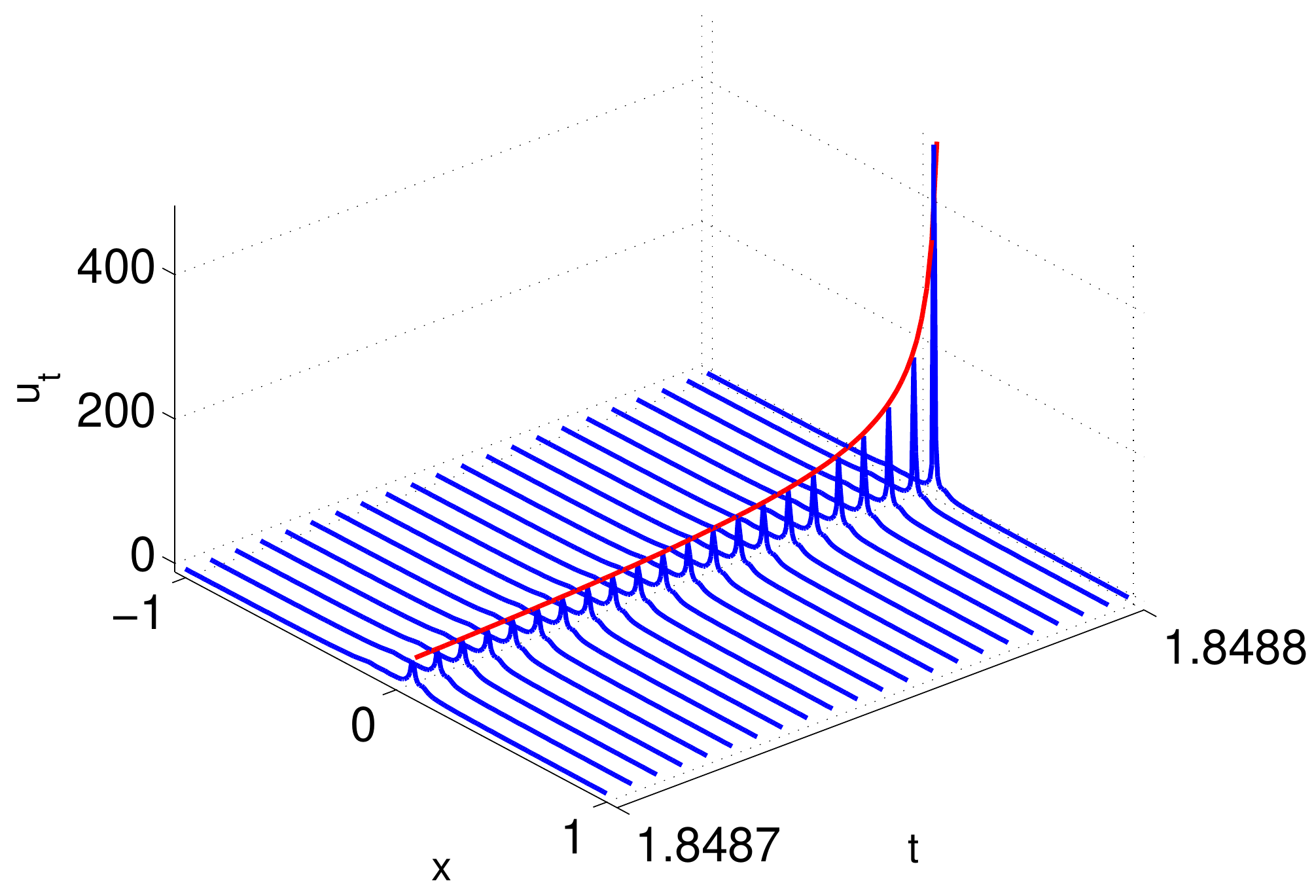,width=2.5in,height=1.38in}}
~~
\parbox[t]{12.8cm}{\scriptsize{\bf Figure 13.} Three-dimensional curvature views of $u$ [LEFT] and $u_t$ [RIGHT] 
corresponding to the first set of random values. The blue lines represent profiles of $u$ 
and $u_t$ at different times, while the magenta and red lines represent $\textstyle\max_{-1\le x\le 1} u$ and 
$\textstyle\max_{-1\le x\le 1} u_t,$ respectively. The temporal derivative has its largest values concentrated about the 
quenching point with $\max u_t \approx 484.1416720746672.$}
\end{center}

The two different random variable function outputs used are displayed in Fig. 12. Note that the 
vector $\epsilon(x)=\l(\epsilon_1,\dots,\epsilon_N\r)^{\tT}$ is generated randomly, with each component 
consisting of a uniformly distributed random number $\epsilon_i\in[0.01,1],~i=1,\dots,N.$  
Three-dimensional curvature views of the numerical solution $u$ for $t\in[0,T_1]$ and 
and $u_t$ for $t\in [T^0_1,T_1]$ corresponding the left figure are given in Fig. 13. Values of $T^0_1= 1.848738391680962$ and 
$T_1= 1.848843391680954$ are used in the second frame. Again, the heat monotonically increases from left to right
until quenching at $P=(0.020653228859545, 1.848843391680954).$ On the other hand, the velocity $u_t$ 
maximum trajectory matches that of $u$ and 
explodes at the aforementioned quenching position $P$ in an extremely short period of time. Numerical solutions
in last 105 temporal steps immediately before quenching are used.

\begin{center}
{\epsfig{file=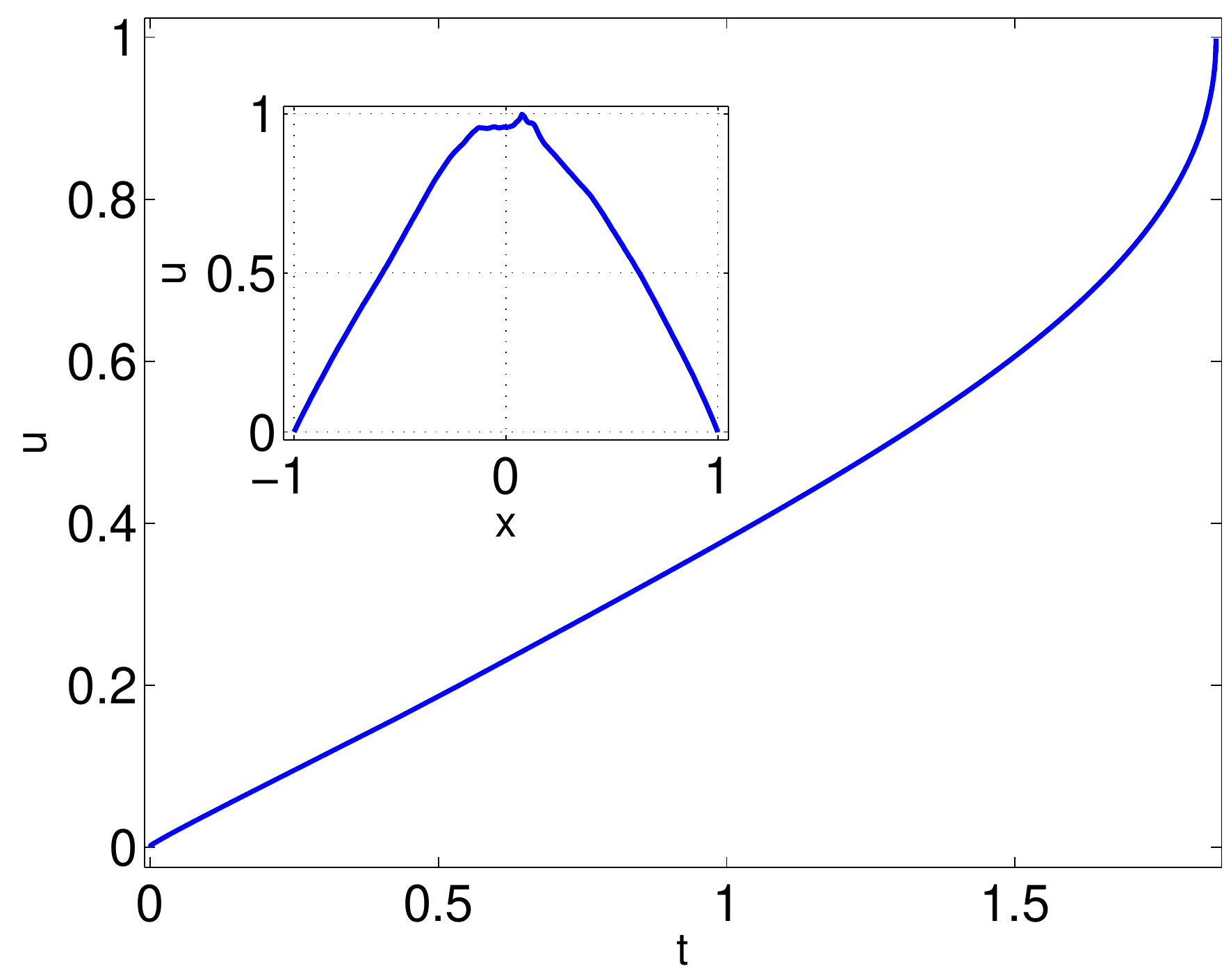,width=2.3in,height=1.28in}}
~~
{\epsfig{file=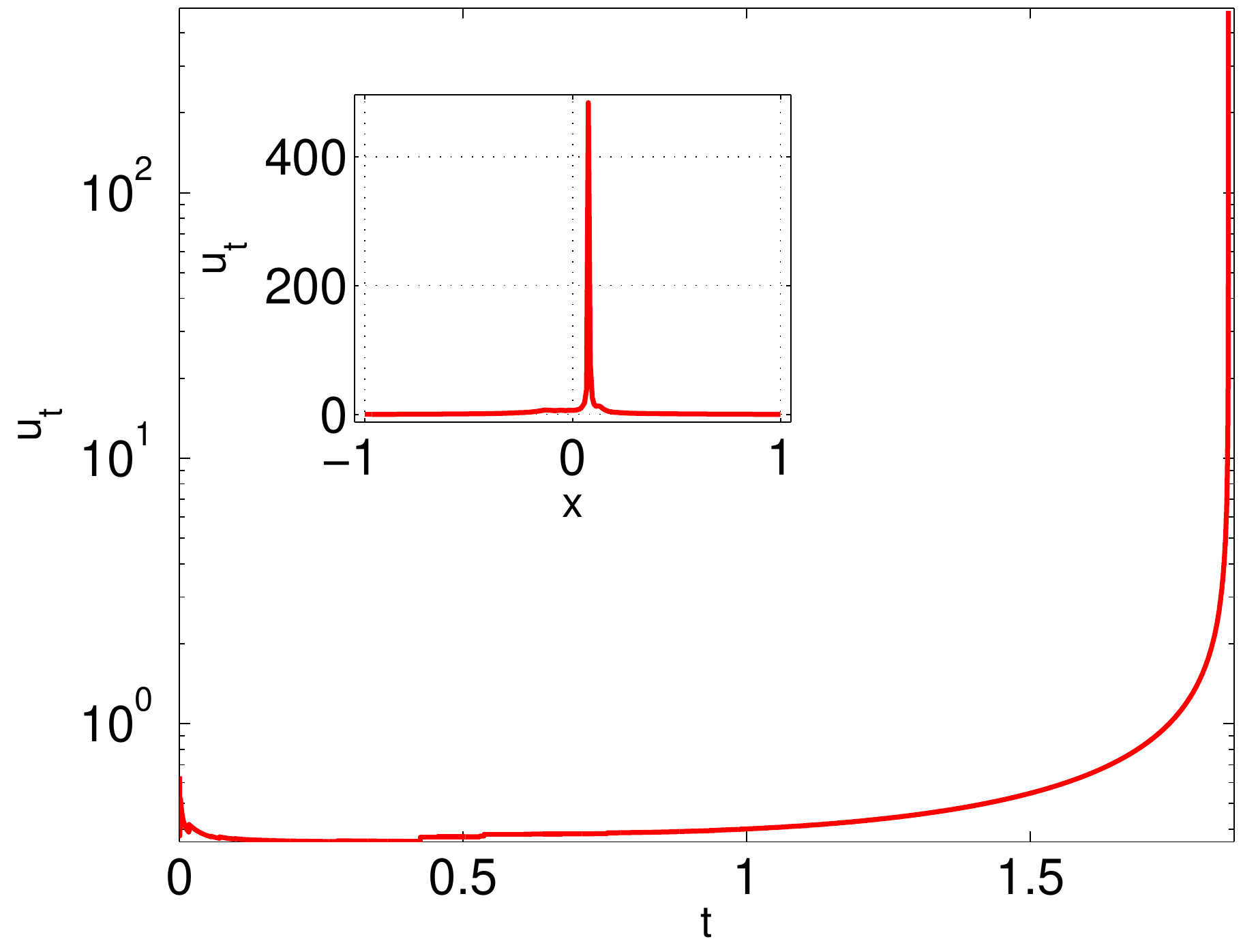,width=2.3in,height=1.28in}}

\parbox[t]{12.8cm}{\scriptsize{\bf Figure 14.} Profiles of the maximal values of the numerical solution $u$ [LEFT] and its derivative $u_t$ [RIGHT]. 
The two embedded pictures are for $u$ and $u_t$ in the last position immediately before 
quenching. }
\end{center}

Note from Fig. 13 and 14 how the location of quenching has shifted, just as in the degeneracy case, but this solution is 
slightly to the right of the origin. The most notable feature is the lack of an initial smoothness in the solution profile of $u.$
The temporal adaptation is trigged once $\textstyle\max_{-1\le x\le 1} u \approx 0.90$ and remains active throughout 
the remainder of computations. The quenching is found at $x^*=0.066585749194076$ and the quenching time is 
$T^{*}_1\approx 1.848843391680954.$

\begin{center}
{\epsfig{file=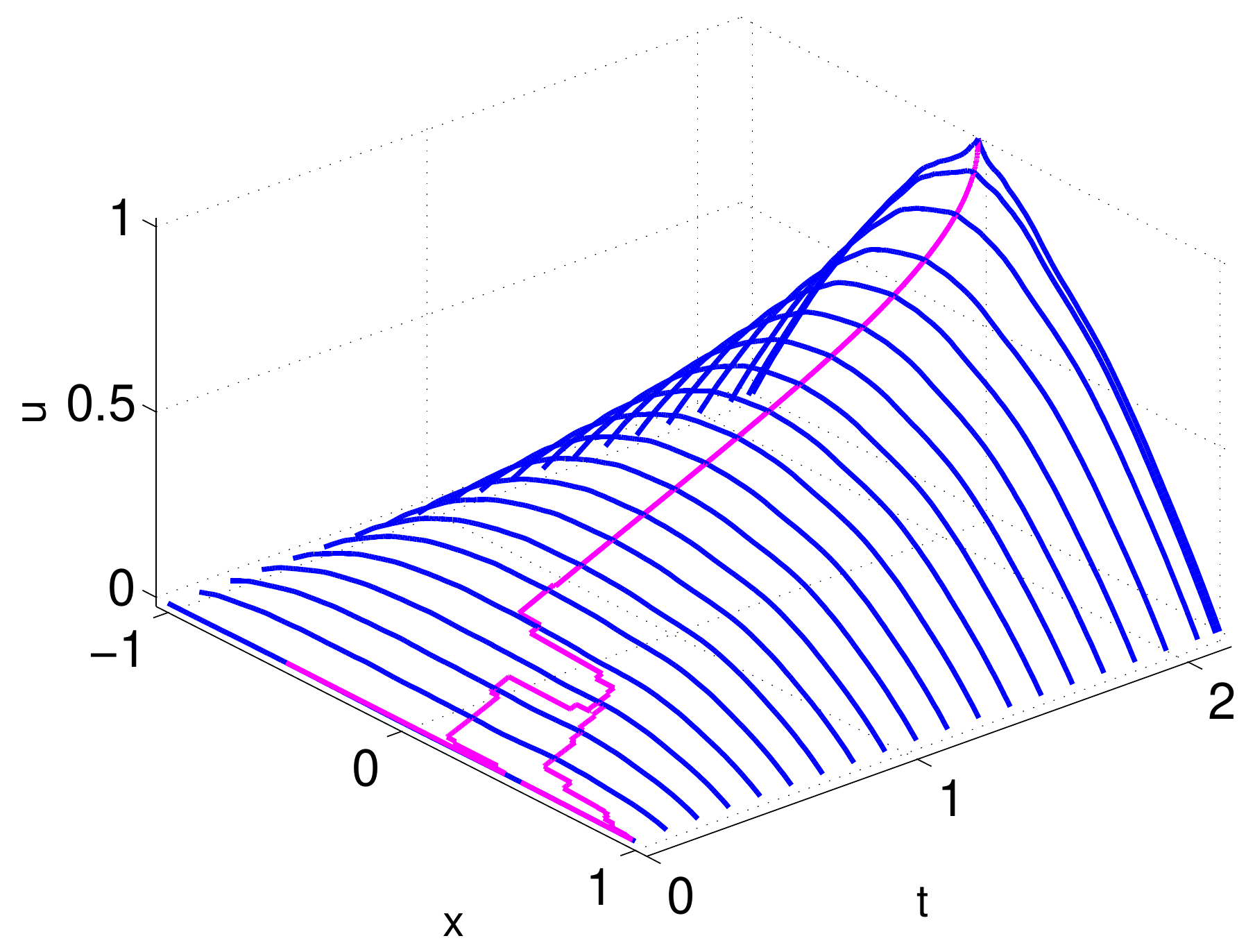,width=2.5in,height=1.38in}}
~~
{\epsfig{file=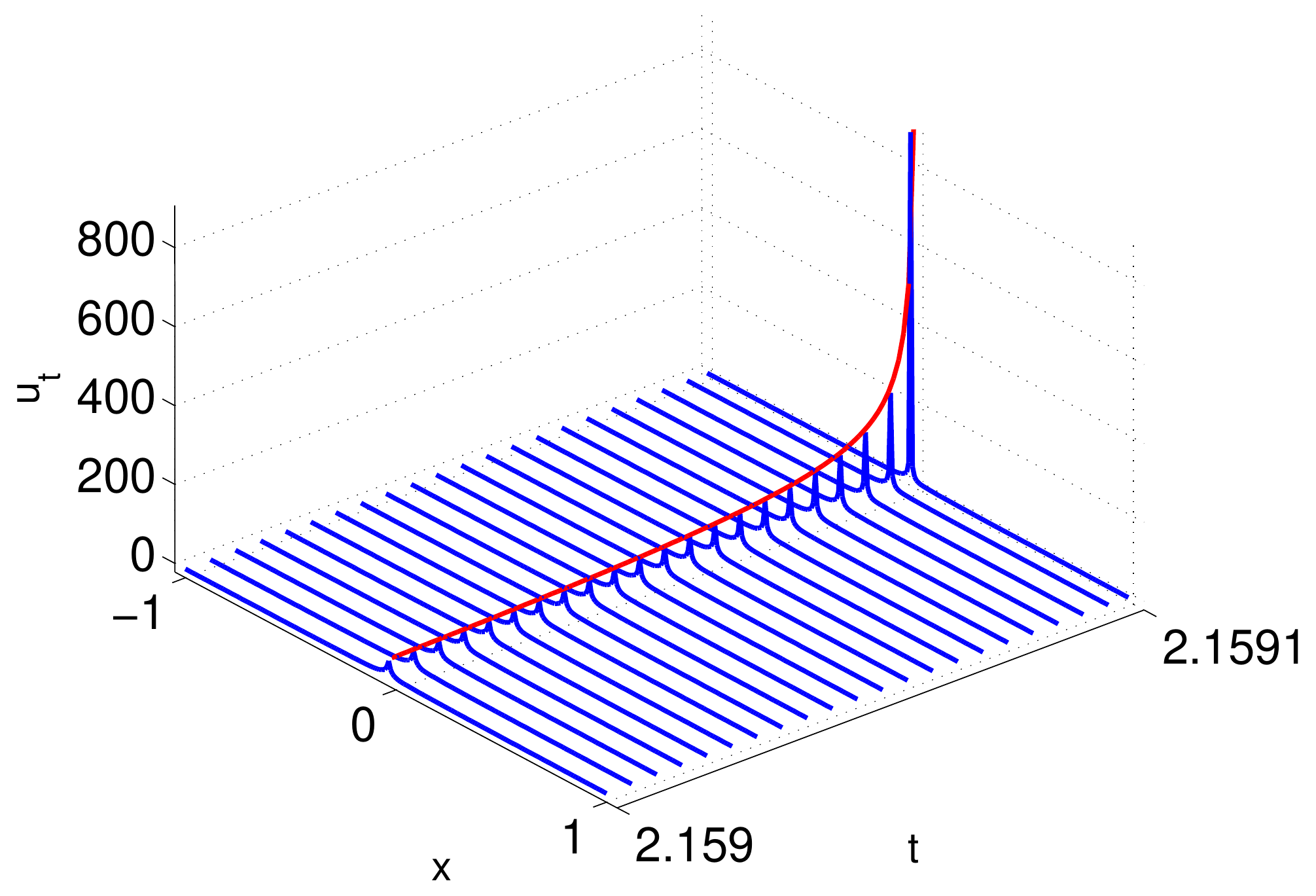,width=2.5in,height=1.38in}}
~~
\parbox[t]{12.8cm}{\scriptsize{\bf Figure 15.} Three-dimensional curvature views of $u$ [LEFT] and $u_t$ [RIGHT]. 
The blue lines represent profiles of $u$ and $u_t$ at different times, while the magenta and red lines 
represent $\textstyle\max_{-1\le x\le 1} u$ and $\textstyle\max_{-1\le x\le 1} u_t,$ respectively. The temporal 
derivative has its largest values concentrated about the quenching point with $\max u_t \approx 885.0797753481299.$}
\end{center}

\begin{center}
{\epsfig{file=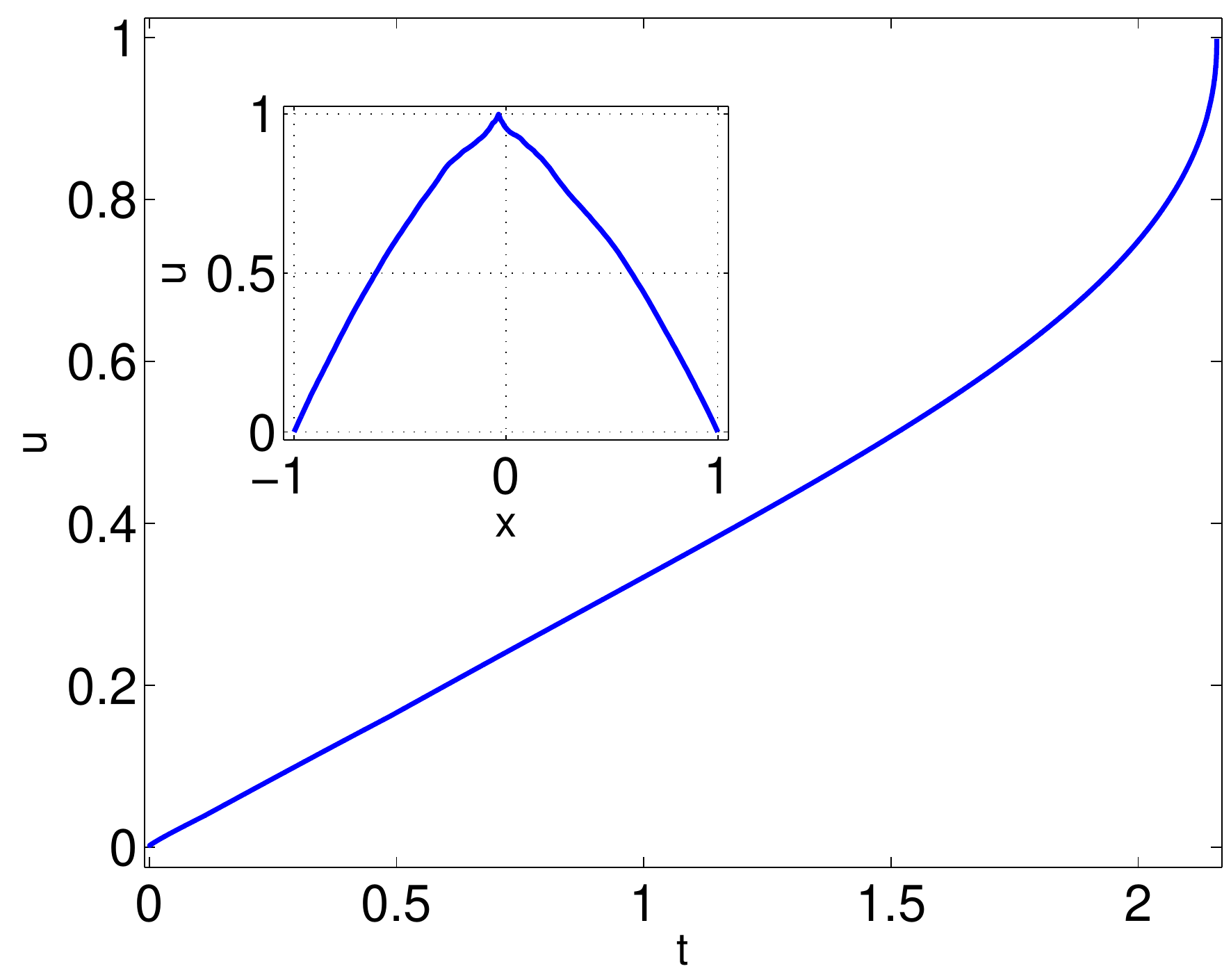,width=2.3in,height=1.28in}}
~~
{\epsfig{file=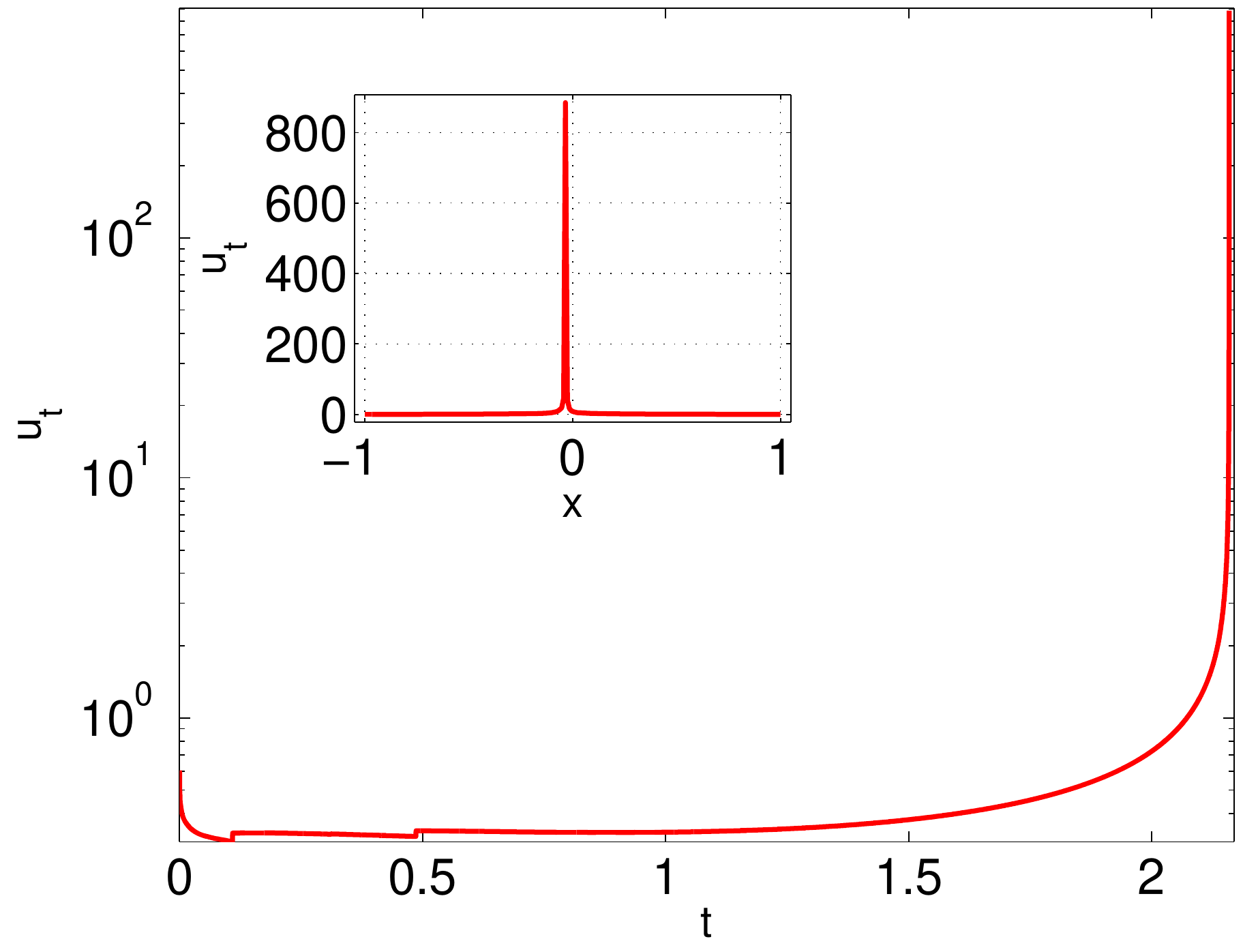,width=2.3in,height=1.28in}}

\parbox[t]{12.8cm}{\scriptsize{\bf Figure 16.} Profiles of the maximal values of the numerical solution $u$ [LEFT] and its derivative $u_t$ [RIGHT]. 
The two embedded pictures are for $u$ and $u_t$ in the last position immediately before quenching. 
The quenching location is $x^*=-0.043597691787030$ and the quenching time is $T^{*}_2\approx 2.159108137916605.$}
\end{center}

Fig. 15 and 16 are for the case when the second random output in Fig. 12 is selected. $T^0_2= 2.159003137916590$ and 
$T_2=2.159108137916605$ are utilized for $u_t.$ The heat $u$ flows smoothly and increases monotonically
until quenching at $P=(0.020653228859545,2.159108137916605).$ On the other hand, the trajectory of the maximal velocity of $u_t$ 
matches that of $u.$ It explodes at the aforementioned quenching position $P$ to the peak value. Numerical solutions
in last 105 temporal steps immediately before quenching are used for $u_t.$ It can also be seen 
that the stochastic term causes the maximum value of the solution to shift drastically. The temporal adaptation is activated 
once $\textstyle\max_{-1\le x\le 1} u \approx 0.90$ and remains active throughout the remainder of the computations. 

\begin{center}
{\epsfig{file=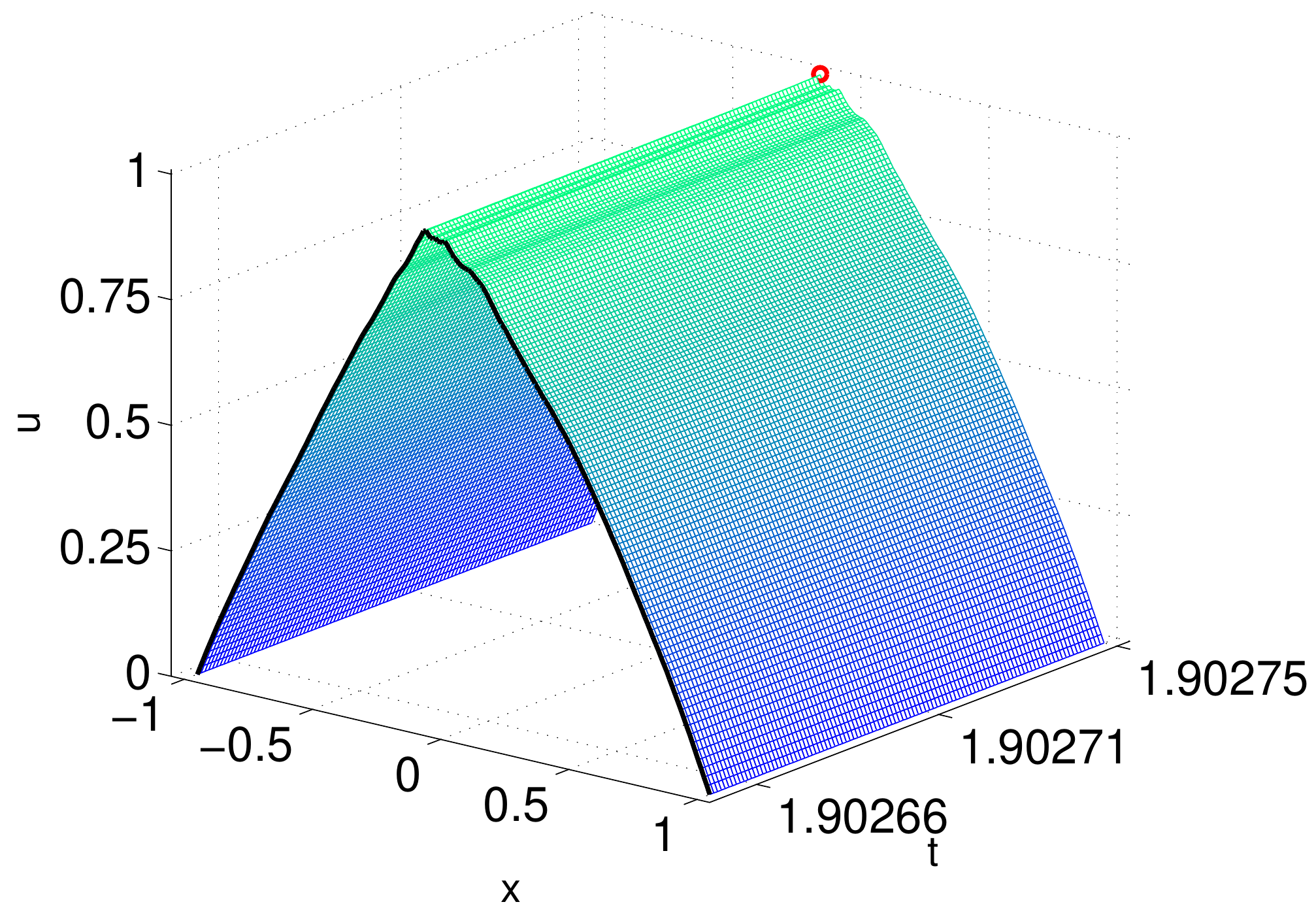,width=2.63in,height=1.6in}}
{\epsfig{file=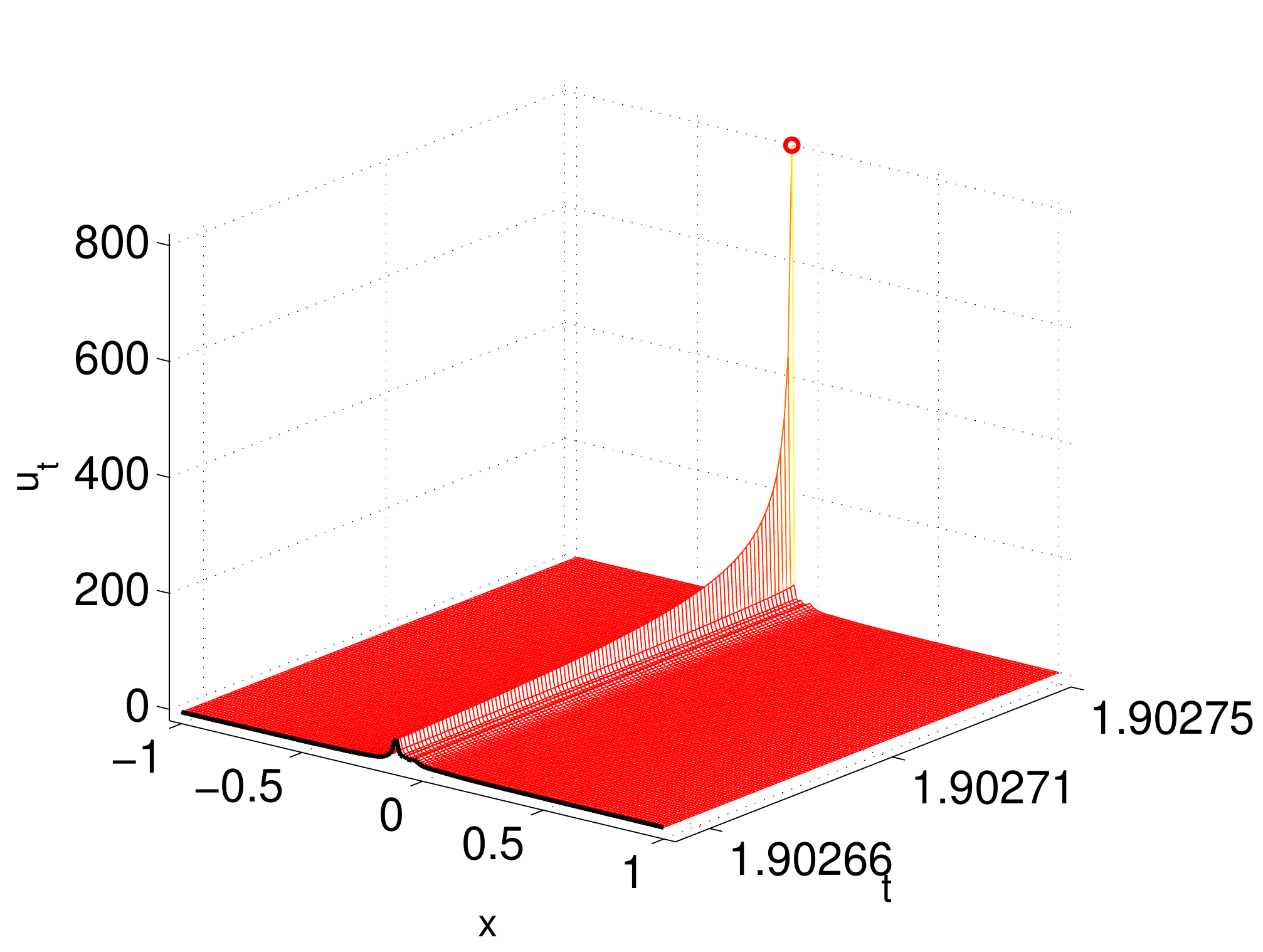,width=2.63in,height=1.6in}}
~~
\parbox[t]{12.8cm}{\scriptsize{\bf Figure 17.} More detailed three-dimensional surface views of $u$ [LEFT] 
and $u_t$ [RIGHT]. Both surfaces increase monotonically with respect to the time. 
The function $u$ and its temporal derivative have their largest values concentrated about the 
quenching point with $\max u\approx 0.998883860103612$ and $\max u_t \approx 799.5162246152710.$
Quenching time $T^*\approx 1.902752501571197.$}
\end{center}

To conclude this experiment, we rerun the second case for surface plots of the last 105 $t$-level numerical solutions 
before quenching in Fig. 17. Note that 
in all sets of figures the quenching time is considerably longer. This is due to the fact that the constraints 
on $\epsilon$ dampen the effects of the source term, thereby increasing the amount of time for a solution component to reach 
unity. Further note how small differences in the random function $\varphi(\epsilon)$ can shift the quenching location to either 
side of the origin. Further, it can be observed in Fig. 13-17 that the solution $u$ becomes quite non-smooth near the 
quenching point. We may also note how differently the maximum values of the solution 
$u$ flows with respect to time in Fig. 13, 15 and 17. While the maximal values converge to the appropriate spatial position fairly 
quickly, their paths to that point are extremely different. 

{\red Since the Lax equivalence theorem does not apply to nonlinear schemes such as \R{c3}, rigorous tests on the nonlinear
convergence are in general important. But our experiments are carried out in cases when nonlinear source terms are linearized, 
similar to those in \cite{Sheng4,Beau1,Beau11}. This linearization allows for the problem to be locally considered linear, and in 
this sense, convergence can be ensured through the numerical stability. Note that such an argument is only valid in the local 
sense, a global convergence analysis of the numerical method \R{c3}, or \R{approx}, must be conducted via chemical-physical 
energy conservations. Such an endeavor has been in our agenda and will be discussed carefully in our forthcoming papers.}

\subsection*{Experiment 4}

{\red We now consider the following two-dimensional problem:
\bbb
&&\sigma(x,y)u_t = \frac{1}{a^2}u_{xx}+\frac{1}{b^2}u_{yy}+\frac{\varphi(\varepsilon)}{(1-u)^\theta},\quad -1<x,y<1,~t_0<t\le T,\label{num11}\\
&&u(-1,y,t)=u(1,y,t)=u(x,-1,t)=u(x,1,t)=0,\quad t>t_0,\label{num22}\\
&&u(x,y,t_0)=u_0(x,y),\quad -1<x,y<1,\label{num33}
\eee
where $T<\infty,~\theta>0,~u_0(x,y)\in C[[-1,1]\times[-1,1]],$ and $0\leq u_0\ll 1.$ Without loss of generality, 
we choose $a=b=2,$ $\theta =1,$ $t_0=0,$ and $u_0(x,y) = 0.001(1-\cos(2\pi x))(1-\cos(2\pi y)),$ $-1\le x,y\le 1.$ 
For the simplicity we once again set $\sigma(x,y)\equiv 1,~0\le x,y\le 1,$ and $\varphi(\varepsilon) = \varepsilon^2,~0.01\le \varepsilon\le 1.$ 
We wish to demonstrate that the stochastic effects exhibited in the previous experiment remain to be significant in the two-dimensional case. 
To this end, we first approximate \R{num11}-\R{num33} via a standard nonuniform LOD method \cite{Beau1,Beau11,ADI}. 
\begin{center}
{\epsfig{file=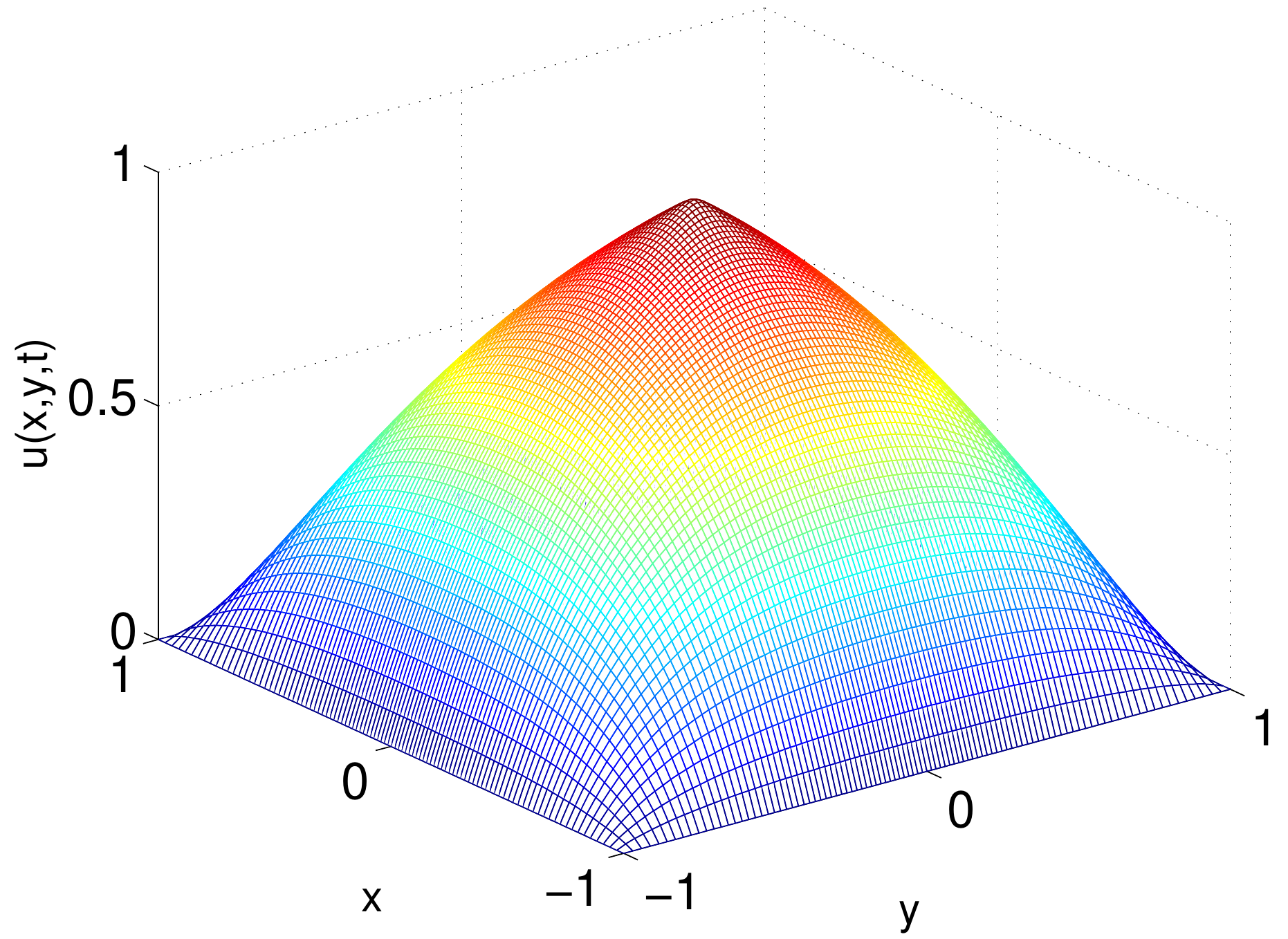,width=2.63in,height=1.66in}}
{\epsfig{file=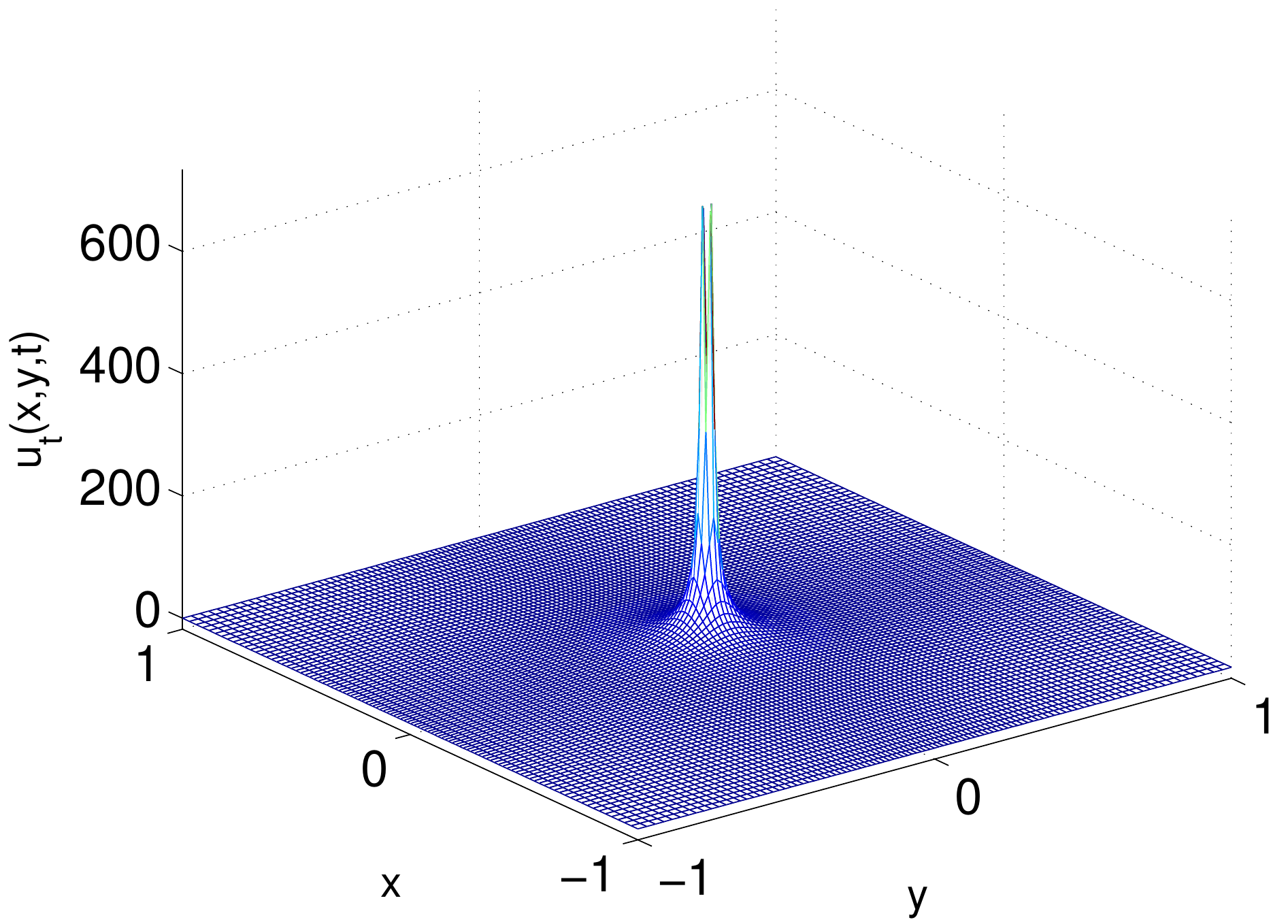,width=2.63in,height=1.66in}}
~~
\parbox[t]{12.8cm}{\scriptsize{\bf Figure 18.} Three-dimensional views of $u$ [LEFT] and $u_t$ [RIGHT] immediately prior 
to quenching. Quenching is observed to occur at $T\approx 2.564180137941836.$ The temporal derivative is observed to 
reach a maximal value $\textstyle\max u_t \approx 331.0481243466621.$ The blow-up of $u_t$ is concentrated around 
the observed spatial quenching point.}
\end{center}
Fig. 18 depicts the final profiles of the solution $u$ and its derivative $u_t,$ immediately prior to quenching. 
Quenching is observed to occur at $T\approx 2.564180137941836.$ We have experienced a slight increase in quenching times as
compared to existing results for $\varphi(\varepsilon)\equiv 1$ \cite{Sheng4,Beau1,Beau11}. 
\begin{center}
{\epsfig{file=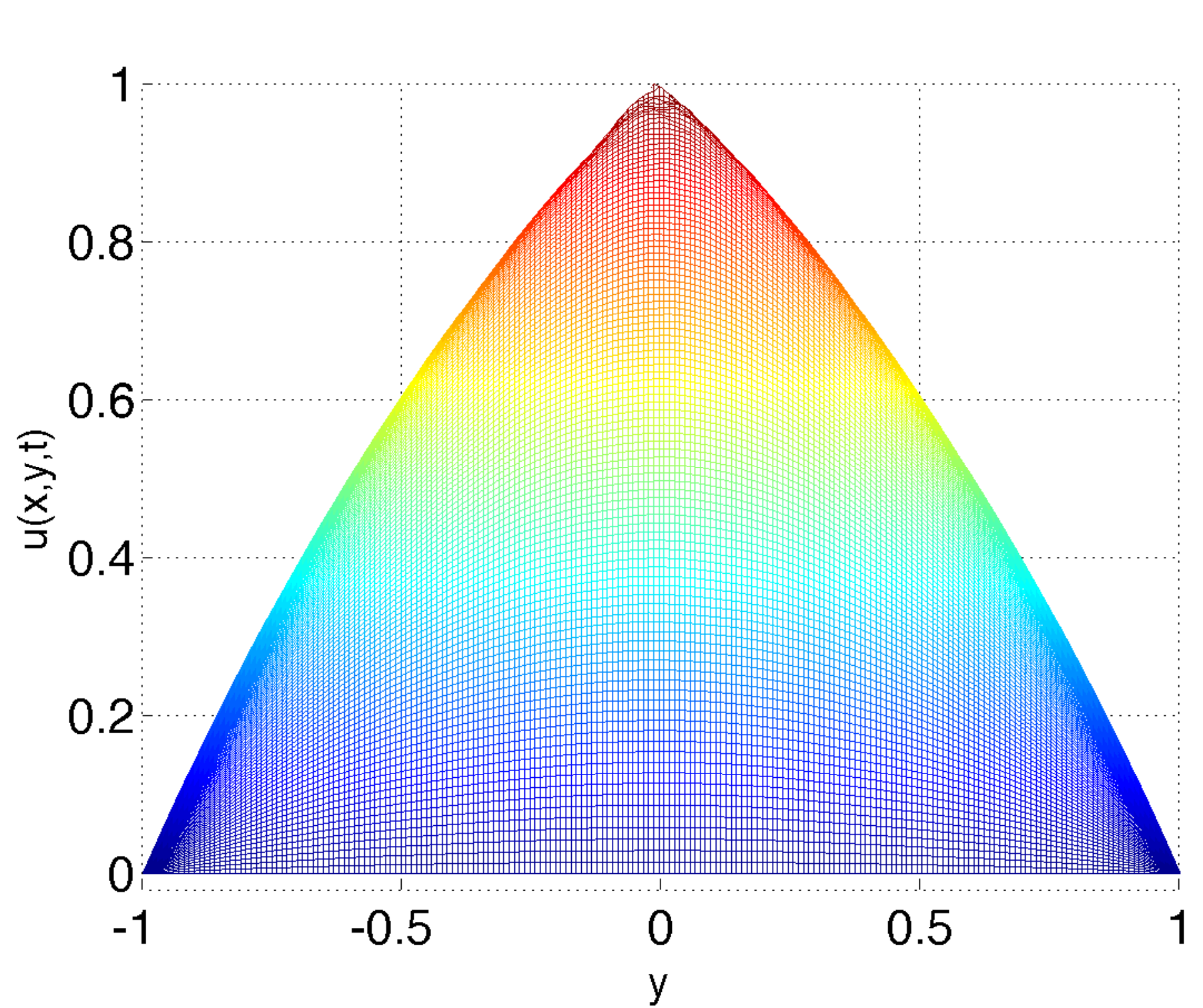,width=2.3in,height=1.28in}}
{\epsfig{file=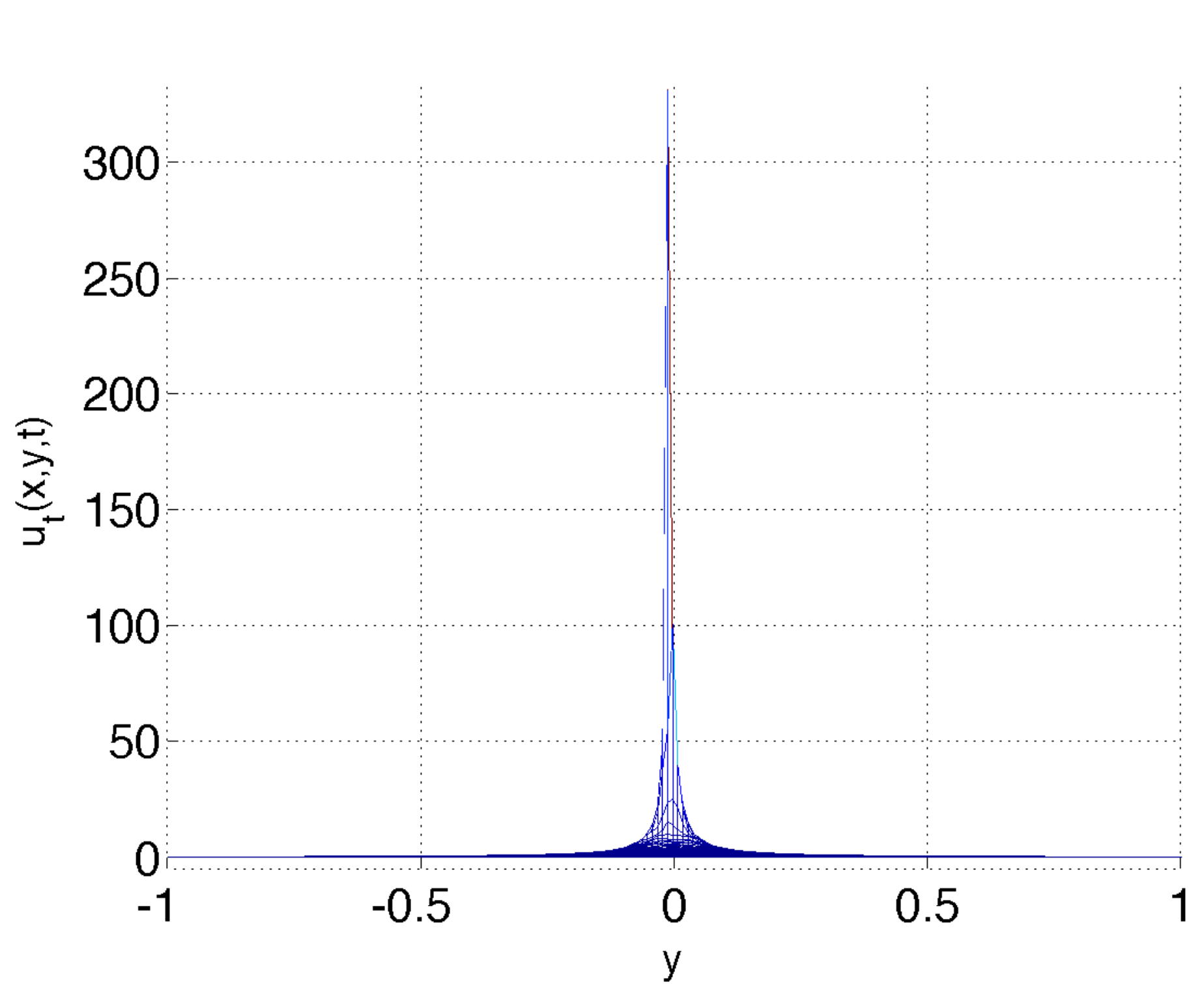,width=2.3in,height=1.28in}}
~~
\parbox[t]{12.8cm}{\scriptsize{\bf Figure 19.} Projections of $u$ [LEFT] and $u_t$ [RIGHT] onto the $xu$-plane and $xu_t$-plane 
immediately prior to quenching, respectively. A slight shift of the location of maximum values is observed. It is evident that the solution 
smoothness is once again affected slightly by the presence of the stochastic influence in reactions.}
\end{center}
\begin{center}
\begin{tabular}{c|c|c}\hline\hline
$k$ & $x_k^*$ & $y_k^*$\\ \hline
1     & 0.034825870646764 & -0.014925373134330\\ \hline
2     & 0.037037037037037 & -0.061728395061728\\ \hline
3     & 0.012345679012346 & -0.012345679012346\\ \hline
4     & 0.086419753086420 & 0.061728395061728\\ \hline
5     & -0.012345679012346 & 0.012345679012346\\ \hline\hline
\end{tabular}

\vspace{2mm}

\parbox[t]{12.8cm}{\scriptsize{\bf Table 1.} Experimented quenching locations $P_k^*=(x_k^*,y_k^*),~k=1,\ldots,5,$ 
due to stochastic influences. Drifting effects can be clearly observed.}
\end{center}

Fig. 19 depicts projections of $\textstyle\max_{-1\le y\le 1} u(x,y,t)$ and $\textstyle\max_{-1\le x\le 1} u_t(x,y,t),$ respectively. 
In the absence of stochastic influences, the theoretical spatial quenching location should be $P=(0,0).$ But we have observed 
$P^* = (0.034825870646764 , -0.014925373134330)$ to be a shifted spatial quenching position under the influence of 
present stochastic source term. Repeated experiments with varying stochastic influence functions suggest 
the same phenomenon. We list results of five randomly selected stochastic functions in Table 1 as an illustration.
We note similarities between the plots in Fig. 19 and those for one-dimensional cases in quenching
location disturbances. This is to be expected, as the results demonstrate that the propagation of the nonsmooth feature 
is not limited to one-dimensional situations only. 
}

\section{Conclusions} \clearallnum

A nontraditional Crank-Nicolson method for solving the nonlinear
stochastic Kawarada differential equation is proposed and studied. 
{\red Conventional uniform or symmetric mesh structures are replaced successfully by fully arbitrary 
grids in the space.} Temporal adaptation is incorporated in order to effectively capture
the strong quenching-combustion singularity and degeneracy built with the nonlinear partial differential equation.
Key properties of the numerical method developed, including the solution positivity, monotonicity,
and stability, are investigated and proven. Stability conditions determined are nonrestrictive. Stochastic impacts 
through the source term are examined and discussed carefully through simulation experiments.

Although linear stability analysis has been effective in the study of numerical solution of
quenching problems while nonlinear source terms are frozen in implicit
schemes \cite{Beau1,Josh2,Sheng3,Josh1,Twi}, an improved semi-linear stability analysis is introduced
and conducted. This modified analysis depends upon the Jacobian of the nonlinear reaction term of the 
Kawarada equation. It is found that the constraints required to guarantee the positivity and monotonicity 
of the underlying nonuniform numerical method are sufficient for ensuring the aforementioned semi-linear 
stability. This sheds further insights as to reasons why a linear stability analysis is often 
adequate in realistic computational applications. 

In our numerical experiments we have studied effects of the size of spatial domain on quenching time, 
which seems to suggest a possible optimal domain size due to the minimum quenching time observed. Further, we have 
explored the effects of the stochastic source term on overall solutions. Computational experiments indicate that different 
nonlinear source terms may have impacts on not only quenching time and location, but also the smoothness of the 
quenching solution profile \cite{Levine,Acker1,Sheng3}. {\red A two-dimensional experiment is also presented
to verify the potential of the semi-adaptive infrastructure introduced in this study, as well as verify the
effects of a stochastic source influence in higher-dimensional cases.}
Our future endeavors include studying stochastic influences which vary with respect to both time and space.
Multi-dimensional Kawarada problems will also be approximated via the latest operator splitting strategies \cite{Sheng4,Josh2}. 
{\red On the other hand, exponential time differencing schemes, such as those explored in \cite{Khaliq,Khaliq2}, 
will be considered in the near future together with proper adaptations. We also plan to extend our investigations of literature
by including balanced fractional derivatives in order to more precisely capture and explore global features 
of the numerical combustion \cite{Pagnini}.}

\section*{Acknowledgements}
\noindent 
The authors would like to thank the anonymous referees for their time and thorough comments. 
Integrating their suggestions has undoubtedly elevated the quality and presentation of this paper.

The second author particularly appreciates Wes Johnson and  Mike Hutcheson of the 
Baylor University Academic and Research Computing Services
for computational validations and technical support.

\end{document}